\documentclass[11pt,
,a4paper]{article} 
\pdfoutput=1 
\usepackage{amsmath}
\usepackage{amssymb}
\usepackage{amsthm}
\usepackage{verbatim}
\usepackage{enumerate}
\usepackage{color}
\usepackage{blkarray}
\usepackage{bm}

\usepackage{tikz}
\usetikzlibrary{cd}
\usepackage{graphicx,amscd,mathrsfs}

\usepackage[top=30truemm,bottom=30truemm,left=30truemm,right=30truemm]{geometry}

\usepackage[affil-it]{authblk}

\theoremstyle{definition}
\newtheorem{definition}{Definition}[section]
\newtheorem{example}[definition]{Example}
\newtheorem{remark}[definition]{Remark}

\theoremstyle{plain}
\newtheorem{theorem}[definition]{Theorem}
\newtheorem{lemma}[definition]{Lemma}
\newtheorem{proposition}[definition]{Proposition}
\newtheorem{corollary}[definition]{Corollary}

\usepackage{mathtools}
\mathtoolsset{showonlyrefs,showmanualtags}
\def\ds{\displaystyle}
\def\p{\partial}
\DeclareMathOperator{\im}{Im}
\DeclareMathOperator{\Ker}{Ker}
\DeclareMathOperator{\Coker}{Coker}

\def\bbP{\mathbb{P}}
\def\bbC{\mathbb{C}}
\def\calC{\mathcal{C}}
\def\calT{\mathcal{T}}
\def\calS{\mathcal{S}}
\def\scrC{\mathscr{C}}
\def\scrP{\mathscr{P}}
\def\calV{\mathcal{V}}
\def\scrB{\mathscr{B}}
\def\calA{\mathcal{A}}
\def\calL{\mathcal{L}}
\def\calK{\mathcal{K}}
\def\calW{\mathcal{W}}

\theoremstyle{remark}

\numberwithin{equation}{section}

\title{Middle Laplace transform and middle convolution for linear Pfaffian systems with irregular singularities}
\author{
	Shunya Adachi\thanks{Supported by JSPS KAKENHI Grant Number 24K22826.}
}
\date{}

\begin{document}

\maketitle

\begin{abstract}
We introduce a transformation of linear Pfaffian systems, which we call the middle Laplace transform, as a formulation of the Laplace transform from the perspective of Katz theory.
While the definition of the middle Laplace transform is purely algebraic, its categorical interpretation is also provided.
We then show the fundamental properties (invertibility, irreducibility) of the middle Laplace transform.
As an application of the middle Laplace transform, we define the middle convolution for linear Pfaffian systems with irregular singularities.
This gives a generalization of Haraoka's middle convolution, which was defined for linear Pfaffian systems with logarithmic singularities.
The fundamental properties (additivity, irreducibility) of the middle convolution follow from the properties of the middle Laplace transform. 
Some examples related to hypergeometric functions with two variables are also given.
\end{abstract}

\tableofcontents

\section{Introduction}

In this paper, we introduce two transformations--the \emph{middle Laplace transform} and the \emph{middle convolution}--for linear Pfaffian systems with irregular singularities. 
Our motivation comes from the Katz theory for linear differential equations. 
In his book \cite{KatzBook}, N. M. Katz introduced an operation called the middle convolution for local systems on a punctured projective line. 
After that, Dettweiler-Reiter \cite{DR2000, DR2007} established an additive analog of Katz's operation as an operation for Fuchsian systems
	\begin{equation}\label{Eq_FS}
	\frac{du}{dx}=\left(\sum_{i=1}^q \frac{A_i}{x-a_i}\right)u,
	\quad A_i \in \mathrm{Mat}(N,\mathbb{C}).
	\end{equation}
This operation is called the (additive) middle convolution, which gives a recursive method for constructing Fuchsian systems. 
In particular, any irreducible rigid (i.e., free from accessory parameters) Fuchsian system can be constructed from the trivial rank one equation $u'=0$ by iterations of the middle convolution (Katz algorithm, cf. \cite[Theorem A.14]{DR2000}).
The additive middle convolution is realized by the composition of the following three steps:
	\begin{enumerate}[\bf Step 1.]
	\item Extend the given system \eqref{Eq_FS} to Okubo normal form
	\begin{equation}\label{Eq_ODE_Okubo}
	(x-T)\frac{dU}{dx}=(A-I)U,
	\end{equation}
	where $T$ is a constant diagonal matrix, $A$ is a constant matrix and $I$ is the identity matrix.
	\item Apply the Riemann-Liouville transform (with a suitable path of integration $\Delta$) 
	\[
	U(x)\mapsto \int_{\Delta}U(t)\,(x-t)^\lambda\,dt.
	\]
	\item Project onto a suitable quotient space.
	\end{enumerate}
Then, the middle convolution can be regarded as a precise formulation of the Riemann-Liouville transform for Fuchsian systems \eqref{Eq_FS}.
We note that, in Dettweiler-Reiter's description, the above procedures are given in a purely algebraic way, namely, as an operation on the residue matrices $(A_1,A_2,\ldots,A_q)$ in \eqref{Eq_FS}.

The theory of middle convolution led to new developments in the study of global analysis of linear ordinary differential equations in the complex domain. 
For example, by using the middle convolution, we can obtain explicit descriptions of the integral representation of solutions \cite{H-H2012}, monodromy representations \cite{DR2007}, connection coefficients and reducibility conditions \cite{Oshima2012} recursively. 
Furthermore, the middle convolution works effectively for the study of the Deligne-Simpson problem, the classification problem of Fuchsian differential equations \cite{Oshima2013}, and isomonodromic deformations \cite{H-F2007}.

In 2012, Haraoka \cite{Haraoka2012} expected that the Katz theory would be useful for the study of holonomic systems, and extended the additive middle convolution to linear Pfaffian systems with logarithmic singularities along hyperplane arrangements 
	\begin{equation}\label{Eq_Pfaff_reg0}
	du=\Omega u, \quad 
	\Omega=\sum_{H \in \mathcal{A}}A_H \, d\log f_H
	\quad (A_{H}\in\mathrm{Mat}(N,\mathbb{C})),
	\end{equation}
where $\mathcal{A}$ is a hyperplane arrangement in $\mathbb{C}^n$ and $f_{H}$ denotes the defining polynomial of $H\in\mathcal{A}$.
Roughly speaking, he showed that Dettweiler-Reiter's middle convolution operation (i.e., the above three steps) for a chosen variable can be prolonged as an operation for linear Pfaffian systems \eqref{Eq_Pfaff_reg0}.
His middle convolution gives a recursive method for constructing linear Pfaffian systems with logarithmic singularities along hyperplane arrangements. 
Inspired by his result, Oshima \cite{Oshima2017, OshimaPre1} and Matsubara-Heo and Oshima \cite{M-Opre} developed the study of regular holonomic systems including KZ-type equations and hypergeometric systems.
We note that the multiplicative version of Haraoka's middle convolution was also given by Haraoka himself \cite{Haraoka2020} and developed by Hiroe-Negami \cite{H-NPre}.

The project discussed in this paper originated from a simple question: \emph{what happens if we apply the Laplace transform instead of the Riemann-Liouville transform in step 2 of (Haraoka's) middle convolution?}
As an answer, we give a formulation of the Laplace transform for Pfaffian systems, which should be called the \emph{middle Laplace transform}\footnote{The name ``middle Laplace transform" is inspired by Sabbah \cite{Sabbah}.}. 
We remark that, when considering the Laplace transform, the irregular singularities naturally appear. 
In fact, by applying the Laplace transform 
	\[
	U(x)~\mapsto~V(x)=\int_{\Delta}U(t)e^{-xt}\,dt
	\]
(with a suitable path of integration $\Delta$) for the Okubo normal form \eqref{Eq_ODE_Okubo}, we obtain 
	\begin{equation}\label{Eq_Birkhoff0}
	x\frac{d V}{d x}=-\left(A+xT\right)V.
	\end{equation}
The equation of the form \eqref{Eq_Birkhoff0} is called the Birkhoff canonical form and has an irregular singular point at $x=\infty$. 
Therefore, when considering the Laplace transform, it is insufficient to treat only Fuchsian systems or linear Pfaffian systems with logarithmic singularities. 
In this paper, we provide a formulation of the middle Laplace transform as a transformation of linear Pfaffian systems with logarithmic singularities along $\mathcal{A}$ and irregular singularities along the hyperplanes $\{x_i=\infty\}$ ($i=1,\ldots,n$)
	\begin{equation}\label{Eq_Pfaff0}
	du=\Omega u, \quad 
	\Omega=\sum_{i=1}^nS_{x_i}(\bm{x})\,dx_i
	+\sum_{H \in \mathcal{A}}A_H \, d\log f_H,
	\end{equation}
where $S_{x_i}(\bm{x})\in\mathbb{C}[x_1,\ldots,x_n]$ are linear polynomials. 
Then, we show that the middle Laplace transform has good properties: if the system \eqref{Eq_Pfaff0} is irreducible and non-exceptional (in the sense of Definitions \ref{def:odeirred} and \ref{Def_irr}), then the middle Laplace transformed system is also irreducible, and the middle Laplace transform admits an inversion formula.
Like Haraoka's middle convolution, the middle Laplace transform provides a recursive and effective method for constructing linear Pfaffian systems with irregular singularities. 
For example, some confluent hypergeometric systems in several variables can be obtained by using the middle Laplace transform.
We expect that the middle Laplace transform will be a useful tool for studying holonomic systems with irregular singularities.

As an application of the middle Laplace transform, we can define the middle convolution for linear Pfaffian systems of the form \eqref{Eq_Pfaff0}. 
This provides a generalization of Haraoka's middle convolution. 
The additivity and the irreducibility of the generalized middle convolution follow from the properties of the middle Laplace transform described above. 
Oshima \cite{Oshima2020} also defined the middle convolution for linear Pfaffian systems with irregular singularities, using the confluence of singularities.
Our definition differs from Oshima's, and the classes of equations to which the middle convolution can be applied are also slightly different (with some overlap). 
It seems that our middle convolution and Oshima's middle convolution are essentially the same when applied to the same equation, and this will be verified.

On the other hand, assuming $n=1$ in \eqref{Eq_Pfaff0}, our middle convolution corresponds to the one for linear ordinary differential equations with an irregular singular point at $x=\infty$.
There are several studies on the middle convolution for linear ordinary differential equations with irregular singular points \cite{Arinkin2010, Kawakami2010, Takemura2011, Yamakawa2011, Yamakawa2016}.
Our result can be regarded as a generalization of their results, although they dealt with more general linear ODEs than the system \eqref{Eq_Pfaff0} with $n=1$.  
In particular, Yamakawa's method \cite{Yamakawa2011, Yamakawa2016} for generalizing the middle convolution is similar to ours. 
Indeed, he first provided a beautiful description of the Laplace transform for linear ODEs by using the Harnad duality. 
He then applied the Harnad duality to generalize the middle convolution. 
Examining whether the middle Laplace transform for the system \eqref{Eq_Pfaff0} with $n=1$ is equivalent to the Harnad duality is left for future work. 
We note that if this holds true, the Katz algorithm for linear Pfaffian systems \eqref{Eq_Pfaff0} using our middle convolution follows immediately from Yamakawa's result (cf. \cite[\S 7]{Yamakawa2011}).

This paper is organized as follows. 
In Sections \ref{sec:odemL}--\ref{sec:odemc}, we focus on the linear Pfaffian system \eqref{Eq_Pfaff0} with $n=1$ (i.e., ordinary differential equation).
Section \ref{sec:odemL} introduces the definitions of the (inverse) middle Laplace transform in one variable case.
The definition is given in terms of linear algebra. 
In that section, we also state their fundamental properties on invertibility and irreducibility (Theorem \ref{thm:odemain}) without proofs.
In Section \ref{sec:odefunc}, we give a categorical interpretation of the (inverse) middle Laplace transform. 
Namely, we regard linear ordinary differential equations as meromorphic connections, and introduce some categories consisting of meromorphic connections.
We then formulate the middle Laplace transform and related transformations as functors between the categories. 
Section \ref{sec:odeinv} is devoted to the proof of the properties stated in Section \ref{sec:odemL}. 
The categorical interpretation in Section \ref{sec:odefunc} is effectively used there.
In Section \ref{sec:odemc}, we give a generalization of middle convolution for linear differential equations with an irregular singular point via the middle Laplace transform. 
In Sections \ref{sec:PfaffmL}--\ref{sec:extwo}, we treat the several variable case $n\ge 2$. 
In Section \ref{sec:PfaffmL}, we extend the middle Laplace transform for the case of several variables and give the fundamental properties (Theorem \ref{Thm_main}).
We also characterize the class of linear Pfaffian systems for which the middle Laplace transform behaves well.
Section \ref{Sec_Categorical} provides the categorical interpretation of the (inverse) middle Laplace transform in several variables, which is a generalization of the description in Section \ref{sec:odefunc}.
Section \ref{Sec_Inv} is devoted to the proof of the properties stated in Section \ref{sec:PfaffmL}. 
We note that Subsection \ref{Subsec_varphi_mor} is devoted solely to the proof of a technical proposition and is somewhat lengthy; readers who are only interested in our results may skip it. 
In Section \ref{Sec_GMC}, as an application of the above results, we generalize the middle convolution for linear Pfaffian systems with irregular singularities. 
The fundamental properties (additivity and irreducibility) of the middle convolution are also given. 
At the end of this paper, in Section \ref{sec:extwo}, we give some examples related to (confluent) hypergeometric functions with two variables.

We end the introduction by mentioning some future problems. 
First, investigating the transformation of local data (spectral type, local exponents, characteristic polynomials at irregular singular points) induced by the middle Laplace transform or middle convolution is an important problem. 
Next, investigating the transformation of global data (connection coefficients, monodromy and Stokes multipliers) induced by the middle Laplace transform or the middle convolution is also important. 
For this problem, the pioneering work of Balser-Jurkat-Lutz \cite{BJL1981} seems to be an important reference. 
They showed that the Stokes multipliers of the Birkhoff canonical form \eqref{Eq_Birkhoff0} can be calculated by using the connection coefficients of the Okubo normal form \eqref{Eq_ODE_Okubo} through the Laplace transform. 
We expect that their method can be generalized to holonomic systems via the middle Laplace transform.
Shimomura \cite{Shimomura1993, Shimomura1997, Shimomura1998, Shimomura1999} and Majima \cite{Majima2004} derived the Stokes multipliers of some confluent hypergeometric systems in several variables. 
Their results seem to provide us with good insight for generalizing the method of Balser-Jurkat-Lutz.
Furthermore, Haraoka \cite{Haraoka2024} has recently calculated the Stokes multiplier of a certain confluent hypergeometric function in two variables by using Majima's strong asymptotic expansion \cite{Majima1984}.  
As we will see in Section \ref{sec:extwo}, the linear Pfaffian system associated with Haraoka's confluent hypergeometric function can be obtained by iterating the middle Laplace transform for a rank one system. 
Therefore, we expect that Haraoka's result may be understood from our perspective.

\section{Middle Laplace transform in one variable}\label{sec:odemL}
First, we introduce the middle Laplace transform for linear ordinary differential equations
\begin{equation}\label{eq:ode}
\frac{du}{dx}=\left(S+\sum_{i=1}^q \frac{A_{i}}{x-a_i}\right)u,
\end{equation}
where $S, A_i \in \mathrm{Mat}(N,\mathbb{C})$ and $x, a_1,\ldots,a_q\in\mathbb{C}$. 
This equation has regular singular points at $x=a_1,\ldots,a_q$ and an irregular singular point of Poincar\'{e} rank one at $x=\infty$. 
Note that, the equation \eqref{eq:ode} can be expressed as
\begin{align}\label{eq:ode1form}
du=\Omega u, \quad \Omega=S\, dx+\sum_{i=1}^q A_i \,d\log (x-a_i)
\end{align}
and hence \eqref{eq:ode} is a linear Pfaffian system \eqref{Eq_Pfaff0} in one variable.

In the following, we assume that the matrix $S$ is diagonalizable.
For simplicity in the following discussion, we further assume that $S$ is already diagonalized as 
\begin{equation}\label{eq:Sdiag}
S=\begin{pmatrix}
\alpha_1 I_{N_1} & & & \\
& \alpha_2 I_{N_2}& & \\
& & \ddots & \\
& & & \alpha_{\hat{q}}I_{N_{\hat{q}}}
\end{pmatrix}
\end{equation}
by a suitable gauge transformation. 
Here $\alpha_1,\dots,\alpha_{\hat{q}}\in\mathbb{C}$ and $N=N_1+N_2+\cdots+N_{\hat{q}}$ is a partition of $N$.

\subsection{Definition}
\label{subsec:onedef}
The middle Laplace transform consists of the following three steps:
	\begin{enumerate}[\bf Step~1.]
	\item Extend the equation \eqref{eq:ode} to the Birkhoff-Okubo normal form.
	\item Apply the Laplace transform.
	\item Project onto a suitable quotient space.
	\end{enumerate}
In the following, we will explain the above steps in detail and in order.

\medskip
\noindent
\textbf{Step~1.} 
Let $u(x)$ be a solution of the equation \eqref{eq:ode} and set
\begin{align}
U(x)={}^t \left(\frac{u(x)}{x-a_1},\frac{u(x)}{x-a_2},\dots,\frac{u(x)}{x-a_{q}}\right).
\end{align}
Then, we can verify that the vector $U(x)$ satisfies the system which we call the \emph{Birkhoff-Okubo normal form}:
\begin{align}\label{eq:odeBO}
(x-T)\frac{dU}{dx}=\left(A-I+S^{\oplus q}(x-T)\right)U,
\end{align}
where
\begin{equation}\label{eq:odeTAS}
\begin{aligned}
&T=\begin{pmatrix}
a_1 I_N & & & \\
& a_2 I_N & & \\
& & \ddots & \\ 
& & & a_q I_N
\end{pmatrix}, 
\quad 
A=\begin{pmatrix}
A_1 & A_2 & \dots & A_q \\
A_1 & A_2 & \dots & A_q \\
\vdots & \vdots &  & \vdots \\
A_1 & A_2 & \dots & A_q
\end{pmatrix},
\\
&S^{\oplus q}=\bigoplus_{i=1}^{q} S=\begin{pmatrix}
S & & & \\
& S & & \\
& & \ddots & \\
& & & S
\end{pmatrix}.
\end{aligned}
\end{equation}
We note that the Birkhoff-Okubo normal form \eqref{eq:odeBO} is an ordinary differential equation with polynomial coefficients, and thus it is well suited to integral transformations such as the Riemann-Liouville and Laplace transforms.

\medskip
\noindent
\textbf{Step~2.} For the Birkhoff-Okubo normal form \eqref{eq:odeBO}, we consider the Laplace transform 
	\begin{equation}\label{eq:op_FL}
		L:\left\{
	\begin{array}{ccc}
	\dfrac{d }{d x} &\,\mapsto\, & x, \\
	x &\,\mapsto\,& -\dfrac{d }{d x}
	\end{array}\right.
	\end{equation}
which corresponds to the integral transform 
	\begin{align}\label{eq:odeLsol}
	U(x)~\mapsto~V(x)=\int_{\Delta} U(t)e^{-xt}\,dt
	\end{align}
with a suitable path of integration $\Delta$.
Formally applying \eqref{eq:op_FL} to the equation \eqref{eq:odeBO}, we have
\begin{align}
\left(-\frac{d}{dx}-T\right)(xV)=\left(A-I+S^{\oplus q}\left(-\frac{d}{dx}-T\right)\right)V
\end{align}
and hence obtain
\begin{equation}\label{eq:odeL}
\left(x-S^{\oplus q}\right)\frac{dV}{dx}=-\left(A+(x-S^{\oplus q})T\right)V.
\end{equation}
By multiplying $(x-S^{\oplus q})^{-1}$ from the left, we have
\begin{align}\label{eq:odeL2}
\frac{dV}{dx}=\left(-T+\sum_{j=1}^{\hat{q}}\frac{B_{j}}{x-\alpha_j}\right)V,
\end{align}
where 
\begin{equation}\label{eq:odeB}
B_j=-\begin{pmatrix}
	E_{N_j}A_{1} & E_{N_j}A_{2}& \cdots &E_{N_j}A_{q} \\
	E_{N_j}A_{1} & E_{N_j}A_{2}& \cdots &E_{N_j}A_{q} \\
	\vdots & \vdots &  &\vdots\\
	E_{N_j}A_{1} & E_{N_j}A_{2}& \cdots &E_{N_j}A_{q} \\
	\end{pmatrix}
\end{equation} 
and
\begin{equation}\label{eq:Eni}
	E_{N_j}
	=\mathrm{diag}[O_{N_1},\ldots,I_{N_j},	\ldots,O_{N_{\hat{q}}}]
	=
	\begin{pmatrix}
	 O_{N_1}& & & & \\
	  & \ddots & &  \\
	 & &  I_{N_j} & & \\
	&  & &\ddots & \\
	 && & & O_{N_{\hat{q}}}
	\end{pmatrix}
\end{equation}
for $1\le j\le \hat{q}$.

\begin{definition}\label{def:odeLaplace}
We call the operation which sends the equation \eqref{eq:ode} to the equation \eqref{eq:odeL2} the \emph{Laplace transform} and denote it by $\mathcal{L}$.
This operation can also be regarded as sending the $1$-form $\Omega$ in \eqref{eq:ode1form} to the $1$-form
\begin{equation}\label{eq:ode1formL}
\mathcal{L}(\Omega):=-T\,dx+\sum_{j=1}^{\hat{q}}B_{j}\,d\log(x-\alpha_{j}).
\end{equation}
\end{definition}
\begin{remark}\label{rem:odeLaplace}
Analytically, the transform \eqref{eq:op_FL} can be realized by the integral transform 
\eqref{eq:odeLsol} with a path of integration $\Delta$ along which the integral converges. 
Under suitable assumptions on the coefficient matrices of the initial equation \eqref{eq:ode}, which ensure an appropriate asymptotic behavior of $U(t)$, such a path can indeed be chosen; see, for example, \cite{BJL1981} or \cite[Chapter 5]{HaraokaBook} for a detailed analytic treatment. 
In this paper we only use the formal correspondence \eqref{eq:op_FL}, and we do not discuss these analytic aspects further.

\end{remark}

\medskip
\noindent
\textbf{Step 3.} The Laplace transformed equation \eqref{eq:odeL} is not irreducible in general, even if the original equation \eqref{eq:ode} is irreducible (a precise definition of irreducibility will be given later). 
Therefore, we consider projecting \eqref{eq:odeL} onto a suitable quotient space in order to extract the irreducible part. 

Let $\mathcal{K}$ be a subspace of $(\mathbb{C}^{N})^q$ defined by
\begin{equation}\label{eq:odeK}
	\mathcal{K}:=\bigoplus_{i=1}^q\Ker A_{i}
	=
	\left\{
	\begin{pmatrix}
	v_1 \\
	\vdots\\
	v_q
	\end{pmatrix}
	\in(\mathbb{C}^N)^q
	\,\middle\vert\,
	v_i \in \Ker A_{i}
	~
	(1\le i\le q)
	\right\}.
\end{equation}
Then, we have the following.
\begin{lemma}\label{lem:odeK}
The subspace $\mathcal{K}$ is a common invariant subspace of the tuple of matrices $(T,B_1,\ldots,B_{\hat{q}})$.
\end{lemma}
\begin{proof}
Take $v={}^t(v_1,\ldots,v_q)\in \mathcal{K}$. 
Then we have
\begin{align}
Tv=\begin{pmatrix}
a_1 I_N & & & \\
& a_2 I_N & & \\
& & \ddots & \\ 
& & & a_q I_N
\end{pmatrix}
\begin{pmatrix}
	v_1 \\
	v_2 \\
	\vdots\\
	v_q
	\end{pmatrix}
=\begin{pmatrix}
	a_1v_1 \\
	a_2v_2 \\
	\vdots\\
	a_qv_q
	\end{pmatrix}
\end{align}
Since $v_i \in \Ker A_i$, we have $a_i v_i \in \Ker A_i$. 
This means that $Tv \in\mathcal{K}$. 
For $1\le j\le \hat{q}$, we have
\begin{align}
B_j v=-\begin{pmatrix}
	E_{N_j}A_{1} & E_{N_j}A_{2}& \cdots &E_{N_j}A_{q} \\
	E_{N_j}A_{1} & E_{N_j}A_{2}& \cdots & E_{N_j}A_{q} \\
	\vdots & \vdots &  &\vdots\\
	E_{N_j}A_{1} & E_{N_j}A_{2}& \cdots & E_{N_j}A_{q} \\
	\end{pmatrix}
	\begin{pmatrix}
	v_1 \\
	v_2 \\
	\vdots\\
	v_q
	\end{pmatrix}
	=0
\end{align}
since $v_i \in \Ker A_i$.
This means $B_j v \in \mathcal{K}$. 
\end{proof}
Lemma \ref{lem:odeK} implies that the tuple of matrices $(T,B_1,\ldots,B_{\hat{q}})$ induces a linear action on the quotient space $(\mathbb{C}^N)^q/\mathcal{K}$. 
Let $(\bar{T},\bar{B}_1,\ldots,\bar{B}_{\hat{q}})$ denote a tuple of matrices representing the induced action on $(\mathbb{C}^N)^q/\mathcal{K}$. 
Then, we have the equation
\begin{align}\label{eq:odeML}
\frac{dv}{dx}=\left(-\bar{T}+\sum_{j=1}^{\hat{q}}\frac{\bar{B}_{j}}{x-\alpha_{j}}\right)v,
\end{align}
or equivalently, in Pfaffian form,
\begin{equation}\label{eq:ode1formML}
dv=\mathcal{ML}(\Omega)v, 
\quad 
\mathcal{ML}(\Omega):=-\bar{T}\,dx+\sum_{j=1}^{\hat{q}}\bar{B}_{j}\,d\log(x-\alpha_{j}).
\end{equation}
\begin{definition}\label{def:odeML}
We call the operation which sends the equation \eqref{eq:ode} to the equation \eqref{eq:odeML} the \emph{middle Laplace transform} and denote it by $\mathcal{ML}$.
\end{definition}
We note that the middle Laplace transform may change the rank of differential equations in general.
Under the assumptions of Theorem \ref{thm:odemain} below, the middle Laplace transformed equation \eqref{eq:odeML} is also irreducible.

\begin{remark}
\begin{enumerate}[(i)]
\item The tuple $(\bar{T},\bar{B}_1,\ldots,\bar{B}_{\hat{q}})$ is obtained as follows. 
Let $(u_1,\ldots,u_m)$ be a basis of the subspace $\mathcal{K}$. 
By adding vectors $(v_{m+1},\ldots,v_{qN})$, we construct a basis of $\mathbb{C}^{qN}$.
We set $P=(u_1,\ldots,u_m,v_{m+1},\ldots,v_{qN})\in\mathrm{GL}(qN,\mathbb{C})$.
Then we have 
\begin{align}
TP=P\left(\begin{array}{c|c} * & * \\ \hline \\[-10pt] O\, & \,\bar{T}\end{array}\right),
\quad 
B_jP=P\left(\begin{array}{c|c} * & * \\ \hline \\[-10pt] O\, & \,\bar{B}_j\end{array}\right) \quad (1\le j\le \hat{q})
\end{align} 
where the matrices on the right-hand side are divided into blocks of size $(m,qN-m)\times(m,qN-m)$.
This corresponds to the decomposition $\mathbb{C}^{qN}=\mathcal{K}\oplus(\mathbb{C}^{qN}/\mathcal{K})$.

\item Let $P$ be as in the above. 
For a solution $V(x)$ for \eqref{eq:odeL2}, we define the $(qN-m)$-vector $\bar{v}(x)$ by
\begin{align}
P^{-1}V=\left(\begin{array}{c}
*\\
\hline 
\bar{v}
\end{array}\right).
\end{align}
Then, we see that $\bar{v}(x)$ gives a solution to the equation \eqref{eq:odeML}. 
Here, setting $Q=(O_{qN-m},I_{qN-m})P^{-1}$, we obtain
\begin{equation}
\bar{v}=QV.
\end{equation}
From this and Remark \ref{rem:odeLaplace}, it follows that the middle Laplace transformed equation \eqref{eq:odeML} admits the following Laplace integral representation of solutions:
\begin{equation}
\bar{v}(x)=QV(x)=Q\int_{\Delta}U(t)e^{-xt}\,dt.
\end{equation}
\end{enumerate}
\end{remark}

\subsection{Inverse (middle) Laplace transform}
\label{subsec:invLone}
We introduce the \emph{inverse middle Laplace transform} of \eqref{eq:ode} by the same three-step procedure as in the middle Laplace transform, except that \textbf{Step~2} is replaced by the inverse Laplace transform
	\begin{equation}\label{eq:op_invFL}
		L^{-1}:\left\{
	\begin{array}{ccc}
	\dfrac{d }{d x} &\,\mapsto\, & -x, \\
	x &\,\mapsto\,& \dfrac{d }{d x}
	\end{array}\right.
	\end{equation}
which corresponds to the integral transform 
	\begin{align}
	U(x)~\mapsto~W(x)=\int_{\Delta} U(t)e^{xt}\,dt
	\end{align}
with a suitable path of integration $\Delta$. 
That is, for the Birkhoff-Okubo extension \eqref{eq:odeBO} of the original equation \eqref{eq:ode}, we formally apply the transform \eqref{eq:op_invFL}.
Then, we have
\begin{equation}
\left(\frac{d}{dx}-T\right)(-xW)=\left(A-I+S^{\oplus q}\left(\frac{d}{dx}-T\right)\right)W
\end{equation}
and hence obtain
\begin{align}
\left(x+S^{\oplus q}\right)\frac{dW}{dx}=\left(-A+(x+S^{\oplus q})T\right)W.
\end{align}
By multiplying $\left(x+S^{\oplus q}\right)^{-1}$ from the left, we have
\begin{align}\label{eq:odeinvL}
\frac{dW}{dx}=\left(T+\sum_{j=1}^{\hat{q}}\frac{B_{j}}{x+\alpha_j}\right)W,
\end{align}
where $B_j$ $(1\le j\le \hat{q})$ are given by \eqref{eq:odeB}.
\begin{definition}\label{def:odeinvLaplace}
We call the operation which sends the equation \eqref{eq:ode} to the equation \eqref{eq:odeinvL} the \emph{inverse Laplace transform} and denote it by $\mathcal{L}^{-1}$.
This operation can also be regarded as sending the $1$-form $\Omega$ in \eqref{eq:ode1form} to the $1$-form
\begin{equation}\label{eq:ode1forminvL}
\mathcal{L}^{-1}(\Omega):=T\,dx+\sum_{j=1}^{\hat{q}}B_{j}\,d\log(x+\alpha_{j}).
\end{equation}
\end{definition}
Then, thanks to Lemma \ref{lem:odeK}, we can project the equation \eqref{eq:odeinvL} onto the quotient space $(\mathbb{C}^N)^q/\mathcal{K}$ and obtain the equation 
\begin{align}\label{eq:odeinvML}
\frac{dw}{dx}=\left(\bar{T}+\sum_{j=1}^{\hat{q}}\frac{\bar{B}_{j}}{x+\alpha_{j}}\right)w,
\end{align}
or equivalently, in Pfaffian form,
\begin{equation}\label{eq:ode1forminvML}
dw=\mathcal{ML}^{-1}(\Omega)w,\quad 
\mathcal{ML}^{-1}(\Omega):=\bar{T}\,dx+\sum_{j=1}^{\hat{q}}\bar{B}_{j}\,d\log(x+\alpha_{j}),
\end{equation}
as in \textbf{Step~3} of the middle Laplace transform.
\begin{definition}\label{def:odeinvML}
We call the operation which sends the equation \eqref{eq:ode} to the equation \eqref{eq:odeinvML} the \emph{inverse middle Laplace transform} and denote it by $\mathcal{ML}^{-1}$.
\end{definition}


While the Laplace transform \eqref{eq:op_FL} and the inverse Laplace transform \eqref{eq:op_invFL} defined for Birkhoff-Okubo normal forms are (formally) inverse to each other, it is not clear whether the middle Laplace transform $\mathcal{ML}$ and the inverse middle Laplace transform $\mathcal{ML}^{-1}$ defined for general ordinary differential equations \eqref{eq:ode}, are truly inverse, since \textbf{Step~1} (extension) and \textbf{Step~3} (projection) are involved. 
In the next subsection, under some suitable assumption for the initial equation \eqref{eq:ode}, we present the inversion formula for the middle Laplace transform and the inverse middle Laplace transform.


\begin{remark}\label{rem:nondiag}
Although we have assumed at the beginning of this section that the matrix $S$ in the initial equation \eqref{eq:ode} is already diagonalized as in \eqref{eq:Sdiag}, the procedures of the (inverse) middle Laplace transform --namely, extension, (inverse) Laplace transform and projection-- remain valid as long as 
$S$ is diagonalizable. 
In other words, even if $S$ is not diagonalized, the same arguments apply provided that 
$S$ is similar to a diagonal matrix of the form
\begin{equation}
S\sim\begin{pmatrix}
\alpha_1 I_{N_1} & & & \\
& \alpha_2 I_{N_2}& & \\
& & \ddots & \\
& & & \alpha_{\hat{q}}I_{N_{\hat{q}}}
\end{pmatrix}
\end{equation}
for some $\alpha_1,\alpha_2,\ldots,\alpha_{\hat{q}}\in\bbC$. 
Therefore, we can define the (inverse) Laplace transform and the (inverse) middle Laplace transform for equations \eqref{eq:ode} with general diagonalizable $S$.  
We note that, in the general case, the description of $B_j$ in \eqref{eq:odeL2} and \eqref{eq:odeinvL} differs from \eqref{eq:odeB}.
\end{remark}

\subsection{Fundamental properties of the middle Laplace transform}
	We summarize some fundamental properties of the (inverse) middle Laplace transform without proofs.
	The proofs will be given in Section \ref{sec:odeinv}. 
	\begin{definition}\label{def:odeirred}
	Let $N\ge1$. 
	For a linear ordinary differential equation \eqref{eq:ode}, or equivalently its corresponding $1$-form $\Omega$ in \eqref{eq:ode1form}, we define the following notions.
	\begin{itemize}
	\item We say that it is \emph{irreducible} if there is no $\langle S, A_1,\ldots,A_q\rangle$-invariant subspace except $\mathbb{C}^N$ and $\{0\}$.
	\item We say that it is \emph{exceptional} if $N=1$ and $A_i=0$ for all $i=1,\ldots,q$. 
	Otherwise, we say it is \emph{non-exceptional}. 
	\end{itemize}
	\end{definition}
	We note that the exceptional case is precisely the scalar equation of the form $u'=\alpha u$ ($\alpha\in\bbC$). 
	As explained in the previous section, the (inverse) middle Laplace transform behaves well for irreducible and non-exceptional equations.
	\begin{theorem}
	\label{thm:odemain}
	Suppose that the ordinary differential equation \eqref{eq:ode} is irreducible and non-exceptional.
	Then, the following hold.
	\begin{enumerate}[(i)]
	\item $\mathcal{ML}^{-1}\circ\mathcal{ML}=\mathcal{ML}\circ\mathcal{ML}^{-1}=\mathrm{id.}$
	\item The middle Laplace transformed system \eqref{eq:odeML} and the inverse middle Laplace transformed system \eqref{eq:odeinvML} are also irreducible and non-exceptional.
	\end{enumerate}
	Here, the above symbol $``="$ is understood up to constant gauge equivalence; that is, the equations obtained by $\mathcal{ML}^{-1} \circ \mathcal{ML}$ and $\mathcal{ML} \circ \mathcal{ML}^{-1}$ coincide with the original one modulo a gauge transformation by some constant matrix $P \in \mathrm{GL}(N, \mathbb{C})$.
	\end{theorem}
	%

\section{Categorical interpretation (one variable case)}
\label{sec:odefunc}
In order to prove Theorem \ref{thm:odemain}, we begin by regarding the ordinary differential equation \eqref{eq:ode} as a meromorphic connection on $\bbP^1=\bbP^1(\bbC)$.
We then formulate the (inverse) middle Laplace transform as a functor between suitable categories of meromorphic connections.

\subsection{Category of meromorphic connections}
Let $\mathcal{O}_{\bbP^1}$ be the sheaf of holomorphic functions on $\mathbb{P}^1$, and let $\mathcal{M}_{\bbP^1}^1$ be the sheaf of meromorphic $1$-forms on $\bbP^1$.
We set $\scrC_{\bbP^1}$ to be the category of meromorphic connections on $\mathbb{P}^1$. 
That is, 
		\begin{equation}\label{eq:odemerom}
	\mathrm{Ob}(\mathscr{C}_{\bbP^1}) = \left\{ (\mathcal{O}_{\mathbb{P}^1}\otimes \mathcal{V},\nabla)~ \middle|~
	 \begin{array}{l}
\text{$\mathcal{V}$ is a finite dimensional $\mathbb{C}$-vector space and} \\
\text{$\nabla:\mathcal{O}_{\bbP^1}\otimes \mathcal{V}\to\mathcal{M}_{\bbP^1}^1 \otimes_{\mathcal{O}_{\bbP^1}} (\mathcal{O}_{\bbP^1}\otimes \mathcal{V})$}
\\
\text{\,is a $\mathbb{C}$-linear map satisfying } \\
\nabla(hv)=dh\otimes_{\mathcal{O}_{\bbP^1}}  v+ h \nabla(v) ~ (h\in\mathcal{O}_{\bbP^1},\, v \in \mathcal{O}_{\bbP^1}\otimes \mathcal{V})
 	\end{array}
 \right\},
	\end{equation}
where we write $\otimes_{\mathbb{C}}=\otimes$ for short.
A $\bbC$-linear map $f:\mathcal{V}\to\mathcal{V}'$ is called a morphism of (meromorphic) connections from $(\mathcal{O}_{\bbP^1}\otimes\mathcal{V},\nabla)$ to $(\mathcal{O}_{\bbP^1}\otimes\mathcal{V}',\nabla')$ if the diagram 
		\[
		\begin{tikzcd}
		\mathcal{O}_{\bbP^1}\otimes \mathcal{V} \arrow[r,"\nabla",""]
		\arrow[d,"\varphi" ']
		&\mathcal{M}_{\bbP^1}^1\otimes_{\mathcal{O}_{\bbP^1}} (\mathcal{O}_{\bbP^1}\otimes \mathcal{V})\arrow[d,""," \mathrm{id.}\otimes_{\mathcal{O}_{\bbP^1}}\varphi"]\\
		\mathcal{O}_{\bbP^1}\otimes \mathcal{V}'  \arrow[r,"\nabla'",""]
		&\mathcal{M}_{\bbP^1}^1\otimes_{\mathcal{O}_{\bbP^1}} (\mathcal{O}_{\bbP^1}\otimes \mathcal{V}')
	\end{tikzcd}
	\quad 
	\]
is commutative.
Here $\varphi:\mathcal{O}_{{\bbP^1}}\otimes\mathcal{V}\to\mathcal{O}_{{\bbP^1}}\otimes\mathcal{V}'$ is a morphism of $\mathcal{O}_{{\bbP^1}}$-modules obtained through the isomorphism 
	\begin{align*}
	f\in \mathrm{Hom}_{\mathbb{C}}(\mathcal{V},\mathcal{V}') 
	&\cong \mathrm{Hom}_{\mathbb{C}}(\mathbb{C}^N,\mathbb{C}^{N'}) \quad (\mathcal{V}\cong\mathbb{C}^N,~\mathcal{V}'\cong\mathbb{C}^{N'}) \\
	&\cong \mathrm{Hom}_{\mathcal{O}_{{\bbP^1}}({\bbP^1})}(\mathcal{O}_{{\bbP^1}}({\bbP^1})^{N},\mathcal{O}_{{\bbP^1}}({\bbP^1})^{N'}) \\
	&\cong \mathrm{Hom}_{\mathcal{O}_{{\bbP^1}}}(\mathcal{O}_{{\bbP^1}}^N,\mathcal{O}_{{\bbP^1}}^{N'}) \\
	&\cong\mathrm{Hom}_{\mathcal{O}_{{\bbP^1}}}(\mathcal{O}_{{\bbP^1}}\otimes\mathcal{V},\mathcal{O}_{{\bbP^1}}\otimes \mathcal{V}')\ni\varphi.
	\end{align*}
In the above, we used the fact $\mathcal{O}_{\bbP^1}(\bbP^1)=\mathbb{C}$ which follows from the compactness of $\bbP^1$.

A linear map $\nabla$ of $(\mathcal{O}_{\bbP^1}\otimes\mathcal{V},\nabla)$ can be expressed as $\nabla=d-\Omega$, where $\Omega$ is an $\mathrm{End}_{\mathbb{C}}(\mathcal{V})$-valued meromorphic $1$-form. 
Then, we identify $\mathscr{C}_{{\bbP^1}}$ with the category of pairs $(\mathcal{V},\Omega)$ consisting of a finite dimensional $\mathbb{C}$-vector space $\mathcal{V}$ and an $\mathrm{End}_{\mathbb{C}}(\mathcal{V})$-valued meromorphic $1$-form $\Omega$. 
The morphisms $(\mathcal{V},\Omega)\to(\mathcal{V}',\Omega')$ in $\mathscr{C}_{{\bbP^1}}$ are linear maps $f:\mathcal{V}\to\mathcal{V}'$ satisfying $\Omega'f=f\Omega$.
\footnote{This treatment is inspired by Yamakawa \cite[\S 2.2]{Yamakawa2016}.} 

We shall explain the correspondence of meromorphic connections and differential equations.
For a connection $(\mathcal{V}, \Omega)\in\mathscr{C}_{{\bbP^1}}$, we fix an isomorphism (a basis) $\mathcal{V}\cong \mathbb{C}^N$. 
Then the $1$-form $\Omega$ can be expressed as $\Omega=A(x)\,dx$ by using some $A(x)\in\mathrm{Mat}(N,\mathbb{C}(x))$. 
Therefore, the connection $(\mathcal{V},\Omega)$ can be identified with the $\mathrm{GL}(N,\mathbb{C})$-conjugacy class of the ordinary differential equation
\begin{equation}\label{eq:odeA}
\frac{du}{dx}=A(x)u.
\end{equation}
We call the matrix $A(x)$ the coefficient matrix of $\Omega$.
For another connection $(\mathcal{V}',\Omega')\in\mathscr{C}_{\bbP^1}$, we fix an isomorphism $\mathcal{V}'\cong\mathbb{C}^{N'}$ and let $B(x)$ be the coefficient matrices of $\Omega'$. 
Then, the linear map $f:\mathcal{V}\to\mathcal{V}'$ is to be a morphism of connections $f:(\mathcal{V},\Omega)\to(\mathcal{V}',\Omega')$ if and only if the matrix representation $F\in\mathrm{Mat}(N'\times N,\mathbb{C})$ of $f$ with respect to the above two basis satisfies
	\[
	FA(x)=B(x)F.
	\]
In particular, if the morphism $f:(\mathcal{V},\Omega)\to(\mathcal{V}',\Omega')$ is an isomorphism of vector spaces, then the corresponding matrix representation $F$ is invertible. 
This means that the equation \eqref{eq:odeA} can be transformed into 
\begin{equation}
\frac{dv}{dx}=B(x)v
\end{equation} by the gauge transformation $v=Fu$.

\subsection{Subcategories}
We now introduce some subcategories of $\scrC_{\bbP^1}$ consisting of connections corresponding to differential equations of the form \eqref{eq:ode}.
Let 
\begin{align}\label{eq:odeTS}
\mathcal{T}:=\{a_1,\ldots,a_q\},
\quad 
\mathcal{S}:=\{\alpha_1,\ldots,\alpha_{\hat{q}}\}
\end{align}
be finite subsets of $\mathbb{C}$.

\subsubsection{Category of linear ordinary differential equations}
We define the full subcategory $\scrP(\calT,\calS)$ of $\scrC_{\bbP^1}$ by
\begin{align}\label{eq:odecatP}
	\mathrm{Ob}(\scrP(\calT,\calS)):=\left\{(\mathcal{V},\Omega)\in\mathrm{Ob}(\scrC_{\bbP^1}) \,\middle \vert
	\text{ $\Omega$ satisfies the condition (O)}
		\right\}.
\end{align}
Here, condition (O) is the following: the $1$-form $\Omega$ is expressed as
\begin{equation}\label{eq:odeconnOmegaEnd}
\Omega=S\, dx+\sum_{i=1}^q A_i \, d\log (x-a_i), \quad S, A_i \in \mathrm{End}(\calV),
\end{equation}
where $S$ is diagonalizable and its set of eigenvalues is \emph{contained} in $\calS$. 
This means that there exists an isomorphism (i.e., a choice of basis) $\mathcal{V}\cong \mathbb{C}^N$ for some $N\in\mathbb{Z}_{\ge0}$ and a decomposition
\begin{align}
N=N_1+N_2+\cdots+N_{\hat{q}} \quad (N_j\ge0,~ j=1,2,\ldots,\hat{q})
\end{align}
such that the $1$-form can be written as 
\begin{equation}\label{eq:odeconnOmega}
\Omega=S\, dx+\sum_{i=1}^q A_i \, d\log (x-a_i), \quad S, A_i \in \mathrm{Mat}(N,\bbC),
\end{equation}
where $S$ is of the form \eqref{eq:Sdiag}, that is, 
\begin{equation}
S=\begin{pmatrix}
\alpha_1 I_{N_1} & & & \\
& \alpha_2 I_{N_2}& & \\
& & \ddots & \\
& & & \alpha_{\hat{q}}I_{N_{\hat{q}}}
\end{pmatrix}
=
\sum_{j=1}^{\hat{q}}\alpha_jE_{N_j}.
\end{equation}
If $N_j = 0$, we omit the corresponding diagonal block $\alpha_j I_{N_j}$ in the block diagonal matrix, and $E_{N_j}$ is understood to be the zero matrix $O_N$. 
Note that the discussion in Section~\ref{sec:odemL} remains valid even when some $N_j$ vanish.

The connection matrix $A(x)$ of $\Omega$ can be written as
\begin{equation}\label{eq:odeconnOmega2}
A(x)=S+\sum_{i=1}^q \frac{A_i}{x-a_i}
\end{equation}
and hence, the category $\scrP(\calT,\calS)$ can be regarded as the category of ordinary differential equations of the form \eqref{eq:ode} (or Pfaffian systems of the form \eqref{eq:ode1form}) with diagonalizable $S$.
We note that the category $\scrP(\calT,\calS)$ is abelian.
\begin{remark}
Let $(\calV,\Omega),(\calV',\Omega')\in\scrP(\calT,\calS)$ be two objects.
We fix basis identifying $\calV\cong\bbC^N$ and $\calV'\cong\bbC^{N'}$ such that the $1$-form $\Omega$ is written as \eqref{eq:odeconnOmega} and
\begin{equation}\label{eq:odeconnOmegad}
\Omega'=S'+\sum_{i=1}^q A_i'\,d\log(x-a_i), \quad S', A_i' \in \mathrm{Mat}(N',\mathbb{C}).
\end{equation}
Then, the matrix representation $F\in\mathrm{Mat}(N'\times N,\bbC)$ of a morphism $f:(\calV,\Omega)\to(\calV',\Omega')$ with respect to the fixed basis satisfies 
\begin{align}\label{eq:odeCmor}
FS=S'F, \quad
FA_i=A_i'F \quad (1\le i\le q).
\end{align}
Hence, we see that the two objects $(\calV,\Omega)$ and $(\calV',\Omega')$ are isomorphic if and only if $N=N'$ and there exists a matrix $P\in \mathrm{GL}(N,\bbC)$ such that
\begin{align}
S=P^{-1}S'P, \quad 
A_i=P^{-1}A_i'P \quad (1\le i\le q).
\end{align}
\end{remark}

\subsubsection{Category of ordinary linear differential equations of Birkhoff-Okubo normal form}
In general, ordinary differential equations of the following form are called the \emph{Birkhoff-Okubo normal form}:
\begin{equation}\label{eq:odeBOgeneral}
(x-T)\frac{du}{dx}=\left(A+S(x-T)\right)u
\end{equation}
where $S$ and $T$ are diagonal constant matrices, and $A$ is a constant matrix. 
If the set of eigenvalues of $T$ is contained in $\calT$, then, by multiplying $(x-T)^{-1}$ from the left, the equation \eqref{eq:odeBOgeneral} becomes
\begin{equation}
\frac{du}{dx}=\left(S+(x-T)^{-1}A\right)u=\left(S+\sum_{i=1}^{q}\frac{A_i}{x-a_i}\right)u,
\end{equation} 
where $A_i$ $(1\le i\le q)$ are some constant matrices. 
With this in mind, we define the subcategory $\scrB=\scrB(\calT,\calS)$ of $\scrP(\calT,\calS)$, which consists of connections corresponding to ordinary differential equations of Birkhoff-Okubo normal form, as follows.
The set of objects $\mathrm{Ob}(\scrB(\calT,\calS))$ is defined by
\begin{align}
	\mathrm{Ob}(\mathscr{B}(\calT,\calS))=
	\left\{(\mathcal{V},\Omega)\in\mathrm{Ob}(\mathscr{P}(\calT,\calS))
	\,
	\middle \vert
	\,
	\begin{array}{l}
	\text{There exists an isomorphism }\mathcal{V}\cong \mathbb{C}^N 
	\\
	\text{(i.e., a choice of basis) for some $N\in\mathbb{Z}_{\ge0}$}
	\\ \text{such that $\Omega$ satisfies the condition (B)}
	\end{array}
	\right\}.
\end{align}
Here, condition (B) is the following: there exists a decomposition of $N$ 
\begin{align}
N=N_1+N_2+\cdots+N_{q} \quad (N_i\ge0,~ i=1,2,\ldots,q)
\end{align}
and a constant matrix $A\in\mathrm{Mat}(N,\mathbb{C})$ such that the $1$-form $\Omega$ can be written as 
\begin{align}\label{eq:odecatBOomega}
\Omega=\left(S+(x-T)^{-1}A\right)dx,
\end{align}
where $S$ is diagonal and 
\begin{align}\label{eq:odeBOT}
T=\begin{pmatrix}
		a_{1} I_{N_1} & & \\
		& a_2 I_{N_2} & \\
		& & \ddots & \\
		& & &  a_{q}I_{N_{q}}
		\end{pmatrix}.
	\end{align}
If $N_i = 0$, we omit the corresponding diagonal block $a_i I_{N_i}$ in the block diagonal matrix. 
For $(\calV,\Omega)\in \mathrm{Ob}(\scrP(\calT,\calS))$ with $\calV=\{0\}$ (corresponding to the case $N=0$), condition (B) is understood to be satisfied, with all matrices appearing therein interpreted as empty matrices.

We now define morphisms in $\scrB(\calT,\calS)$. 
For $(\calV,\Omega),(\calV',\Omega') \in \scrB(\calT,\calS)$, we set
\begin{equation}
\begin{aligned}
	\mathrm{Hom}_{\scrB(\calT,\calS)}&((\mathcal{V},\Omega),(\mathcal{V}',\Omega'))\\
	&:=\{f \in \mathrm{Hom}_{\scrP(\calT,\calS)}((\mathcal{V},\Omega),(\mathcal{V}',\Omega'))
	\mid
	f \text{ satisfies condition (Bmor)}
	\}.
\end{aligned}
\end{equation}
Here, condition (Bmor) is the following: for basis $\calV\cong \bbC^N$ and $\calV'\cong\bbC^{N'}$, we write the $1$-forms $\Omega$ and $\Omega'$ as \eqref{eq:odecatBOomega} and 
\begin{align}
\Omega'=\left(S'+(x-T')^{-1}A'\right)dx, \quad S',T',A'\in \mathrm{Mat}(N',\bbC)
\end{align}
respectively. 
Then, the matrix representation $F\in\mathrm{Mat}(N'\times N,\bbC)$ of $f$ satisfies 
	\begin{equation}\label{eq:catBmor}
	\begin{aligned}
	&FT=T'F,&
	&FS=S'F,&
	&FA=A'F.&
	\end{aligned}
	\end{equation}
If $N=0$ or $N'=0$, condition (Bmor) is understood to be  automatically satisfied, with the usual convention on empty matrices.
\subsection{Functors}
Retain the notation from the previous subsection.
We now introduce some constructions corresponding to the three steps of the (inverse) middle Laplace transform: extension, (inverse) Laplace transform, and projection.
The (inverse) middle Laplace transform functor is then obtained by applying these constructions in this order.

\subsubsection{Birkhoff-Okubo extension functor}
First, we define the functor 
\begin{align}
BO:\scrP(\calT,\calS)\to\scrB(\calT,\calS),
\end{align}
which corresponds to the extension of the differential equation \eqref{eq:ode} into the Birkhoff-Okubo normal form \eqref{eq:odeBO} as follows.

For $(\calV,\Omega)\in\scrP(\calT,\calS)$, we fix a basis $\calV\cong\bbC^{N}$ so that the $1$-form $\Omega$ is written as \eqref{eq:odeconnOmega} with $S$ diagonalized as in \eqref{eq:Sdiag}. 
Then, we define the $\mathrm{End}_{\bbC}(\calV^q)$-valued $1$-form $BO(\Omega)$ by the matrix representation 
\begin{align}\label{eq:odeBOomega}
BO(\Omega):=\left(S^{\oplus q}+(x-T)^{-1}(A-I)\right)dx,
\end{align}
where $S^{\oplus q}, T,A \in \mathrm{Mat}(qN,\bbC)$ are given by \eqref{eq:odeTAS}.
Then, we define
\begin{equation}
BO(\calV,\Omega):=(\calV^q,BO(\Omega)).
\end{equation}
By definition, it holds that $BO(\calV,\Omega)\in \mathrm{Ob}(\scrB(\calT,\calS))$. 

\medskip

We then define the morphism part of the functor $BO$.
For a morphism $f \in \mathrm{Hom}_{\scrP(\calT,\calS)}((\calV,\Omega),(\calV',\Omega'))$, we define the linear map $BO(f):\calV^q \to (\calV')^q$ by
\begin{equation}\label{eq:odeBOfunctorf}
BO(f):=f^{\oplus q}=\begin{pmatrix}
f  \\
 \vdots \\
 f
\end{pmatrix}.
\end{equation}
Then we have the following.
\begin{lemma}\label{lem:BOmor}
$BO(f)$ defines a morphism from $BO(\calV,\Omega)$ to $BO(\calV',\Omega')$.
\end{lemma}
\begin{proof}
For $(\calV,\Omega), (\calV',\Omega')\in\scrP(\calT,\calS)$, we fix basis $\calV\cong\bbC^{N}$ and $\calV'\cong\bbC^{N'}$ such that the $1$-forms $\Omega$ and $\Omega'$ are written as \eqref{eq:odeconnOmega} and \eqref{eq:odeconnOmegad}, respectively.
Then, the matrix representation $F\in \mathrm{Mat}(N'\times N,\bbC)$ of the morphism $f$ with respect to the above basis satisfies \eqref{eq:odeCmor}. 

Here, the $1$-forms $BO(\Omega)$ and $BO(\Omega')$ are given by \eqref{eq:odeBOomega} and 
\begin{equation}
BO(\Omega')=\left((S')^{\oplus q}+(x-T)^{-1}(A'-I)\right)dx,
\end{equation}
where $(S')^{\oplus q}, A' \in \mathrm{Mat}(qN',\bbC)$ are given by
\begin{equation}\label{eq:odeTASfunc}
\begin{aligned}
A'=\begin{pmatrix}
A_1' & A_2' & \dots & A_q' \\
A_1' & A_2' & \dots & A_q' \\
\vdots & \vdots &  & \vdots \\
A_1' & A_2' & \dots & A_q' \\
\end{pmatrix},
\quad
(S')^{\oplus q}=\bigoplus_{i=1}^{q} S'=\begin{pmatrix}
S' & & & \\
& S' & & \\
& & \ddots & \\
& & & S'
\end{pmatrix}.
\end{aligned}
\end{equation}
The matrix representation of $BO(f)$ with respect to the fixed basis is given by $F^{\oplus q}$.
Then, we can verify that \eqref{eq:catBmor} holds, which shows that $BO(f)$ defines a morphism from $BO(\calV,\Omega)$ to $BO(\calV',\Omega')$.
\end{proof}
\begin{definition}\label{def:BOextfunctor}
We refer to the functor $BO:\scrP(\calT,\calS)\to\scrB(\calT,\calS)$ as the \emph{Birkhoff-Okubo extension functor}, or simply the \emph{BO-extension functor}.
\end{definition}

\subsubsection{Laplace transform functor}

We proceed to define the Laplace transform functor, which we denote by $\calL$, corresponding to the Laplace transform in Definition \ref{def:odeLaplace}.
Set
\begin{align}
-\calT:=\{-a_1,\ldots,-a_q\}, 
\quad 
-\calS:=\{-\alpha_1,\ldots,-\alpha_{\hat{q}}\}.
\end{align}
We first introduce the operation $L$, which corresponds to the Laplace transform for Birkhoff-Okubo normal forms. 
On objects, it sends a connection in $\scrB(\calT,\calS)$ to a connection in $\scrB(\calS,-\calT)$ as follows.
For $(\calV,\Omega)\in\scrB(\calT,\calS)$, we fix a basis $\calV\cong\bbC^{N}$ so that the $1$-form $\Omega$ is written as \eqref{eq:odecatBOomega} with \eqref{eq:odeBOT}. 
Then, we define the $\mathrm{End}_{\bbC}(\calV)$-valued $1$-form $L(\Omega)$ by the matrix representation 
\begin{equation}\label{eq:odeLomega}
L(\Omega):=\left(-T-(x-S)^{-1}(A+I)\right)dx.
\end{equation}
Note that the corresponding differential equation of this connection is 
\begin{align}
\frac{dv}{dx}=\left(-T-(x-S)^{-1}(A+I)\right)v 
~
\Leftrightarrow
~
(x-S)\frac{dv}{dx}=-\left(A+I+(x-S)T\right)v,
\end{align}
which is nothing but the differential equation obtained by applying the Laplace transform \eqref{eq:op_FL} for the Birkhoff-Okubo normal form \eqref{eq:odeBOgeneral}.
Then, we define
\begin{equation}
L(\calV,\Omega):=(\calV,L(\Omega)) \in \mathrm{Ob}(\scrB(\calS,-\calT)).
\end{equation}
For $f\in\mathrm{Hom}_{\scrB(\calT,\calS)}((\calV,\Omega),(\calV',\Omega'))$, we define the linear map $L(f):\calV\to\calV'$ by
\begin{align}
L(f):=f.
\end{align}
Then, it is easily verified that the map $L(f)$ defines a morphism from $L(\calV,\Omega)$ to $L(\calV',\Omega')$ in $\scrB(\calS,-\calT)$. 

\medskip

We can define the operation $L^{-1}$ in the same way. 
That is, for $(\calV,\Omega)\in \scrB(\calT,\calS)$, we fix a basis $\calV\cong\bbC^{N}$ so that the $1$-form $\Omega$ is written as \eqref{eq:odecatBOomega} with \eqref{eq:odeBOT}. 
Then, we define the $\mathrm{End}_{\bbC}(\calV)$-valued $1$-form $L^{-1}(\Omega)$ by the matrix representation 
\begin{equation}
L^{-1}(\Omega):=\left(T-(x+S)^{-1}(A+I)\right)dx.
\end{equation}
Note that the corresponding differential equation of this connection is 
\begin{align}
\frac{dv}{dx}=\left(T-(x+S)^{-1}(A+I)\right)v 
~
\Leftrightarrow
~
(x+S)\frac{dv}{dx}=\left(-(A+I)+(x+S)T\right)v,
\end{align}
which is nothing but the differential equation obtained by applying the inverse Laplace transform \eqref{eq:op_invFL} for the Birkhoff-Okubo normal form \eqref{eq:odeBOgeneral}.
Then, we define
\begin{equation}
L^{-1}(\calV,\Omega):=(\calV,L^{-1}(\Omega)) \in \mathrm{Ob}(\scrB(-\calS,\calT)).
\end{equation}
For $f\in\mathrm{Hom}_{\scrB(\calT,\calS)}((\calV,\Omega),(\calV',\Omega'))$, we define the linear map $L^{-1}(f):\calV\to\calV'$ by
\begin{align}
L^{-1}(f):=f.
\end{align}
Then, it is easily verified that the map $L^{-1}(f)$ defines a morphism from $L^{-1}(\calV,\Omega)$ to $L^{-1}(\calV',\Omega')$ in $\scrB(-\calS,\calT)$. 

We are now ready to define the functors $\calL$ and $\calL^{-1}$. 
\begin{definition}
Let $I: \scrB\to\scrP$ denote the inclusion functor. 
We define the \emph{Laplace transform functor} $\calL:\scrP(\calT,\calS)\to\scrP(\calS,-\calT)$ by
\begin{align}
\calL:=I\circ L\circ BO.
\end{align}
For $(\calV,\Omega)\in \scrP(\calT,\calS)$, we write
\begin{equation}\label{eq:odecalLnotation}
\calL(\calV,\Omega)=(\calV^q,\calL(\Omega)).
\end{equation}
Similarly, we define the \emph{inverse Laplace transform functor} $\calL^{-1}:\scrP(\calT,\calS)\to\scrP(-\calS,\calT)$ by
\begin{align}
\calL^{-1}:=I\circ L^{-1}\circ BO.
\end{align}
For $(\calV,\Omega)\in \scrP(\calT,\calS)$, we write
\begin{equation}\label{eq:odecalinvLnotation}
\calL^{-1}(\calV,\Omega)=(\calV^q,\calL^{-1}(\Omega)).
\end{equation}
\end{definition}
\begin{remark}\label{rem:odefuncL}
For $(\calV,\Omega)\in \scrP(\calT,\calS)$, we fix a basis $\calV\cong\bbC^{N}$ such that the $1$-form $\Omega$ is written as \eqref{eq:odeconnOmega} with $S$ diagonalized as in \eqref{eq:Sdiag}. 
Then, the matrix representation of $1$-forms $\calL(\Omega)$ of \eqref{eq:odecalLnotation} and $\calL^{-1}(\Omega)$ of \eqref{eq:odecalinvLnotation} with respect to the basis $\calV^q\cong(\bbC^{N})^q$ are given by \eqref{eq:ode1formL} and \eqref{eq:ode1forminvL}, respectively.
\end{remark}

\begin{proposition}\label{prop:odefuncLexact}
The functors $\mathcal{L}$ and $\mathcal{L}^{-1}$ are exact.
\end{proposition}
\begin{proof}
Since the proof for $\mathcal{L}^{-1}$ is identical, it suffices to prove exactness for $\mathcal{L}$. 
Recall that $\scrP(\calT,\calS)$ and $\scrP(\calS,-\calT)$ are abelian categories (with kernels and cokernels computed on the underlying vector spaces, with the induced connection).
More concretely, for a morphism $f:(\calV,\Omega)\to (\calV',\Omega')$ in $\scrP(\calT,\calS)$, i.e.,\
$\Omega' f = f\Omega$, the subspace $\Ker f\subset \calV$ is $\Omega$-invariant,
hence inherits a connection $\Omega_{\Ker f}:=\Omega|_{\Ker f}$. 
Then the kernel of $f$ is given by $(\Ker f, \Omega_{\Ker f})$.
The quotient space
$\Coker f=\calV'/\im f$ also inherits a connection $\Omega'_{\Coker f}$
because $\im f$ is $\Omega'$-invariant.
Then the cokernel of $f$ is given by $(\Coker f,\Omega'_{\Coker f})$.  
The defining conditions of $\scrP(\calT,\calS)$ are preserved under passing to
$\Omega$-invariant subspaces and to quotients, so $(\Ker f,\Omega_{\Ker f})$ and
$(\Coker f,\Omega'_{\Coker f})$ are again objects of $\scrP(\calT,\calS)$.

Since $\calL$ is an additive functor between abelian categories, by \cite[Proposition 8.3.18]{KS2006} and its dual statement, it is enough to show that $\calL$ preserves kernels and cokernels. 
Since $\calL(f)=f^{\oplus q}$ by the definitions of $BO(f)$ and $L(f)$,
we see that
\[
\Ker \calL(f)=\Ker(f^{\oplus q})=(\Ker f)^{q}=\calL(\Ker f)
\] as vector spaces. 
We next compare the induced $1$-forms.
Let $i:\Ker f\hookrightarrow \calV$ be the inclusion.
Since $i$ defines a morphism $i:(\Ker f,\Omega_{\Ker f})\to (\calV,\Omega)$ in $\scrP(\calT,\calS)$ and $\calL$ is a functor, we have $\calL(i):(\calL(\Ker f),\calL(\Omega_{\Ker f}))\to(\calL(\calV),\calL(\Omega))$ is a morphism in $\scrP(\calS,-\calT)$. 
Hence, $\calL(\Omega) \calL(i)=\calL(i) \calL(\Omega_{\Ker f})$ holds. 
Since $\calL(i)=i^{\oplus q}$, we have
\[
\calL(\Omega) \circ i^{\oplus q}=i^{\oplus q}\circ \calL(\Omega_{\Ker f}).
\]
Combining this with $\Ker \calL(f)=(\Ker f)^{q}=(\im i)^{q}=\im(i^{\oplus q})$, we see that $\Ker \calL(f)$ is $\calL(\Omega)$-invariant and the induced $1$-form $\calL(\Omega)_{\Ker\calL(f)}=\calL(\Omega)|_{\Ker \calL(f)}$
coincides with $\calL(\Omega_{\Ker f})$.
Therefore $(\Ker\calL(f),\calL(\Omega)_{\Ker\calL(f)})$ is isomorphic to $(\calL(\Ker f),\calL(\Omega_{\Ker f}))$ in
$\scrP(\calS,-\calT)$.

Similarly, 
\[
\Coker \calL(f)= \Coker(f^{\oplus q})\cong (\Coker f)^{q}=\calL(\Coker f)
\]
as vector spaces.
Let $p:\calV'\twoheadrightarrow \Coker f$ be the projection.
Since $p$ defines a morphism $p:(\calV',\Omega')\to (\Coker f,\Omega'_{\Coker f})$ in $\scrP(\calT,\calS)$ and $\calL$ is a functor, we have $\calL(p):(\calL(\calV'),\calL(\Omega'))\to(\calL(\Coker f),\calL(\Omega'_{\Coker f}))$ is a morphism in $\scrP(\calS,-\calT)$. 
Hence, $\calL(\Omega'_{\Coker f}) \calL(p)=\calL(p) \calL(\Omega')$ holds.
Since $\calL(p)=p^{\oplus q}$, we have
\[\calL(\Omega'_{\Coker f}) \circ p^{\oplus q}=p^{\oplus q}\circ \calL(\Omega').\]
Combining this with $\im\calL(f)=\im(f^{\oplus q})=(\im f)^{q}=(\Ker p)^{q}=\Ker (p^{\oplus q})$, we see that $\im \calL(f)$ is $\calL(\Omega')$-invariant and hence $\calL(\Omega')$ descends to $\Coker\calL(f)$. 
Moreover, the induced $1$-form coincides with $\calL(\Omega'_{\Coker f})$, so $(\Coker\calL(f),\calL(\Omega')_{\Coker \calL(f)})$ is isomorphic to $(\calL(\Coker f),\calL(\Omega'_{\Coker f}))$ in $\scrP(\calS,-\calT)$.
This proves that $\calL$ is exact. 
\end{proof}

\subsubsection{Middle Laplace transform functor}
We consider defining the functor from $\scrP(\calT,\calS)$ to $\scrP(\calS,-\calT)$, which corresponds to the middle Laplace transform in Definition \ref{def:odeML}. 
For $(\calV,\Omega)\in\scrP(\calT,\calS)$, with the $1$-form \eqref{eq:odeconnOmegaEnd}, we define the linear subspace
\begin{equation}\label{eq:odeKv}
\calK_{\calV}:=\bigoplus_{i=1}^q\Ker A_i=
\left\{
	\begin{pmatrix}
	v_1 \\
	\vdots\\
	v_q
	\end{pmatrix}
	\in \mathcal{V}^q
	~\middle\vert~
	v_i \in \Ker A_{i}
	~
	(1\le i\le q)
	\right\}.
\end{equation}
We fix a basis $\calV\cong\bbC^{N}$ such that the $1$-form $\Omega$ is written as \eqref{eq:odeconnOmega} with $S$ diagonalized as in \eqref{eq:Sdiag}. 
Then, the subspace \eqref{eq:odeKv} is nothing but the subspace $\calK$ defined by \eqref{eq:odeK}. 
Moreover, by Remark~\ref{rem:odefuncL}, the matrix representation of $1$-form $\calL(\Omega)$ in $\calL(\calV,\Omega)$ with respect to the basis $\calV^q\cong(\bbC^{N})^q$ is given by \eqref{eq:ode1formL}. 
Therefore, thanks to Lemma \ref{lem:odeK}, we can consider a subconnection and a quotient connection of $(\calV,\calL(\Omega))$. 
This ensures the well-definedness of the following connections:
\begin{align}
\calK^{+}(\calV,\Omega)&:=(\calK_{\calV},\calL(\Omega)\vert_{\calK_{\calV}}) \in \scrP(\calS,-\calT), \\
\mathcal{ML}(\calV,\Omega)&:=\left(\calV^q/\calK_{\calV},\mathcal{ML}(\Omega)\right) \in \scrP(\calS,-\calT),
\end{align}
where $\calL(\Omega)\vert_{\calK}$ denotes the $\mathrm{End}(\calK)$-valued $1$-form obtained by restriction, and $\mathcal{ML}(\Omega)$ the induced $\mathrm{End}(\calV^q/\calK)$-valued $1$-form. 
By our choice of basis on $\calV$, the $1$-form $\mathcal{ML}(\Omega)$ of $\mathcal{ML}(\calV,\Omega)$ is indeed given by \eqref{eq:ode1formML}. 
Similarly, we can define the connections
\begin{align}
\calK^{-}(\calV,\Omega)&:=(\calK_{\calV},\calL^{-1}(\Omega)\vert_{\calK_{\calV}})\in \scrP(-\calS,\calT), \\
\mathcal{ML}^{-1}(\calV,\Omega)&:=\left(\calV^q/\calK_{\calV},\mathcal{ML}^{-1}(\Omega)\right) \in \scrP(-\calS,\calT).
\end{align}
Again, by fixing the basis of $\calV$, the $1$-form $\mathcal{ML}^{-1}(\Omega)$ of $\mathcal{ML}^{-1}(\calV,\Omega)$ is given by \eqref{eq:ode1forminvML}. 
\begin{remark}
$\calK^{+}$ and $\mathcal{ML}$ (resp. $\calK^{-}$ and $\mathcal{ML}^{-1}$) are functors from $\scrP(\calT,\calS)$ to $\scrP(\calS,-\calT)$ (resp. $\scrP(-\calS,\calT)$).
\end{remark}
\begin{definition}\label{def:odeMLfunctor}
We refer to the functors $\mathcal{ML}: \scrP(\calT,\calS)\to\scrP(\calS,-\calT)$ and $\mathcal{ML}^{-1}:\scrP(\calT,\calS)\to\scrP(-\calS,\calT)$ as the \emph{middle Laplace transform functor} and the \emph{inverse middle Laplace transform functor}, respectively.
\end{definition}

Since the inclusion $\iota:\calK\hookrightarrow \calV^q$ induces a morphism in $\scrP(\calS,-\calT)$ and $\scrP(-\calS,\calT)$, we obtain the following exact sequence in $\scrP(\calS,-\calT)$:
	\begin{equation}\label{eq:odeexactML}
		\begin{tikzcd}
		0\arrow[r]&\mathcal{K}^{+}(\mathcal{V},\Omega)\arrow[r,"\iota^+",hookrightarrow]
		&\mathcal{L}(\mathcal{V},\Omega) \arrow[r]
		&\Coker\iota^{+}
		\arrow[r]&0
		\end{tikzcd}
	\end{equation}
where $\iota^{+}$ denotes the morphism in $\scrP(\calS,-\calT)$ induced by $\iota$. 
Then, by definition, we have
\begin{equation}
\mathcal{ML}(\calV,\Omega)=\Coker \iota^{+}=\calL(\calV,\Omega)/\calK^{+}(\calV,\Omega).
\end{equation}
Similarly, we obtain the following exact sequence in $\scrP(-\calS,\calT)$:
	\begin{equation}\label{eq:odeexactinvML}
		\begin{tikzcd}
		0\arrow[r]&\mathcal{K}^{-}(\mathcal{V},\Omega)\arrow[r,"\iota^-",hookrightarrow]
		&\mathcal{L}^{-1}(\mathcal{V},\Omega) \arrow[r]
		&\Coker\iota^{-}
		\arrow[r]&0
		\end{tikzcd}
	\end{equation}
where $\iota^{-}$ denotes the morphism in $\scrP(-\calS,\calT)$ induced by $\iota$. 
Hence we have
\begin{equation}
\mathcal{ML}^{-1}(\calV,\Omega)=\Coker \iota^{-}=\calL^{-1}(\calV,\Omega)/\calK^{-}(\calV,\Omega).
\end{equation}

\section{Inversion formula and irreducibility (one variable case)}
\label{sec:odeinv}
Retain the notation from the previous section. In this section, we prove Theorem \ref{thm:odemain} by formulating it in terms of (meromorphic) connections.
First, we introduce the notions of irreducibility and exceptionality for connections.
\begin{definition}\label{def:odeconnirred}
Let $(\calV,\Omega)\in\scrP(\calT,\calS)$ be a connection with $\calV\neq \{0\}$, whose $1$-form $\Omega$ is given by \eqref{eq:odeconnOmegaEnd}.
\begin{itemize}
	\item We say that $(\calV,\Omega)$ is \emph{irreducible} if there is no $\langle S, A_1,\ldots,A_q\rangle$-invariant subspace of $\calV$ except $\calV$ and $\{0\}$.
	\item We say that $(\calV,\Omega)$ is \emph{exceptional} if $\dim\calV=1$ and $A_i=0$ for all $i=1,\ldots,q$. 
	Otherwise, we say it is \emph{non-exceptional}.
\end{itemize}
After fixing an isomorphism $\calV\simeq \bbC^N$, the above terminology agrees with the corresponding terminology for the 1-form \eqref{eq:odeconnOmega} in the sense of Definition \ref{def:odeirred}.
\end{definition}

	\begin{remark}\label{rem:odeDR}
	If a connection $(\mathcal{V},\Omega)\in\scrP(\calT,\calS)$ with the $1$-form \eqref{eq:odeconnOmegaEnd} is irreducible and non-exceptional, then it holds that 
	\begin{align}
	&\bigcap_{i=1}^{q}\Ker A_{i} \cap \Ker (S+c)=\{0\}	\quad (\forall c\in\mathbb{C})\tag{$\star$},\label{eq:odeDRstar1}\\
	&\sum_{i=1}^q \im A_{i}+\im (S+c)=\mathcal{V} \quad (\forall c\in\mathbb{C}).	\label{eq:odeDRstar2}\tag{$\star\star$}
	\end{align}
	These conditions are additive versions of the conditions $(*)$ and $(**)$ in Dettweiler-Reiter \cite[\S 3]{DR2000}.
	\end{remark}
Theorem \ref{thm:odemain} can be rephrased as follows. 
\begin{theorem}\label{thm:odeconn}
Suppose that $(\mathcal{V},\Omega)\in\scrP(\calT,\calS)$ is irreducible and non-exceptional.
Then the following hold.
	\begin{enumerate}[(i)]
	\item $\mathcal{ML}^{-1}\circ\mathcal{ML}(\mathcal{V},\Omega)\sim(\mathcal{V},\Omega)$ and 
	$\mathcal{ML}\circ\mathcal{ML}^{-1}(\mathcal{V},\Omega)\sim(\mathcal{V},\Omega)$. 

	\item Both $\mathcal{ML}(\mathcal{V},\Omega)$ and $\mathcal{ML}^{-1}(\mathcal{V},\Omega)$ are also irreducible and non-exceptional. 
		\end{enumerate}
Here, the symbol $\sim$ denotes an isomorphism of connections.
\end{theorem}

\subsection{Inversion formula}
\label{subsec:invformulaode}
This subsection is devoted to the proof of the inversion formula (Theorem \ref{thm:odeconn}~(i)).
We will prove only the statement 
\[
\mathcal{ML}^{-1}\circ\mathcal{ML}(\mathcal{V},\Omega)\sim(\mathcal{V},\Omega),
\] 
since the other one can be shown in the same way.

Suppose that $(\calV,\Omega)\in\scrP(\calT,\calS)$ is irreducible and non-exceptional, and that the $1$-form $\Omega$ is expressed as \eqref{eq:odeconnOmegaEnd}. 
Let $\calS'$ be the set of eigenvalues of $S$. 
Then we have $(\calV,\Omega)\in\scrP(\calT,\calS')$.
Since $\calS'\subseteq\calS$, the category $\scrP(\calT,\calS')$ is a full subcategory of $\scrP(\calT,\calS)$. 
Therefore, for the proof of the inversion formula, it suffices to prove the assertion after replacing $\calS$ by $\calS'$.
After this replacement, using the same notation $\calS=\{\alpha_1,\ldots,\alpha_{\hat q}\}$, we may assume in \eqref{eq:odeconnOmega} that $N_j\ge1$ for all $j=1,2,\ldots,\hat{q}$.

Our goal is to show the existence of an isomorphism as connections
\begin{equation}
\mathcal{ML}^{-1}(\mathcal{ML}(\calV,\Omega))\to (\calV,\Omega).
\end{equation}
To this end, we will proceed in the following two steps. 
\begin{enumerate}[\bf Step 1.]
\item  Show that $(\calV,\Omega)$ is isomorphic to the quotient connection
\[
(\calV,\Omega) \sim \calL^{-1}(\calL(\calV,\Omega))/\calK'(\calV,\Omega),
\]
where $\calK'(\calV,\Omega)$ is a subconnection of $\calL^{-1}(\calL(\calV,\Omega))$ to be introduced later.

\item Show that 
\begin{equation}\label{eq:odeinvStep2}
\calL^{-1}(\calL(\calV,\Omega))/\calK'(\calV,\Omega)\sim \mathcal{ML}^{-1}(\mathcal{ML}(\calV,\Omega)).
\end{equation}
\end{enumerate}

We begin with \textbf{Step~1}, that is, we describe the connection $(\calV,\Omega)$ as a quotient of $\calL^{-1}(\calL(\calV,\Omega))$. 
Fix an isomorphism (a basis) $\calV\cong\bbC^N$ such that the $1$-form $\Omega$ is expressed as \eqref{eq:odeconnOmega} with $S$ diagonalized as in \eqref{eq:Sdiag}. 
That is, the coefficient matrix $A(x)$ of $\Omega$ is given by
\begin{equation}\label{eq:coeffinvform}
A(x)=S+\sum_{i=1}^q \frac{A_i}{x-a_i}, 
\quad 
S, A_i \in\mathrm{Mat}(N,\bbC),
\end{equation}
where 
\begin{equation}\label{eq:Sdiaginvform}
S=\begin{pmatrix}
\alpha_1 I_{N_1} & & & \\
& \alpha_2 I_{N_2}& & \\
& & \ddots & \\
& & & \alpha_{\hat{q}}I_{N_{\hat{q}}}
\end{pmatrix}.
\end{equation}
By applying the Laplace transform functor and the inverse Laplace transform functor to $(\calV,\Omega)$ in succession, we obtain the connection
	\begin{equation}\label{eq:odeLL}
	\mathcal{L}^{-1}(\mathcal{L}(\mathcal{V},\Omega))=((\mathcal{V}^q)^{\hat{q}},\mathcal{L}^{-1}(\mathcal{L}(\Omega))).
	\end{equation}
For this connection, we now fix a basis $(\calV^q)^{\hat{q}} \cong (\mathbb{C}^{Nq})^{\hat{q}}$ so that the matrix representation of the associated $1$-form 
\begin{equation}\label{eq:ode1-formLinvL}
\calL^{-1}(\calL(\Omega))=G_0\,dx+\sum_{i=1}^{q}G_{i}d\log(x-a_i)
\end{equation}
is obtained by successively applying the Laplace transform and then the inverse Laplace transform (see Definitions~\ref{def:odeLaplace} and~\ref{def:odeinvLaplace}) to the original $1$-form $\Omega$. 
A direct calculation yields the explicit form of the matrices $G_0, G_i\in\mathrm{Mat}(Nq\hat{q},\bbC)$ as follows:
	\begin{align}
	&G_0=
	\begin{pmatrix}
		\alpha_{1}I_{qN} & & &  \\
		& \alpha_{2}I_{qN}   & &  \\
& & \ddots & \\
& & & \alpha_{\hat{q}}I_{qN}
	\end{pmatrix},
	\label{eq:odeG0} 
\\
	&G_{i}
	=\begin{pmatrix}
		G_1^i & G_2^i& \cdots & G_{\hat{q}}^i \\
		G_1^i & G_2^i& \cdots & G_{\hat{q}}^i \\
		\vdots & \vdots &  & \vdots \\
		G_1^i & G_2^i& \cdots & G_{\hat{q}}^i
	\end{pmatrix} 
	,\quad 
	G_{j}^{i}=
	\begin{pmatrix}
	O_N & O_N & \cdots   &O_N \\
	\vdots & \cdots & \cdots &  \vdots \\
	E_{N_j}A_{1} & E_{N_j}A_{2} & \cdots & E_{N_j}A_{q}  \\
	\vdots & \cdots & \cdots &  \vdots \\
	O_N & O_N & \cdots  &O_N \\	
	\end{pmatrix}(\,i.
	\label{eq:odeGi} 
	\end{align}

We then define a linear map $\varphi:(\mathcal{V}^q)^{\hat{q}}\to\mathcal{V}$ as follows:
for $\bm{v}={}^t(v_1,\ldots,v_{\hat{q}})\in(\mathcal{V}^q)^{\hat{q}}$, where each $v_j={}^t(v_1^j,\ldots,v_q^j)\in\mathcal{V}^q$, we set
	\begin{equation}\label{eq:odevarphi}
	\varphi(\bm{v}):=\sum_{j=1}^{\hat{q}}E_{N_j} \sum_{k=1}^q A_{k}v_k^j.
	\end{equation}
Then, this linear map \eqref{eq:odevarphi} defines a morphism of connections.
\begin{proposition}
\label{prop:odevarphimor}
	The linear map \eqref{eq:odevarphi} defines a morphism of connections 
	\begin{equation}\label{eq:odevarphiasmor}
	\varphi:\mathcal{L}^{-1}(\mathcal{L}(\mathcal{V},\Omega))=((\calV^q)^{\hat{q}},\calL^{-1}(\calL(\Omega)))
	\to(\mathcal{V},\Omega).
	\end{equation}
\end{proposition}
\begin{proof}
It is sufficient to show that
\begin{align}
 &\varphi\circ G_0 =S \circ \varphi, \label{eq:odemorproof0}\\
 &\varphi\circ G_i=A_i \circ \varphi \quad (1\le i\le q). \label{eq:odemorproofi}
\end{align}
We prove the first statement \eqref{eq:odemorproof0}. 
Let $\bm{v}={}^t(v_1,\ldots,v_{\hat{q}})\in(\mathcal{V}^q)^{\hat{q}}$, where each $v_i={}^t(v_1^i,\ldots,v_q^i)\in\mathcal{V}^q$. 
Then, we have
\begin{align}
G_0\bm{v}=	\begin{pmatrix}
		\alpha_{1}I_{qN} & & &  \\
		& \alpha_{2}I_{qN}   & &  \\
& & \ddots & \\
& & & \alpha_{\hat{q}}I_{qN}
	\end{pmatrix}
	\begin{pmatrix}
	v_1 \\
	v_2 \\
	\vdots \\
	v_{\hat{q}}
	\end{pmatrix} 
	=
		\begin{pmatrix}
	\alpha_{1} v_1 \\
	\alpha_{2} v_2 \\
	\vdots \\
	\alpha_{\hat{q}} v_{\hat{q}}
	\end{pmatrix}
\end{align}
and hence
\begin{align}
(\varphi\circ G_0)(\bm{v})=\sum_{j=1}^{\hat{q}}E_{N_j} \sum_{k=1}^q  A_{k}(\alpha_j v_k^j)
=\sum_{j=1}^{\hat{q}}\alpha_j E_{N_j} \sum_{k=1}^q  A_{k}v_k^j.
\end{align}
On the other hand, since $S=\alpha_1 E_{N_1}+\alpha_2 E_{N_2}+\cdots+\alpha_{\hat{q}}E_{N_{\hat{q}}}$ and $E_{N_j}E_{N_k}=\delta_{jk}E_{N_j}$ (where $\delta_{jk}$ denotes the Kronecker delta), we have
\begin{align}
(S\circ \varphi)(\bm{v})&=S\sum_{j=1}^{\hat{q}}E_{N_j} \sum_{k=1}^q A_{k}v_k^j \\
&=(\alpha_1 E_{N_1}+\alpha_2 E_{N_2}+\cdots+\alpha_{\hat{q}}E_{N_{\hat{q}}})\sum_{j=1}^{\hat{q}}E_{N_j} \sum_{k=1}^q A_{k}v_k^j \\
&= \sum_{j=1}^{\hat{q}}\alpha_j E_{N_j}\sum_{k=1}^q  A_{k}v_k^j,
\end{align}
which shows \eqref{eq:odemorproof0}.

We proceed to show the statement \eqref{eq:odemorproofi}.
Take $\bm{v}={}^t(v_1,\ldots,v_{\hat{q}})\in(\mathcal{V}^q)^{\hat{q}}$, where each $v_j={}^t(v_1^j,\ldots,v_q^j)\in\mathcal{V}^q$. 
Then, we have 
\begin{align}
G_{i}\bm{v}=\begin{pmatrix}
		G_1^i & G_2^i& \cdots & G_{\hat{q}}^i \\
		G_1^i & G_2^i& \cdots & G_{\hat{q}}^i \\
		\vdots & \vdots &  & \vdots \\
		G_1^i & G_2^i& \cdots & G_{\hat{q}}^i
	\end{pmatrix} 
	\begin{pmatrix}
	v_1 \\
	v_2 \\
	\vdots \\
	v_{\hat{q}}
	\end{pmatrix} 
	=
	\begin{pmatrix}
	w_1 \\
	w_2 \\
	\vdots\\
	w_{\hat{q}}
	\end{pmatrix}, 
\end{align}
where
\begin{align}
w_j&=\sum_{k=1}^{\hat{q}}G_k^i v_k=\sum_{k=1}^{\hat{q}}\begin{pmatrix}
	O_N & O_N & \cdots   &O_N \\
	\vdots & \cdots & \cdots &  \vdots \\
	E_{N_k}A_{1} & E_{N_k}A_{2} & \cdots & E_{N_k}A_{q}  \\
	\vdots & \cdots & \cdots &  \vdots \\
	O_N & O_N & \cdots  &O_N \\	
	\end{pmatrix}(i
	\begin{pmatrix}
	v_1^k \\
	v_2^k \\
	\vdots \\
	v_q^k
	\end{pmatrix} \\
&=\sum_{k=1}^{\hat{q}}
	\begin{pmatrix}
	0_N \\
	\vdots \\
	\ds E_{N_k}\sum_{l=1}^q A_lv_l^k \\
	\vdots \\
	0_N
	\end{pmatrix}(i 
=
	\begin{pmatrix}
	0_N \\
	\vdots \\
	\ds\sum_{k=1}^{\hat{q}}E_{N_k}\sum_{l=1}^q A_lv_l^k \\
	\vdots \\
	0_N
	\end{pmatrix}(i 
=	\begin{pmatrix}
	0_N \\
	\vdots \\
	\varphi(\bm{v})\\
	\vdots \\
	0_N
	\end{pmatrix}(i
	=:\begin{pmatrix}
	w_1^j \\
	\vdots \\
	w_i^j \\
	\vdots \\
	w_q^j
	\end{pmatrix}.
\end{align}
Although $w_1 = w_2 = \cdots = w_{\hat{q}}$ and $w_m^j=0$ ($m\neq i$), we retain the subscripts for the sake of clarity in subsequent calculations.
Then, we obtain
\begin{align}
(\varphi\circ G_i)(\bm{v})&=\sum_{j=1}^{\hat{q}}E_{N_j} \sum_{m=1}^q  A_{m}w_{m}^j 
=\sum_{j=1}^{\hat{q}}E_{N_j}  A_{i}w_{i}^j 
 \\
&=\sum_{j=1}^{\hat{q}}E_{N_j}  A_{i}\varphi(\bm{v}) =(E_{N_1}+E_{N_2}+\cdots+E_{N_{\hat{q}}})A_{i}\varphi(\bm{v}) \\
&=A_i \varphi(\bm{v}),
\end{align}
which shows \eqref{eq:odemorproofi}.
\end{proof}

Here, we shall show the properties of \eqref{eq:odevarphi}.
\begin{lemma}\label{lem:odevarphi}
The linear map \eqref{eq:odevarphi} satisfies 
	\[
	\Ker \varphi=\mathcal{K}', 
	\quad 
	\im \varphi=\mathcal{V}.
	\]
Here, 
	\begin{equation}\label{eq:odeK'}
	\mathcal{K}':=
	\left\{
	\begin{pmatrix} 
	v_1 \\
	\vdots \\
	v_{\hat{q}}
	\end{pmatrix}
	\in (\mathcal{V}^q)^{\hat{q}}
	~\middle \vert\,
	\begin{array}{l} 
	v_j={}^t(v_1^j,\ldots,v_q^j)\in\mathcal{V}^q
	\text{ satisfies}
	\\\ds
	E_{N_j}\sum_{k=1}^q A_{k}v_k^j=0
	\quad
	(j=1,\ldots,\hat{q})
	\end{array}
	\right\}.
	\end{equation}
\end{lemma}
\begin{proof}
First, we show that $\Ker \varphi=\mathcal{K}'$. 
Since the inclusion $\Ker \varphi\supset\mathcal{K}'$ is obvious, it suffices to prove $\Ker \varphi\subset\mathcal{K}'$.
Take $\bm{v}={}^t(v_1,\ldots,v_{\hat{q}}) \in \Ker\varphi\subset(\mathcal{V}^q)^{\hat{q}}$, and write each $v_j={}^t(v_1^j,\ldots,v_q^j)$.
Then we have 
	\begin{align*}
	\varphi(\bm{v})&=\sum_{j=1}^{\hat{q}}E_{N_j} \sum_{k=1}^q A_{k}v_k^j \\
	&=\begin{pmatrix}
	I_{N_1} & & & \\
	& O_{N_2} & & \\
	& & \ddots & \\
	& & & O_{N_{\hat{q}}}
	\end{pmatrix}
	\sum_{k=1}^q A_{k}v_k^1 
	+\cdots
	+\begin{pmatrix}
	O_{N_1} & & & \\
	& O_{N_2} & & \\
	& & \ddots & \\
	& & & I_{N_{\hat{q}}}
	\end{pmatrix}
	\sum_{k=1}^q A_{k}v_k^{\hat{q}} \\
	&
	=\begin{pmatrix}
	0_{N_1} \\
	\vdots \\
	0_{N_{\hat{q}}}
	\end{pmatrix}.
	\end{align*}
Hence, for each $j=1,\ldots,\hat{q}$, we have
	\[
	E_{N_j}\sum_{k=1}^q A_{k}v_{k}^{j}=0,
	\]
which shows $\bm{v}\in\mathcal{K}'$.
This proves $\Ker \varphi = \mathcal{K}'$.

Next, we show the statement $\im\varphi=\mathcal{V}$.
Take $w \in \mathcal{V}$ and express
	\[
	w=E_{N_1}w+E_{N_2}w+\cdots+E_{N_{\hat{q}}}w.
	\]
Since $(\calV,\Omega)$ is irreducible and non-exceptional, by \eqref{eq:odeDRstar2}, for each $j=1,\ldots,\hat{q}$, we have
	\[
	\sum_{i=1}^q \im A_{i}+\im \left(S-\alpha_j\right)=\mathcal{V}.
	\]
Thus, for each $j=1,\ldots,\hat{q}$, there exist $v_0^j, v_1^j,\ldots,v_q^j\in \calV$ such that 
	\[
	E_{N_j}w=\sum_{i=1}^q A_{i}v_{i}^{j}+\left(S-\alpha_{j}\right)v_0^j.
	\]
By multiplying $E_{N_j}$, we have
	\[
	E_{N_j}w=E_{N_j}\sum_{i=1}^q A_{i}v_{i}^{j}
	\]
since $E_{N_j} (S - \alpha_{j}) = O$.
Now, define $\bm{v} = {}^t(v_1, \ldots, v_{\hat{q}}) \in (\mathcal{V}^q)^{\hat{q}}$ with $v_j = {}^t(v_1^j, \ldots, v_q^j) \in \mathcal{V}^q$. 
Then, we have
\[
\varphi(\bm{v}) = \sum_{j=1}^{\hat{q}} E_{N_j} \sum_{i=1}^q A_{i} v_i^j = \sum_{j=1}^{\hat{q}} E_{N_j} w = w.
\]
This shows that $\varphi$ is surjective, as desired.
\end{proof}
As a corollary of Proposition \ref{prop:odevarphimor} and Lemma \ref{lem:odevarphi}, we have the following.
\begin{corollary}
The connection
\[
\mathcal{K}'(\mathcal{V},\Omega):=(\mathcal{K}',\mathcal{L}^{-1}(\mathcal{L}(\Omega))\vert_{\mathcal{K}'}) \in \scrP(\calT,\calS)
\]
is a subconnection of $\mathcal{L}^{-1}(\mathcal{L}(\mathcal{V},\Omega))$. 
Moreover, the morphism \eqref{eq:odevarphiasmor} induces the isomorphism of connections
	\begin{equation}\label{eq:odemorphi}
	\bar{\varphi}: \mathcal{L}^{-1}(\mathcal{L}(\mathcal{V},\Omega))/\mathcal{K}'(\mathcal{V},\Omega)\stackrel{\sim}{\longrightarrow}(\mathcal{V},\Omega).
	\end{equation}
\end{corollary}
We note that the quotient connection $\mathcal{L}^{-1}(\mathcal{L}(\mathcal{V},\Omega))/\mathcal{K}'(\mathcal{V},\Omega)$ is described as
\begin{equation}
\mathcal{L}^{-1}(\mathcal{L}(\mathcal{V},\Omega))/\mathcal{K}'(\mathcal{V},\Omega)=\left((\mathcal{V}^q)^{\hat{q}}/\mathcal{K}',
	\mathcal{L}^{-1}(\mathcal{L}(\Omega))\Big\vert_{(\mathcal{V}^q)^{\hat{q}}/\mathcal{K}'}\right).
\end{equation}

Let us proceed to \textbf{Step~2}, that is, we shall prove \eqref{eq:odeinvStep2}.
To do that, we prepare the following lemma concerning the expression of $\calK'$, which plays an important role later in the proof.

\begin{lemma}\label{lem:odeK'}
The vector space $\mathcal{K}'$ defined in \eqref{eq:odeK'} can be expressed as
	\begin{equation}\label{Eq_K'2ode}
	\mathcal{K}'
	=\left\{
	\begin{pmatrix} 
	v_1 \\
	\vdots \\
	v_{\hat{q}}
	\end{pmatrix}
	\in (\mathcal{V}^q)^{\hat{q}}
	~\middle \vert\,
	B_{j}v_j \in \mathcal{K}_{\calV}
	~ (1\le j \le\hat{q})
	\right\}
	\end{equation}
where $B_{j}$ and $\calK_{\calV}$ are given by \eqref{eq:odeB} and \eqref{eq:odeKv}, respectively.
\end{lemma} 
\begin{proof}
We abbreviate the right-hand side of \eqref{Eq_K'2ode} to $\mathcal{K}_0$ for short. 
We take $\bm{v}={}^t(v_1,\ldots,v_{\hat{q}})\in \mathcal{K}'$.
Then, we have $B_{j}v_j =0$ $(j=1,\ldots,\hat{q})$ by the definition \eqref{eq:odeK'} of $\mathcal{K}'$, which shows $\bm{v}\in \mathcal{K}_0$ and hence
	$
	\mathcal{K}'
	\subset
	\mathcal{K}_0.
	$

Conversely, let $\bm{v}={}^t(v_1,\ldots,v_{\hat{q}})\in \mathcal{K}_0$ where each
	$
	v_j={}^t(v_1^j,\ldots,v_q^j)\in\mathcal{V}^q
	$.
Then we have
	\[
		B_{j}v_j=
	-
	\begin{pmatrix} \ds
	E_{N_j}A_{1} & E_{N_j}A_{2}& \cdots &E_{N_j}A_{q} \\ 
	E_{N_j}A_{1} & E_{N_j}A_{2}& \cdots & E_{N_j}A_{q} \\
	\vdots & \vdots &  &\vdots\\
	E_{N_j}A_{1}& E_{N_j}A_{2}&\cdots &E_{N_j}A_{q} 
	\end{pmatrix}
	\begin{pmatrix}
	v_1^j \\ 
	v_2^j \\
	\vdots \\
	v_q^j
	\end{pmatrix}
	=
	-
	\begin{pmatrix}
	E_{N_j}\sum_{k=1}^q A_{k}v_k^j \\
	E_{N_j}\sum_{k=1}^q A_{k}v_k^j \\
	\vdots \\
	E_{N_j}\sum_{k=1}^q A_{k}v_k^j
	\end{pmatrix}
	\]
	for each $j=1,\ldots, \hat{q}$. 
	Since we assumed $B_{j}v_j \in \mathcal{K}_{\calV}$, we have
	\[
	E_{N_j}\sum_{k=1}^q A_{k}v_k^j
	\in
	\bigcap_{i=1}^q \Ker A_{i}.
	\]
Besides, it holds that
	\[
	\left(S-\alpha_{j}\right)E_{N_j}\sum_{k=1}^q A_{k}v_k^j
	=0.
	\]
Therefore we have
	\[
	E_{N_j}\sum_{k=1}^q A_{k}v_k^j\in \bigcap_{i=1}^q\Ker A_{i} \cap \Ker\left(S-\alpha_j\right) 
	=\{0\}
	\]
for each $j=1,\ldots,\hat{q}$.
Here we used the property \eqref{eq:odeDRstar1}. 
As a consequence, we obtain that each $v_j\in\calV^q$ satisfies
	\[
	E_{N_j}\sum_{k=1}^q A_{k}v_k^j=0
	\quad (j=1,\ldots,\hat{q}).
	\]
It shows the desired implication
	$
	\mathcal{K}_0\subset\mathcal{K}'
	$.
\end{proof}
We now consider showing the isomorphism \eqref{eq:odeinvStep2}.
To do this, we first construct a morphism from $\calL^{-1}(\calL(\mathcal{V},\Omega))$ to $\calL^{-1}(\mathcal{ML}(\mathcal{V},\Omega))$.
Let $f_{*}:\calL(\mathcal{V},\Omega)\to\mathcal{ML}(\mathcal{V},\Omega)$ be a quotient morphism. 
It is nothing but the quotient map from $\mathcal{V}^q$ to $\mathcal{V}^q/\mathcal{K}$ as a linear map. 
Then, by definition, we see that the morphism
$f:=\calL^{-1}(f_{*})$ from $\calL^{-1}(\calL(\mathcal{V},\Omega))$ to $\calL^{-1}(\mathcal{ML}(\mathcal{V},\Omega))$ is given by the linear map
	\begin{equation}\label{Eq_mor_f_ode}
	\begin{array}{ccccc}
	f:&(\mathcal{V}^q)^{\hat{q}}& \to & (\mathcal{V}^q/\mathcal{K}_{\calV})^{\hat{q}}& \\
	& \rotatebox{90}{$\in$}  & & \rotatebox{90}{$\in$} & \\
	& \begin{pmatrix}v_1 \\
	\vdots \\
	v_{\hat{q}}
	\end{pmatrix} & \mapsto & 
	\begin{pmatrix}\bar{v}_1 \\
	\vdots \\
	\bar{v}_{\hat{q}}
	\end{pmatrix}, &\quad \bar{v}_j=v_j+\mathcal{K}_{\calV}.
	\end{array}
	\end{equation}
We then construct a morphism 
\begin{equation}
f':\calL^{-1}(\calL (\mathcal{V},\Omega))/\mathcal{K}'(\mathcal{V},\Omega)\to\mathcal{ML}^{-1}(\mathcal{ML}(\mathcal{V},\Omega))
\end{equation}
 as follows. 
Let $\pi:\calL^{-1}(\mathcal{ML}(\mathcal{V},\Omega))\to\mathcal{ML}^{-1}(\mathcal{ML}(\mathcal{V},\Omega))$ be the quotient morphism.
As a linear map, the quotient morphism $\pi$ is nothing but
\begin{equation}
\pi:(\mathcal{V}^q/\mathcal{K}_{\calV})^{\hat{q}} \to (\mathcal{V}^q/\mathcal{K}_{\calV})^{\hat{q}}\left/\bigoplus_{j=1}^{\hat{q}}\Ker \bar{B}_{j}, \right.
\end{equation}
where $\bar{B}_{j}$ are the coefficient matrices of $\mathcal{ML}(\Omega)$ in \eqref{eq:ode1formML}.
Therefore, the composition $\pi\circ f:\calL^{-1}(\calL (\mathcal{V},\Omega))\to\mathcal{ML}^{-1}(\mathcal{ML}(\mathcal{V},\Omega))$ is given by the linear map
	\[
	\begin{array}{ccccc}
	\pi\circ f:&(\mathcal{V}^q)^{\hat{q}}& \to & \ds(\mathcal{V}^q/\mathcal{K}_{\calV})^{\hat{q}}\left/\bigoplus_{j=1}^{\hat{q}}\Ker \bar{B}_{j}\right.& \\
	& \rotatebox{90}{$\in$}  & & \rotatebox{90}{$\in$} & \\
	& \begin{pmatrix}
	v_1 \\ 
	\vdots \\
	v_{\hat{q}}
	\end{pmatrix} & \mapsto & 
	\begin{pmatrix}
	\bar{v}_1+\Ker \bar{B}_{1} \\
	\vdots \\
	\bar{v}_{\hat{q}}+\Ker \bar{B}_{\hat{q}}
	\end{pmatrix}, & 
	\bar{v}_j=v_j+\mathcal{K}_{\calV}.
	\end{array}
	\]
For $\bm{v},\bm{w} \in (\mathcal{V}^q)^{\hat{q}}$ satisfying $\bm{v}-\bm{w}\in\mathcal{K}'$, it can be easily verified that
	\[
	(\pi\circ f)(\bm{v})=(\pi\circ f)(\bm{w})
	\]
by using Lemma \ref{lem:odeK'}.
Hence, the linear map $\pi\circ f$ induces the linear map
	\begin{equation}\label{Eq_f'}
	\begin{array}{ccccc}
	f':&(\mathcal{V}^q)^{\hat{q}}/\mathcal{K}'& ~~~~~\to &~~~~~\ds (\mathcal{V}^q/\mathcal{K}_{\calV})^{\hat{q}}\left/\bigoplus_{j=1}^{\hat{q}}\Ker \bar{B}_{j} \right.& \\
	& \rotatebox{90}{$\in$}  & & \rotatebox{90}{$\in$} & \\
	& \begin{pmatrix}
	v_1 \\ 
	\vdots \\
	v_{\hat{q}}
	\end{pmatrix}+\mathcal{K}' &~~~~ \mapsto & 
	\begin{pmatrix}
	\bar{v}_1+\Ker \bar{B}_{1} \\
	\vdots \\
	\bar{v}_{\hat{q}}+\Ker \bar{B}_{\hat{q}}
	\end{pmatrix}, & 
	\bar{v}_j=v_j+\mathcal{K}_{\calV}.
	\end{array}
	\end{equation}
	%
We can check that this linear map defines a morphism of connections. 
Then, we obtain the morphism 
\begin{equation}\label{eq:odemorf'}
f':\calL^{-1}(\calL (\mathcal{V},\Omega))/\mathcal{K}'(\mathcal{V},\Omega)\to\mathcal{ML}^{-1}(\mathcal{ML}(\mathcal{V},\Omega)).
\end{equation}
Then, the diagram
	\begin{equation}\label{Eq_com_MLMode}
		\begin{tikzcd}
		\calL^{-1}(\calL (\mathcal{V},\Omega)) \arrow[r,"\pi'"]
		\arrow[d,"f "']
		&\calL^{-1}(\calL (\mathcal{V},\Omega))/\mathcal{K}'(\mathcal{V},\Omega)\arrow[d,"f'"]\\
		 \calL^{-1}(\mathcal{ML}(\mathcal{V},\Omega))  \arrow[r,"\pi"]
		&\mathcal{ML}^{-1}(\mathcal{ML}(\mathcal{V},\Omega))
	\end{tikzcd}
	\end{equation}
is commutative.
Here $\pi'$ denotes the quotient morphism.
Finally, we shall show that the morphism \eqref{eq:odemorf'} indeed gives an isomorphism \eqref{eq:odeinvStep2} of connections, which follows from the following proposition.
\begin{proposition}\label{Prop_bijc}
The linear map $f'$ defined by \eqref{eq:odemorf'} is bijective.
\end{proposition}
\begin{proof}
The surjectivity follows from the surjectivity of $f$ and the commutative diagram \eqref{Eq_com_MLMode}.
We shall show the injectivity. 
Suppose that $\bm{v}={}^t(v_1,\ldots,v_{\hat{q}})\in(\mathcal{V}^q)^{\hat{q}}$ satisfies $f'(\bm{v}+\mathcal{K}')=0$. 
Then we have $\bar{B}_{j}\bar{v}_j=0$ in $\mathcal{V}^q/\mathcal{K}_{\calV}$ ($1\le j\le \hat{q}$). 
This means that $B_{j}v_j\in\mathcal{K}_{\calV}$ ($1\le j\le \hat{q}$).
Then, by the expression \eqref{Eq_K'2ode}, we have $\bm{v}\in \mathcal{K}'$.
\end{proof}

As a consequence, we have an isomorphism of connections
	\begin{equation}
	\bar{\varphi}\circ (f')^{-1}:\mathcal{ML}^{-1}(\mathcal{ML}(\mathcal{V},\Omega))\stackrel{\sim}{\longrightarrow}(\mathcal{V},\Omega),
	\end{equation}
which shows $\mathcal{ML}^{-1}\circ\mathcal{ML}(\mathcal{V},\Omega)\sim(\mathcal{V},\Omega)$, as desired.

\subsection{Irreducibility}
This subsection is devoted to proving Theorem \ref{thm:odeconn} (ii).
Let us prepare a lemma.
\begin{lemma}[cf. Lemma 2.8 in Dettweiler-Reiter \cite{DR2000}]\label{Lem_inclusion}
The functor $\mathcal{ML}:\scrP(\calT,\calS)\to\scrP(\calS,-\calT)$ satisfies the following.
	\begin{enumerate}[(i)]
	\item It preserves the inclusion.
	Namely, if $(\mathcal{W},\Omega\vert_{\mathcal{W}}) \in \scrP(\calT,\calS)$ is a subconnection of $(\mathcal{V},\Omega) \in \scrP(\calT,\calS)$, then
	\begin{equation}\label{Eq_lem_incode}
	\mathcal{ML}(\mathcal{W},\Omega\vert_{\mathcal{W}})
	\subset
	\mathcal{ML}(\mathcal{V},\Omega).
	\end{equation}
	\item It preserves the direct sum.
	Namely, if $(\mathcal{V},\Omega)$ can be written $(\mathcal{V},\Omega)=(\mathcal{W}_1,\Omega_1)\oplus(\mathcal{W}_2,\Omega_2)$ by $(\mathcal{W}_1,\Omega_1), (\mathcal{W}_2,\Omega_2) \in \scrP(\calT,\calS)$, then
	\[
	\mathcal{ML}(\mathcal{V},\Omega)=
	\mathcal{ML}(\mathcal{W}_1,\Omega_1)\oplus\mathcal{ML}(\mathcal{W}_2,\Omega_2).
	\]
	\end{enumerate}
The functor $\mathcal{ML}^{-1}:\scrP(\calT,\calS)\to\scrP(-\calS,\calT)$ also satisfies the above two statements.
\end{lemma}
\begin{proof}
We prove the assertion (i) for $\mathcal{ML}$. 
Since the functor $\mathcal{L}$ is exact (see Proposition \ref{prop:odefuncLexact}), it preserves the inclusion. 
Hence, we have $\calL(\mathcal{W},\Omega\vert_{\mathcal{W}})\subset\calL(\mathcal{V},\Omega)$ as connection. 
Note that the morphism $\iota:=\calL(\iota_{*})$ coming from the inclusion $\iota_{*}:\mathcal{W}\hookrightarrow\mathcal{V}$ gives the inclusion from $\calL(\mathcal{W},\Omega\vert_{\mathcal{W}})\subset\calL(\mathcal{V},\Omega)$. 
Here, it is nothing but the inclusion map $\iota:\mathcal{W}^q\hookrightarrow\mathcal{V}^q$ as a linear map.
Since $\mathcal{K}_{\mathcal{W}}=\mathcal{W}^q\cap \mathcal{K}_{\mathcal{V}}$, the linear map $\iota$ induces the inclusion map
	\[
	\bar{\iota}: \mathcal{W}^q/\mathcal{K}_\mathcal{W}\hookrightarrow\mathcal{V}^q/\mathcal{K}_{\mathcal{V}}.
	\]
This can be regarded as a morphism from $\mathcal{ML}(\mathcal{W},\Omega\vert_{\mathcal{W}})$ to $\mathcal{ML}(\mathcal{V},\Omega)$ and gives the desired inclusion \eqref{Eq_lem_incode}.
We can show the assertion for $\mathcal{ML}^{-1}$ in the same way. 

The assertion (ii) follows from (i).  
\end{proof}
Now we show Theorem \ref{thm:odeconn} (ii).
\begin{proof}[Proof of Theorem \ref{thm:odeconn} (ii) \text{(cf. the proof of \cite[Corollary 3.6]{DR2000})}] 
Suppose that $(\calV,\Omega)\in\scrP(\calT,\calS)$ is irreducible and non-exceptional. 
Fix an isomorphism $\calV\cong\bbC^N$ such that the $1$-form $\Omega$ is expressed as \eqref{eq:odeconnOmega} with $S$ diagonalized as in \eqref{eq:Sdiag}. 
Then the $1$-form $\mathcal{ML}(\Omega)$ of $\mathcal{ML}(\calV,\Omega)=(\calV^q/\calK_{\calV},\mathcal{ML}(\Omega))$ is expressed as
	\begin{equation}
	\mathcal{ML}(\Omega)=-\bar{T} dx+\sum_{j=1}^{\hat q}\bar{B}_j \,d\log(x-\alpha_j).
	\end{equation}
Here $\bar{T},\bar{B}_j$ are the endomorphisms of $\calV^q/\calK_{\calV}$ induced by $T$ and $B_j$ in \eqref{eq:odeTAS} and \eqref{eq:odeB}, respectively. 

By Theorem \ref{thm:odeconn} (i), the connection $\mathcal{ML}(\calV,\Omega)$ is nonzero.
Let $(\mathcal{W},\mathcal{ML}(\Omega)\vert_{\mathcal{W}})$ be a minimal nonzero subconnection of $\mathcal{ML}(\mathcal{V},\Omega)$. 
By the minimality, $(\mathcal{W},\mathcal{ML}(\Omega)\vert_{\mathcal{W}})$ is irreducible.

We first show that $(\mathcal{W},\mathcal{ML}(\Omega)\vert_{\mathcal{W}})$ is non-exceptional. 
Suppose, to the contrary, that it is exceptional. 
Then $\dim \calW=1$. 
Since $\bar{T}$ is induced by $T=\mathrm{diag}[a_1I_N,\ldots,a_qI_N]$, the scalar by which $\bar{T}$ acts on $\calW$ is one of $a_1,\ldots,a_q$. 
Hence there exists a unique index $m$ such that $\bar{T}|_{\calW}=a_m\,\mathrm{id}_{\calW}$. 
Moreover, $\bar{B}_{j}\vert_{\calW}=0$ for all $j=1,\ldots,\hat q$.
Take a nonzero element $\bar{w}\in\calW$ and let $w={}^t(w_1,w_2,\ldots,w_q)\in\calV^{q}$ be a representative of $\bar{w}$. 
The condition $\bar{B}_{j}\bar{w}=0$ means $B_jw\in\calK_{\calV}$ ($j=1,\ldots,\hat{q}$). 
Then, by the same argument as in the proof of Lemma \ref{lem:odeK'}, using \eqref{eq:odeDRstar1} for $(\calV,\Omega)$, we obtain
	\[
	E_{N_j}\sum_{k=1}^q A_{k}w_k=0
	\quad (j=1,\ldots,\hat{q})
	\]
and hence
	\begin{equation}\label{eq:lemirred}
	\sum_{k=1}^q A_{k}w_k=0. 
	\end{equation}
Next we consider the condition $(\bar{T}-a_m)\bar{w}=0$, which is equivalent to $(T-a_m)w\in\calK_{\calV}$. 
Thus \begin{equation}\label{eq:lemirred2}
	(a_i-a_m)A_iw_i=0 \quad (i=1,\ldots,q). 
	\end{equation}
Since $a_1,\ldots,a_q$ are distinct, we have $A_iw_i=0$ for all $i\neq m$.
Together with \eqref{eq:lemirred}, this gives $A_mw_m=0$. 
Hence $A_iw_i=0$ for all $i=1,\ldots,q$.
Therefore we obtain $w\in \calK_{\calV}$, and so $\bar{w}=0$, a contradiction. 
Thus $(\calW,\mathcal{ML}(\Omega)\vert_{\calW})$ is non-exceptional.

Applying the inverse middle Laplace transform $\mathcal{ML}^{-1}$ to $(\mathcal{W},\mathcal{ML}(\Omega)\vert_{\mathcal{W}})$, we have
	\[
	\mathcal{ML}^{-1}(\mathcal{W},\mathcal{ML}(\Omega)\vert_{\mathcal{W}})\subset (\mathcal{V},\Omega)
	\]
by Lemma \ref{Lem_inclusion} (i) and Theorem \ref{thm:odeconn} (i).
Moreover, $(\mathcal{W},\mathcal{ML}(\Omega)\vert_{\calW})$ is irreducible and non-exceptional, so Theorem \ref{thm:odeconn} (i) shows that $\mathcal{ML}^{-1}(\mathcal{W},\mathcal{ML}(\Omega)\vert_{\mathcal{W}})$ is nonzero. 
Since $(\calV,\Omega)$ is irreducible, we obtain 
	\[
	\mathcal{ML}^{-1}(\mathcal{W},\mathcal{ML}(\Omega)\vert_{\mathcal{W}})=(\mathcal{V},\Omega).
	\]
Applying the middle Laplace transform $\mathcal{ML}$ again and using Theorem \ref{thm:odeconn} (i), we get
	\[
	(\mathcal{W},\mathcal{ML}(\Omega)\vert_{\mathcal{W}})=\mathcal{ML}(\mathcal{V},\Omega).
	\]
This shows that $\mathcal{ML}(\mathcal{V},\Omega)$ is irreducible and non-exceptional. 
The assertion for $\mathcal{ML}^{-1}(\calV,\Omega)$ is proved in the same way.
\end{proof}

\section{Generalized middle convolution in one variable}
\label{sec:odemc}
By employing the middle Laplace transform, we generalize the middle convolution to linear ordinary differential equations with an irregular singular point. 
We consider the same setting as in Section~\ref{sec:odemL}, namely, the linear ordinary differential equation \eqref{eq:ode}, or equivalently, the Pfaffian system \eqref{eq:ode1form},
where $S$ is diagonalizable and satisfies 
\begin{equation}\label{eq:Sdiag2}
S\sim\begin{pmatrix}
\alpha_1 I_{N_1} & & & \\
& \alpha_2 I_{N_2}& & \\
& & \ddots & \\
& & & \alpha_{\hat{q}}I_{N_{\hat{q}}}
\end{pmatrix}.
\end{equation}
Here, $\alpha_1, \dots, \alpha_{\hat{q}} \in \mathbb{C}$, and $N = N_1 + N_2 + \cdots + N_{\hat{q}}$ is a partition of $N$.

\begin{definition}\label{def:odeaddition}
For the differential equation \eqref{eq:ode} and $\lambda=(\lambda_1,\lambda_2,\ldots,\lambda_{q})\in\bbC^{q}$, we refer to the operation that sends \eqref{eq:ode} to the equation 
\begin{align}
\frac{dv}{dx}=\left(S+\sum_{i=1}^q \frac{A_{i}+\lambda_i}{x-a_i}\right)v
\end{align}
as the \emph{addition} with parameter $\lambda$, and denote it by $add_{\lambda}$.
\end{definition}
By definition, it is clear that this operation is invertible and keeps the irreducibility. 
We then define the middle convolution as follows. 
\begin{definition}\label{def:odemc}
For the differential equation \eqref{eq:ode} and $\beta=(\beta_1,\beta_2,\ldots,\beta_{\hat{q}})\in\bbC^{\hat{
q}}$, we define the \emph{middle convolution} $mc_{\beta}$ with parameter $\beta$ by the following composition:
\begin{equation}\label{eq:odemc}
mc_{\beta}:=\mathcal{ML}^{-1}\circ add_{-\beta}\circ \mathcal{ML}.
\end{equation}
\end{definition}
We note that the addition $add_{\lambda}$ and the middle convolution $mc_{\beta}$ induce the endofunctors on $\scrP(\calT,\calS)$ with \eqref{eq:odeTS}, and we call these functors the \emph{addition functor} and the \emph{middle convolution functor}, respectively.
By abuse of notation, we also write $add_{\lambda}$ and $mc_{\beta}$ for these functors.

\medskip

As a corollary of Theorem \ref{thm:odeconn}, we obtain some fundamental properties of the middle convolution functor.
\begin{theorem}\label{thm:odemc}
	Suppose that $(\calV,\Omega)\in\scrP(\calT,\calS)$ is irreducible and non-exceptional. 
	Then 
	$mc_{0}(\calV,\Omega)\sim(\calV,\Omega)$.
	Furthermore, for any $\beta\in\bbC^{\hat{q}}$ such that $add_{-\beta}\circ\mathcal{ML}(\calV,\Omega)$ is non-exceptional, $mc_{\beta}(\calV,\Omega)$ is also irreducible and non-exceptional.
	In this case, 
	\begin{equation}\label{eq:mcpropode}
	mc_{\gamma}\circ mc_{\beta}(\calV,\Omega)\sim mc_{\beta+\gamma}(\calV,\Omega)
	\end{equation}
	holds for any $\gamma\in\bbC^{\hat{q}}$.
Analogous statements hold for the irreducible and non-exceptional equation \eqref{eq:ode}, under the corresponding non-exceptionality condition on $add_{-\beta}\circ \mathcal{ML}$.
\end{theorem}
\begin{proof}
The statement $mc_{0}(\calV,\Omega)\sim(\calV,\Omega)$ is a direct consequence of Theorem \ref{thm:odeconn} (i). 
By Theorem \ref{thm:odeconn} (ii), $\mathcal{ML}(\calV,\Omega)$ is irreducible. 
Since the addition functor preserves irreducibility, it follows that $add_{-\beta}\circ\mathcal{ML}(\calV,\Omega)$ is irreducible. 
Besides, since we assumed that $add_{-\beta}\circ\mathcal{ML}(\calV,\Omega)$ is non-exceptional, we can apply Theorem \ref{thm:odeconn} (ii) to $add_{-\beta}\circ\mathcal{ML}(\calV,\Omega)$. 
Then we see that $mc_{\beta}(\calV,\Omega)$ is irreducible and non-exceptional.
Finally, applying Theorem \ref{thm:odeconn} (i) to $add_{-\beta}\circ\mathcal{ML}(\calV,\Omega)$, we obtain
\begin{align}
	mc_{\gamma}\circ mc_{\beta}(\calV,\Omega)&=\mathcal{ML}^{-1}\circ add_{-\gamma}\circ (\mathcal{ML}\circ\mathcal{ML}^{-1})\circ add_{-\beta}\circ\mathcal{ML}(\calV,\Omega) \\
	&\sim \mathcal{ML}^{-1}\circ add_{-\gamma}\circ add_{-\beta}\circ\mathcal{ML}(\calV,\Omega) \\
	&=\mathcal{ML}^{-1}\circ add_{-(\beta+\gamma)}\circ\mathcal{ML}(\calV,\Omega)=mc_{\beta+\gamma}(\calV,\Omega).
\end{align}
This proves the assertion. 
The corresponding statement for the equation \eqref{eq:ode} follows from the connection version.
\end{proof}
Moreover, the middle convolution in Definition~\ref{def:odemc} gives a generalization of the middle convolution introduced by Dettweiler and Reiter for Fuchsian systems. 
To see this, consider the Fuchsian system
\begin{equation}\label{eq:Fuchsian}
\frac{du}{dx}=\left(\sum_{i=1}^{q}\frac{A_i}{x-a_i}\right)u, \quad A_i \in \mathrm{Mat}(N,\bbC).
\end{equation}
This corresponds to the special case of equation \eqref{eq:ode} with $\hat{q}=1$, $N_1=N$ and $\alpha_1=0$ (i.e., $S=O$). 
Here, we briefly recall the original middle convolution with $\beta\in\bbC$ introduced by Dettweiler-Reiter \cite{DR2000}. 
For the equation \eqref{eq:Fuchsian}, we set 
\begin{align}
A=\begin{pmatrix}
	A_{1} & A_{2}& \cdots &A_{q} \\
	A_{1} & A_{2}& \cdots &A_{q} \\
	\vdots & \vdots &  &\vdots\\
	A_{1} & A_{2}& \cdots &A_{q} 
	\end{pmatrix},
\quad 
C_i=\begin{pmatrix}
	O_N&\cdots&\cdots&\cdots& O_N\\
	&\cdots&\cdots&\cdots& \\
	A_1 & \cdots  & A_i+\beta& \cdots & A_q \\
	&\cdots&\cdots&\cdots& \\
	O_N&\cdots&\cdots&\cdots& O_N \\
	\end{pmatrix}(i
	\quad (1\le i\le q).
\end{align} 
Then the subspaces $\calK=\bigoplus_{i=1}^q \Ker A_i$ and 
\begin{equation}
\calK_{\infty}:=\Ker(C_1+C_2+\cdots+C_q)
=\Ker(A+\beta)
\end{equation}
are common invariant subspaces of $C_1,C_2,\ldots,C_q$. 
We note that the subspace $\calK$ is the same as the case of the middle Laplace transform.
Let $\bar{C}_i$ be the matrix which represents the linear actions of $C_i$ on the quotient space $\bbC^{qN}/(\calK+\calK_\infty)$.
Then, the operation which sends the equation \eqref{eq:Fuchsian} to 
\begin{equation}\label{eq:odeDRmcFuchsian}
\frac{dy}{dx}=\left(\sum_{i=1}^{q}\frac{\bar{C}_i}{x-a_i}\right)y
\end{equation}
is called the middle convolution with $\beta$, which was originally introduced by Dettweiler-Reiter.

We now show the composition $\mathcal{ML}^{-1}\circ add_{-\beta}\circ \mathcal{ML}$ for the Fuchsian system \eqref{eq:Fuchsian} realizes Dettweiler-Reiter's original middle convolution \eqref{eq:odeDRmcFuchsian}.
By applying the Laplace transformation in Definition~\ref{def:odeLaplace} for the Fuchsian system \eqref{eq:Fuchsian}, we obtain 
\begin{equation}
\frac{dV}{dx}=\left(-T+\frac{B_1}{x}\right)V,
\end{equation}
where 
\begin{equation}
T=\begin{pmatrix}
a_1 I_N & & & \\
& a_2 I_N & & \\
& & \ddots & \\ 
& & & a_q I_N
\end{pmatrix}, 
\quad 
B_1=-A=-\begin{pmatrix}
	A_{1} & A_{2}& \cdots &A_{q} \\
	A_{1} & A_{2}& \cdots &A_{q} \\
	\vdots & \vdots &  &\vdots\\
	A_{1} & A_{2}& \cdots &A_{q} 
	\end{pmatrix}.
\end{equation}
By projecting this equation to the quotient space $\bbC^{qN}/\calK$ with $\calK=\bigoplus_{i=1}^q \Ker A_i$, we have
\begin{equation}
\frac{dv}{dx}=\left(-\bar{T}+\frac{\bar{B}_1}{x}\right)v, 
\quad  \bar{T},\bar{B}_1 \in\mathrm{Mat}(N',\bbC),
\end{equation}
where $N'=qN-\dim\calK$ and $\bar{T},\bar{B}_1$ are matrices representing the induced action of $T,B_1$ on the quotient space $\bbC^{qN}/\calK$.
Without loss of generality, we may fix a basis of the vector space $\bbC^{N'}\cong\bbC^{qN}/\calK$ such that the matrix $\bar{T}$ is of the form
\begin{equation}
\bar{T}=\begin{pmatrix}
a_1 I_{k_1} & & & \\
 & a_2 I_{k_2} & & \\
 & & \ddots & \\
 & & & a_q I_{k_q}
\end{pmatrix},
\quad 
k_i=N-\dim \Ker A_i~(1\le i \le q).
\end{equation}
Here we note that, letting $\iota:\bbC^{qN}\to \bbC^{N'}$ be the quotient map, we have
\begin{align}
\Ker\bar{B}_1=\{\iota(v)\in\bbC^{N'} \mid v\in \Ker B_1\}.
\end{align}
We then apply the addition $add_{-\beta}$ with $-\beta$. 
Then we have
\begin{equation}
\frac{dv_{\rm add}}{dx}=\left(-\bar{T}+\frac{\bar{B}_1-\beta}{x}\right)v_{\rm add}.
\end{equation}
By applying the inverse Laplace transform in Definition \ref{def:odeinvLaplace}, we have
\begin{equation}
\frac{dW}{dx}=\left(\sum_{i=1}^{q}\frac{G_i}{x-a_i}\right)W
\end{equation}
where $G_i=-E_{k_i}(\bar{B}_1-\beta)$ for $1\le i\le q$. 
By considering the projection to the quotient space $\bbC^{N'}/\Ker(\bar{B}_1-\beta)$, we obtain
\begin{equation}\label{eq:odecompMLadd}
\frac{dw}{dx}=\left(\sum_{i=1}^{q}\frac{\bar{G}_i}{x-a_i}\right)w,
\end{equation}
where $\bar{G}_i$ denotes the matrix which represents the induced action of $G_i$ on $\bbC^{N'}/\Ker(\bar{B}_1-\beta)$. 
We shall show that the equation \eqref{eq:odecompMLadd} is constant gauge equivalent to the equation \eqref{eq:odeDRmcFuchsian}.
Noting that $\calK_{\infty}=\Ker(A+\beta)=\Ker(B_1-\beta)$, we have 
\begin{equation}\label{eq:odemcFuchscong}
\begin{aligned}
\bbC^{N'}/\Ker(\bar{B}_1-\beta)&\cong(\bbC^{qN}/\calK)/\Ker(\bar{B}_1-\beta) \\
&\cong (\bbC^{qN}/\calK)/\left\{\left(\calK+\calK_{\infty}\right)/\calK\right\} \\
&\cong \bbC^{qN}/(\calK+\calK_{\infty}).
\end{aligned}
\end{equation}
By chasing the isomorphism \eqref{eq:odemcFuchscong}, we see that the action on $\bbC^{N'}/\Ker(\bar{B}_1-\beta)$ induced by $G_i$ is the same as the action on $\bbC^{qN}/(\calK+\calK_{\infty})$ induced by the matrix
\begin{align}
\begin{pmatrix}
	O_N &  &    & & \\
	 & \ddots &  &&  \\
	 & & I_N & &  \\
	 & &&  \ddots& \\
	&  & & &O_N\\	
	\end{pmatrix}
	&
	\begin{pmatrix}
	A_1+\beta &A_2 &\cdots & A_q \\
	A_1 & A_2+\beta & \cdots &A_q \\
	\vdots & & \ddots& \vdots \\
	A_1 & A_2 & \cdots &A_q+\beta
	\end{pmatrix} \\
	&
	=\begin{pmatrix}
	O_N&\cdots&\cdots&\cdots& O_N\\
	&\cdots&\cdots&\cdots& \\
	A_1 & \cdots  & A_i+\beta& \cdots & A_q \\
	&\cdots&\cdots&\cdots& \\
	O_N&\cdots&\cdots&\cdots& O_N \\
	\end{pmatrix}(i
	=C_i.
\end{align}
This means that the tuple of matrices $(\bar{G}_1,\bar{G}_2,\ldots,\bar{G}_{q})$ is constant gauge equivalent to $(\bar{C}_1,\bar{C}_2,\ldots,\bar{C}_{q})$. 
This shows that the equation \eqref{eq:odeDRmcFuchsian} is constant gauge equivalent to the equation \eqref{eq:odecompMLadd}. 
\begin{remark}
As mentioned in the introduction, several researchers have proposed generalizations of the middle convolution for linear ordinary differential equations with irregular singular points \cite{Arinkin2010, Kawakami2010, Takemura2011, Yamakawa2011, Yamakawa2016}. 
Although our generalized middle convolution is similar to Yamakawa's in that both make use of the Laplace transform (which he interprets it as Harnad duality), our definition explicitly proceeds via the Birkhoff-Okubo normal form.
\end{remark}

\section{Middle Laplace transform in several variables}
\label{sec:PfaffmL}

In what follows, we consider generalizing the middle Laplace transform for linear Pfaffian systems of several variables. 
Namely, we prolong the middle Laplace transform, introduced for linear ordinary differential equations in the previous section, as an operation for linear Pfaffian systems. 
Let us recall the setting explained in the introduction.
Let $n$ be a natural number and $\mathcal{A}$ be a hyperplane arrangement in $\mathbb{C}^n$.
Fix a coordinate $(x_1,\ldots,x_n)$ of $\mathbb{C}^n$.
For $H \in \mathcal{A}$, we denote $f_H$ by the defining linear polynomial of $H$.
We then consider linear Pfaffian systems with logarithmic singularities along $\mathcal{A}$ and irregular singularities along the hyperplanes $\{x_i=\infty\}$ ($i=1,\ldots,n$)
	\begin{equation}\label{Eq_Pfaff}
	du=\Omega u, \quad 
	\Omega=\sum_{i=1}^nS_{x_i}(\bm{x})\,dx_i
	+\sum_{H \in \mathcal{A}}A_H \, d\log f_H,
	\end{equation}
where $\bm{x}=(x_1,\ldots,x_n)$,
	\begin{equation}\label{Eq_Si}
	S_{x_i}(\bm{x})=A_{x_i}+\sum_{\substack{j=1 \\ j\neq i}}^n A_{x_ix_j} x_j 
	\end{equation}
and $A_{x_i}, A_{x_i x_j}, A_{H}$ are constant $N\times N$ matrices.
Throughout the paper, we assume
	\begin{equation}\label{Eq_xixj}
	A_{x_ix_j}=A_{x_jx_i}~\left(\Leftrightarrow~\frac{\p S_{x_i}}{\p x_j}=\frac{\p S_{x_j}}{\p x_i}\right)\quad (i\neq j).
	\end{equation}
Then, the completely integrable condition of the system \eqref{Eq_Pfaff} is given by
	\begin{equation}\label{Eq_integrable}
	\Omega\wedge\Omega=O.
	\end{equation}
We also assume \eqref{Eq_integrable} throughout this paper.
This can be expressed as the commutative conditions for coefficient matrices $A_{x_i}$, $A_{x_i x_j}$ and $A_{H}$, which will be explained later.

	Following Haraoka \cite{Haraoka2012}, we prepare some notation.
For each $i$ ($1\le i\le n$), we set
 	\begin{equation}\label{Eq_A_xi}
 	\mathcal{A}_{x_i}=\{H \in \mathcal{A}\mid(f_{H})_{x_i} \neq 0\}.
 	\end{equation}
Then the system \eqref{Eq_Pfaff} is written as
	\begin{equation}\label{Eq_Pfaff2}
	\frac{\p u}{\p x_i}=
	\Omega_{x_i}u
	=
	\left(S_{x_i}(\bm{x})+
	 \sum_{H\in\mathcal{A}_{x_i}} A_H \frac{(f_H)_{x_i}}{f_H} \right) u \quad (1\le i\le n).
	\end{equation}
Here we note that $S_{x_i}(\bm{x})$ does not depend on $x_i$. 
We refer to the equation \eqref{Eq_Pfaff2} as the $x_i$-equation of \eqref{Eq_Pfaff}.

Take a pair of distinct variables $(x_i,x_j)$ and write $(x,y)=(x_i,x_j)$.
Then, we set
	\begin{equation}\label{Eq_notation}
	x'=(x_1,\ldots,\check{x}_i,\ldots,x_n),
	\quad 
	x''=(x_1,\ldots,\check{x}_i,\ldots,\check{x}_j,\ldots,x_n),
	\end{equation}
where $\check{x}_i$ stands for the omission of $x_i$.
Namely, $(x,x')=(x,y,x'')$.

For $H\in\mathcal{A}_x$, we define the linear polynomial $a_H \in\mathbb{C}[x']$ by
	\begin{equation}\label{Eq_aH}
	f_H=(f_H)_x(x-a_H).
	\end{equation}
Then the $x$-equation of the system \eqref{Eq_Pfaff} can be written as
	\[
	\frac{\p u}{\p x}=\left(S_{x}(\bm{x})+
	\sum_{H\in\mathcal{A}_x} \frac{A_H}{x-a_{H}}\right)u.
	\]
Similarly, for $H\in\mathcal{A}_y$, we define the linear polynomial $b_H\in\mathbb{C}[x,x'']$ by
	\[
	f_H=(f_H)_y(y-b_H).
	\]
For $H \in \mathcal{A}_x$, we define the subset
	\[
	\mathcal{C}_{H,y}
	=\{H' \in \mathcal{A}_x\setminus\{H\}\mid
	(a_{H}-a_{H'})_y \neq 0\}.
	\]
For each $H' \in \mathcal{C}_{H,y}$, we define the linear polynomial $c_{HH'}\in\mathbb{C}[x'']$ by
	\[
	a_{H}-a_{H'}=(a_{H}-a_{H'})_y(y-c_{HH'}).
	\]
Clearly, $c_{HH'}=c_{H'H}$ holds. 
\begin{example}
Set $n=3$ and consider the hyperplane arrangement $\calA=\{H_1,H_2,H_3,H_4\}$ consisting of 
\[
H_1=\{x_1=0\}, \quad
H_2=\{x_1=x_2\}, \quad 
H_3=\{2x_1+x_2+3x_3=0\}, \quad 
H_4=\{x_3=1\}.
\]
Then we have
\[
f_{H_1}=x_1,\quad f_{H_2}=x_1-x_2, \quad f_{H_3}=2x_1+x_2+3x_3, \quad f_{H_4}=x_3-1
\]
and 
\[
\calA_{x_1}=\{H_1,\, H_2,\,H_3\}, \quad 
\calA_{x_2}=\{H_2,\,H_3\}, \quad 
\calA_{x_3}=\{H_3,\,H_4\}.
\]
If we choose $(x,y)=(x_1,x_2)$, then we have
\begin{equation}
a_{H_1}=0,\quad a_{H_2}=x_2, \quad a_{H_3}=-\frac{x_2}{2}-\frac{3}{2}x_3, \quad 
b_{H_2}=x_1, \quad b_{H_3}=-2x_1-3x_3
\end{equation}
and 
\begin{align}
&\calC_{H_1,y}=\{H_2,H_3\}, \quad \calC_{H_2,y}=\{H_1,H_3\}, \quad \calC_{H_3,y}=\{H_1,H_2\}, \\
&c_{H_1H_2}=0, \quad c_{H_2H_3}=-x_3, \quad c_{H_1H_3}=-3x_3.
\end{align}
\end{example}

\begin{proposition}[cf. Haraoka \cite{Haraoka2012}, Theorem 2.1]\label{prop:cic}
	The system \eqref{Eq_Pfaff} satisfies the completely integrable condition \eqref{Eq_integrable} if and only if, for any pair  $(i,j)$ ($1\le i,j \le n)$ of distinct indices, by setting $(x_i,x_j)=(x,y)$, the following conditions hold:
\begin{align}
&[A_{x},A_{xy}]=[A_{xy},A_{xz}]=O \quad (z\in x''), \label{eq:cond1p}
\\
&[A_{x},A_{y}]+\left[A_{xy},\sum_{H\in\mathcal{A}_{x}^{c}\cap\mathcal{A}_{y}}A_{H}-\sum_{H\in\mathcal{A}_{x}}A_{H}\right]=O, \label{eq:cond3p}
\\
&[A_{x},A_{yz}]+[A_{xz},A_{y}]=O \quad (z\in x''), \label{eq:cond4p} 
\\
&[A_{xz},A_{yz}]=O \quad (z\in x''), \label{eq:cond5p} 
\\
&[A_{xz_1},A_{yz_2}]+[A_{xz_2},A_{yz_1}]=O \quad (z_1,z_2 \in x''),
\end{align}
$\bullet$ For $H\in \mathcal{A}_{x}\cap\mathcal{A}_{y}$, 
\begin{align}
&[A_{H},A_{xy}]=O,\label{eq:cond8p} 
\\
& [A_{H},A_{y}+(a_{H})_yA_{x}]=O, \label{eq:cond9p} 
\\
&[A_{H},(a_{H})_yA_{xz}+A_{yz}]=O\quad (z\in x''), \label{eq:cond10p} 
\\
&\left[A_{H},\sum_{\substack{H''\in\mathcal{C}_{H,y} \\ c_{HH''}=b_{H'}}}A_{H''}+A_{H'}\right]=O \quad (H'\in \mathcal{A}_{x}^{c}\cap\mathcal{A}_{y}), \label{eq:cond11p} 
\\
&\left[A_{H},\sum_{\substack{H''\in\mathcal{C}_{H,y} \\ c_{HH''}=c_{HH'}}}A_{H''}\right]=O \quad\left(\begin{array}{c}H'\in \mathcal{C}_{H,y} \text{ such that}\\ c_{HH'}\neq b_{K} \text{ for any }K\in\mathcal{A}_{x}^{c}\cap\mathcal{A}_{y}\end{array}\right). \label{eq:cond12p}
\end{align}
$\bullet$ For $H\in\mathcal{A}_{x}\cap\mathcal{A}_{y}^{c}$, 
\begin{align}
&[A_{H},A_{y}+a_{H}(0)A_{xy}]=O, \label{eq:cond13p} 
\\
&[A_{H},(a_{H})_zA_{xy}+A_{yz}]=O \quad (z\in x''), \label{eq:cond14p} 
\\
&\left[A_{H},\sum_{\substack{H''\in\mathcal{A}_{x}\cap\mathcal{A}_{y} \\ c_{HH''}=b_{H'}}}A_{H''}+A_{H'}\right]=O \quad (H' \in \mathcal{A}_{x}^{c}\cap\mathcal{A}_{y}),  \label{eq:cond15p} 
\\
&\left[A_{H},\sum_{\substack{H''\in \mathcal{A}_{x}\cap\mathcal{A}_{y} \\ c_{HH'}=c_{HH''}}}A_{H''}\right]=O \quad\left(\begin{array}{c}H'\in \mathcal{A}_{x}\cap\mathcal{A}_{y} \text{ such that} \\ c_{HH'}\neq b_{K} \text{ for any }K\in\mathcal{A}_{x}^{c}\cap\mathcal{A}_{y}
\end{array}\right). \label{eq:cond16p}
\end{align}
Here $a_{H}(0)$ and $b_{H}(0)$ denote the constant terms of $a_{H}$ and $b_{H}$, respectively.
Since $A_{xy}=A_{yx}$, the above conditions are symmetric in $x$ and $y$.
In other words, the conditions obtained by the interchange $(x,a_{H})\leftrightarrow (y,b_{H})$ also hold. 
\end{proposition}
This proposition can be shown by a direct computation of \eqref{Eq_integrable}.
We remark that the conditions \eqref{eq:cond11p}, \eqref{eq:cond12p}, \eqref{eq:cond15p} and \eqref{eq:cond16p} were given by Haraoka \cite[Theorem 2.1]{Haraoka2012}.

\subsection{Definition}\label{subsec:deftwo}

	In this subsection, we generalize the middle Laplace transform for the linear Pfaffian systems \eqref{Eq_Pfaff}
	satisfying the following conditions: 
	there exists a variable $x_\ell=x$ such that
	\begin{align}
	&A_{x}\text{ and }A_{xz} \text{ are diagonalizable} \quad (\forall z \in x' ),\label{Eq_ass_diag}
	\\
	&\exists x_j \in x' ~\text{such that}~\mathcal{A}_{x} \cap \mathcal{A}_{x_j} \neq \emptyset
	~\Rightarrow~
	A_{x z}=O \quad (\forall z \in x').\label{Eq_ass_inf}
	\end{align}
Here, we used the notation \eqref{Eq_notation}, namely, we set $x'=\{x_1,\ldots,x_n\}\setminus\{x\}$.
By the commutativity conditions \eqref{eq:cond1p}, the matrices $ \{ A_x, A_{xz} \mid z \in x' \} $ are pairwise commutative.  
Hence, by the diagonalizability condition \eqref{Eq_ass_diag}, these matrices can be simultaneously diagonalized by a suitable constant gauge transformation.
Therefore, without loss of generality, we may assume that all the matrices $\{A_x, A_{xz} \mid z \in x' \}$ are diagonal.
Thus $S_{x}(\bm{x})=A_{x}+ \sum_{z\in x'}A_{xz}z$ is diagonal and its eigenvalues are linear polynomials in $x'$. 
After relabeling the distinct eigenvalues of $S_{x}(\bm{x})$, write them as $a_{\hat{H}_1},\ldots,a_{\hat{H}_{\hat{p}}}, a_{\hat{H}_{\hat{p}+1}}(x'),\ldots,a_{\hat{H}_{\hat{q}}}(x')$, where $a_{\hat{H}_1},\ldots,a_{\hat{H}_{\hat{p}}}\in\bbC$ and $a_{\hat{H}_{\hat{p}+1}}(x'),\ldots,a_{\hat{H}_{\hat{q}}}(x')$ are nonconstant linear polynomials in $x'$. 
Thus we may write $S_{x}(\bm{x})$ in the form
	\begin{equation}\label{Eq_A_yA}
	S_{x}(\bm{x})=
	\begin{pmatrix}
		a_{\hat{H}_1}I_{N_1} & & & & &  \\
		& \ddots& &  & &  \\
		& & a_{\hat{H}_{\hat{p}}} I_{N_{\hat{p}}} & &  &  \\
		& & & a_{\hat{H}_{\hat{p}+1}}(x') I_{N_{\hat{p}+1}} &  &  \\
		& & & & \ddots &  \\
		& & & & &  a_{\hat{H}_{\hat{q}}}(x') I_{N_{\hat{q}}}
	\end{pmatrix},
	\end{equation}
	where $N_1+N_2+\cdots+N_{\hat{q}}=N$. 
We note that
	\begin{equation}\label{Eq_A_yA_2}
	S_{x}(\bm{x})=A_{x}+\sum_{z\in x'}A_{xz}z
	=\sum_{i=1}^{\hat{q}}a_{\hat{H}_i}E_{N_i}.
	\end{equation}
For each $i=1,2,\ldots,\hat{q}$, we define the hyperplane $\hat{H}_{i}$ by the defining polynomial $f_{\hat{H}_{i}}:=x-a_{\hat{H}_i}$. 
Consequently, we have
\vspace{-0.8\baselineskip}
	\begin{equation}
	\{\bm{x}\in\bbC^{n}\mid \det(x-S_{x}(\bm{x}))=0\}=\bigcup_{i=1}^{\hat{q}}\hat{H}_{i}.
	\end{equation}
Next we set
	\begin{equation}\label{Eq_A}
	\mathcal{A}=\{H_1,\ldots,H_p,
	\ldots,H_{q},
	\ldots,H_r\}
	\end{equation}
with
	\begin{align*}
	\mathcal{A}_x=\{H_1,\ldots,H_q\}, 
	\quad 
	\mathcal{A}_x\setminus \bigcup_{x_j \in x'}\mathcal{A}_{x_j}
	=\{H_1,\ldots,H_p\}.
	\end{align*}
Namely, the defining polynomials $f_{H_1},\ldots,f_{H_{p}}$ depend only on $x$, $f_{H_{p+1}},\ldots,f_{H_{q}}$ depend on $x$ and the other variables, and $f_{H_{q+1}},\ldots,f_{H_r}$ do not depend on $x$.

As in the case of ordinary differential equations, it consists of the following three steps.
	\begin{enumerate}[\bf Step~1.]
	\item Extend the system \eqref{Eq_Pfaff} to the Birkhoff-Okubo normal form.
	\item Apply the Laplace transform.
	\item Project onto a suitable quotient space.
	\end{enumerate}
We shall explain the above steps in detail and in order.

\medskip
\noindent
\textbf{Step~1.} 
Let us consider extending the $x$-equation 
	\begin{equation}\label{Eq_u_x}
	\frac{\p u}{\p x}=\left(S_{x}(\bm{x})+
	\sum_{i=1}^q \frac{A_{H_i}}{x-a_{H_i}}\right)u
	\end{equation}
to the system of Birkhoff-Okubo normal form. 
Set
	\begin{equation}\label{Eq_U}
	U(\bm{x})={}^t\left(U_1,\ldots,U_q\right)
	={}^t \left(\frac{u(\bm{x})}{x-a_{H_1}},\frac{u(\bm{x})}{x-a_{H_2}},\ldots,\frac{u(\bm{x})}{x-a_{H_q}}\right).
	\end{equation}
Then, the vector $U(\bm{x})$ satisfies	
	\begin{equation}\label{Eq_U_xeq_BO}
	(x-T_{x}(\bm{x}))\frac{\p U}{\p x}
	=\left(A-I+S_{x}(\bm{x})^{\oplus q}(x-T_{x}(\bm{x}))\right)U,
	\end{equation}
where 	\begin{align}
		A&=\begin{pmatrix}
		A_{H_1} & A_{H_2} &\cdots & A_{H_q} \\
		A_{H_1} & A_{H_2} &\cdots & A_{H_q} \\
		\vdots & \vdots &  &\vdots\\
		A_{H_1}& A_{H_2}&\cdots &A_{H_q} 
		\end{pmatrix},
\\
	\label{Eq_Stil}
	S_x(\bm{x})^{\oplus q}&=
	\begin{pmatrix}
		S_x(\bm{x}) & & & \\
		& S_x(\bm{x})& & \\
		 & & \ddots & \\
		 & & & S_x(\bm{x})
	\end{pmatrix}
	=A^{\oplus q}_x+\sum_{z \in x'}A_{xz}^{\oplus q} z,
\\
	\label{Eq_Ttil}
	T_{x}(\bm{x})&=\begin{pmatrix}
	a_{H_1}I_{N} & & & & &  \\
	& \ddots& &  & &  \\
	& & a_{H_p} I_{N} & &  &  \\
	& & & a_{H_{p+1}}(x') I_{N} &  &  \\
	& & & & \ddots &  \\
	& & & & &  a_{H_q}(x') I_{N}
	\end{pmatrix}.
	\end{align}
This equation is the same as \eqref{eq:odeBO}.
By multiplying $(x-\widetilde{T}_{x}(\bm{x}))^{-1}$ from the left, 
we obtain
	\begin{equation}\label{Eq_U_x_sch}
	\frac{\p U}{\p x}=
	\left(S_x(\bm{x})^{\oplus q}
	+\sum_{i=1}^q \frac{\widetilde{A}_{H_i}}{x-a_{H_i}} \right)U,
	\end{equation}
where
	\begin{equation}\label{Eq_G_H_i}
	\begin{aligned}
	\widetilde{A}_{H_i}&= \sum_{j=1}^q E_{ij}^{(q)}\otimes (A_{H_j}-\delta_{ij}I_N) \\
	&= \begin{pmatrix}
	O & \cdots & \cdots  & \cdots &O \\
	\vdots & \cdots & \cdots & \cdots &  \vdots \\
	A_{H_1} & \cdots & A_{H_i}-I_N & \cdots & A_{H_q}  \\
	\vdots & \cdots & \cdots & \cdots &  \vdots \\
	O & \cdots & \cdots  & \cdots &O \\	
	\end{pmatrix}(\,i
	\end{aligned}
	\end{equation}
for $1\le i\le q$. 
Here $E_{ij}^{(q)}$ denotes the $q\times q$ matrix with the only non-zero entry $1$ at $(i,j)$-th position.

Next, we proceed to derive $y$-equations ($y\in x'$) for $U(\bm{x})$.
For the pair $(x,y)$, we use the notation \eqref{Eq_notation}. 
Namely, we write $(x_1,\ldots,x_n)=(x,y')=(x,y,x'')$.
Then, the $y$-equation of \eqref{Eq_Pfaff} is given by
	\begin{equation}\label{Eq_u_y}
	\frac{\p u}{\p y}=\left(S_y(\bm{x})+
	\sum_{\substack{p+1\le i \le r \\ H_i \in \mathcal{A}_y}}
	\frac{A_{H_i}}{y-b_{H_i}(x,x'')}\right)u.
	\end{equation}
First, we consider $\p U_i/\p y$ for $i=1,\ldots,p$.
Since the polynomials $a_{H_i}$ ($1\le i\le p$) is independent on $y$, we have
	\begin{align*}
	\frac{\p U_i}{\p y}&=\frac{\p }{\p y}\left(\frac{u(\bm{x})}{x-a_{H_i}}\right)
	=\frac{1}{x-a_{H_i}}\frac{\p u}{\p y} \\
	&=\left(S_y(\bm{x})+
	\sum_{\substack{p+1\le j \le r \\ H_j \in \mathcal{A}_y}}
	\frac{A_{H_j}}{y-b_{H_j}(x,x'')}\right)\frac{u(\bm{x})}{x-a_{H_i}}.
	\end{align*}
We focus on the latter term.
The polynomials $b_{H_j}$ ($q+1\le j \le r$) are independent on $x$.
On the other hand, for $j=p+1,\ldots,q$, we have
	\begin{align*}
\frac{1}{y-b_{H_j}}\frac{1}{x-a_{H_i}}
	&=\frac{(f_{H_j})_y}{f_{H_j}}\frac{1}{x-a_{H_i}} \\
	&=\frac{(f_{H_j})_y}{(f_{H_j})_x}\frac{1}{x-a_{H_j}}\frac{1}{x-a_{H_i}} \\
	&=\frac{(f_{H_j})_y}{(f_{H_j})_x}\frac{1}{a_{H_j}-a_{H_i}}\left(\frac{1}{x-a_{H_j}}-\frac{1}{x-a_{H_i}}\right) \\
	&=\frac{(f_{H_j})_y}{(f_{H_j})_x}\frac{1}{(a_{H_j}-a_{H_i})_y}\frac{1}{y-c_{H_iH_j}}\left(\frac{1}{x-a_{H_j}}-\frac{1}{x-a_{H_i}}\right) \\
	&=\frac{1}{y-c_{H_i H_j}}\left(\frac{1}{x-a_{H_i}}-\frac{1}{x-a_{H_j}}\right).
	\end{align*}
Therefore, we obtain
	\begin{equation}\label{Eq_Ui_y_1}
	\frac{\p U_i}{\p y}
	=\left(S_y(\bm{x})
	+
	\sum_{\substack{q+1\le j \le r \\ H_j \in \mathcal{A}_y}}
	\frac{A_{H_j}}{y-b_{H_j}(x'')}\right)U_i
	+\sum_{\substack{p+1\le j\le q \\ H_{j} \in\mathcal{A}_y} }
	\frac{A_{H_j}}{y-c_{H_i H_j}}(U_i-U_j).
	\end{equation}

Next, we consider $\p U_i/\p y$ for $i=p+1,\ldots,q$. 
For $i$ such that $a_{H_i}$ is independent on $y$, i.e., $H_i \in \mathcal{A}_x\cap\mathcal{A}_y^c$, we can calculate $\p U_i/\p y$ in the same way as above and obtain \eqref{Eq_Ui_y_1}.
On the other hand, for $i$ such that $a_{H_i}$ depends on $y$, i.e., $H_i \in \mathcal{A}_x\cap\mathcal{A}_y$, we have
	\[
	\frac{\p U_i}{\p y}=\frac{\p}{\p y}\left(\frac{u(\bm{x})}{x-a_{H_i}}\right)
	=\frac{1}{x-a_{H_i}}\frac{\p u}{\p y}
	-\frac{(x-a_{H_i})_y}{(x-a_{H_i})^2}u(\bm{x}).
	\]
By using the $x$-equation for $U_i$
	\[
	\frac{\p U_i}{\p x}=\frac{\p}{\p x}{\left(\frac{u(\bm{x})}{x-a_{H_i}}\right)
	=\frac{1}{x-a_{H_i}}\frac{\p u}{\p x}-\frac{1}{(x-a_{H_i})^2}}u(\bm{x}),
	\]	
and $(f_{H_i})_y=(f_{H_i})_x(x-a_{H_i})_y$, we obtain
	\begin{equation} \label{Eq_Ui_y2}
	\begin{aligned}
	\frac{\p U_i}{\p y}&=\frac{1}{x-a_{H_i}}\frac{\p u}{\p y}-\frac{(f_{H_i})_y}{(f_{H_i})_x}\left(\frac{1}{x-a_{H_i}}\frac{\p u}{\p x}-\frac{\p U_i}{\p x}\right)
\\
	&=\frac{1}{x-a_{H_i}}\left(\frac{\p u}{\p y}-\frac{(f_{H_i})_y}{(f_{H_i})_x}\frac{\p u}{\p x}\right)+
	\frac{(f_{H_i})_y}{(f_{H_i})_x}\frac{\p U_i}{\p x}.
	\end{aligned}
	\end{equation}
By using the $x$-equation \eqref{Eq_u_x} and the $y$-equation \eqref{Eq_u_y} for $u(\bm{x})$, we obtain
	\begin{align*}
	&\frac{\p u}{\p y}-\frac{(f_{H_i})_y}{(f_{H_i})_x}\frac{\p u}{\p x} \\
	&=\left[
	S_y(\bm{x})+
	\sum_{\substack{p+1\le j \le r \\ H_j \in \mathcal{A}_y}}
	\frac{A_{H_j}}{y-b_{H_j}(x,x'')}
	-\frac{(f_{H_i})_y}{(f_{H_i})_x}
	\left(
	A_{x}+
	\sum_{1\le j\le q} \frac{A_{H_j}}{x-a_{H_j}}
	\right)
	\right]u(\bm{x}) \\
	&=\left[
	S_y(\bm{x})+
	\sum_{\substack{p+1\le j \le r \\ H_j \in \mathcal{A}_y,~ j\neq i}}
	\frac{A_{H_j}}{y-b_{H_j}(x,x'')}
	-\frac{(f_{H_i})_y}{(f_{H_i})_x}
	\left(
	A_{x}+
	\sum_{\substack{1\le j\le q \\ j \neq i}} \frac{A_{H_j}}{x-a_{H_j}}
	\right)
	\right]u(\bm{x})\\
	&\quad 
	+\left[\frac{A_{H_i}}{y-b_{H_i}(x,x'')}-\frac{(f_{H_i})_y}{(f_{H_i})_x}\frac{A_{H_i}}{x-a_{H_i}}\right]u(\bm{x}).
	\end{align*}
Since the latter term vanishes, we have
	\[
	\frac{\p u}{\p y}-\frac{(f_{H_i})_y}{(f_{H_i})_x}\frac{\p u}{\p x}
	=
	\left[
	S_y(\bm{x})+
	\sum_{\substack{p+1\le j \le r \\ H_j \in \mathcal{A}_y,~ j\neq i}}
	\frac{A_{H_j}}{y-b_{H_j}}
	-\frac{(f_{H_i})_y}{(f_{H_i})_x}
	\left(
	A_{x}+
	\sum_{\substack{1\le j\le q \\ j \neq i}} \frac{A_{H_j}}{x-a_{H_j}}
	\right)
	\right]u(\bm{x}).
	\]
Besides, by using the $x$-equation for $U_i$ which is obtained from \eqref{Eq_U_x_sch}, we have
	\begin{equation}\label{Eq_BO_xPoly}
	\begin{aligned}
	\frac{(f_{H_i})_y}{(f_{H_i})_x}\frac{\p U_i}{\p x}
	&=\frac{(f_{H_i})_y}{(f_{H_i})_x} \left[A_{x}U_i+\frac{1}{x-a_{H_i}}\sum_{1\le j\le q}(A_{H_j}-\delta_{ij}I_N)U_j\right] \\
	& =\frac{(f_{H_i})_y}{(f_{H_i})_x} A_{x}U_i+\frac{1}{y-b_{H_i}}\sum_{1\le j\le q}(A_{H_j}-\delta_{ij}I_N)U_j.
	\end{aligned}
	\end{equation}
Then, the equation \eqref{Eq_Ui_y2} turns into 
	\begin{equation}\label{Eq_Ui_y3}
	\begin{aligned}
	\frac{\p U_i}{\p y}& =\frac{1}{x-a_{H_i}}
	\left[
	S_y(\bm{x})+
	\sum_{\substack{p+1\le j \le r \\ H_j \in \mathcal{A}_y,~ j\neq i}}
	\frac{A_{H_j}}{y-b_{H_j}}
	-\frac{(f_{H_i})_y}{(f_{H_i})_x}
	\left(
	A_{x}+
	\sum_{\substack{1\le j\le q \\ j \neq i}} \frac{A_{H_j}}{x-a_{H_j}}
	\right)
	\right]u(\bm{x}) \\
	&\qquad + \frac{(f_{H_i})_y}{(f_{H_i})_x} A_{x}U_i+\frac{1}{y-b_{H_i}}\sum_{1\le j\le q}(A_{H_j}-\delta_{ij}I_N)U_j  \\
	&=\left(S_y(\bm{x})+
	\sum_{\substack{q+1\le j \le r \\ H_j \in \mathcal{A}_y}}
	\frac{A_{H_j}}{y-b_{H_j}}\right)U_i
	\\
	&\qquad +\frac{1}{x-a_{H_i}}\left(
	\sum_{\substack{p+1 \le j \le q \\ H_j \in \mathcal{A}_y,~ j\neq i}}\frac{A_{H_j}}{y-b_{H_j}}
	-\frac{(f_{H_i})_y}{(f_{H_i})_x}\sum_{\substack{1\le j\le q \\ j \neq i}} \frac{A_{H_j}}{x-a_{H_j}}
	\right)u(\bm{x})
	\\
	&\qquad +\frac{1}{y-b_{H_i}}\sum_{1\le j\le q}(A_{H_j}-\delta_{ij}I_N)U_j.
	\end{aligned}
	\end{equation}
We focus on the second term of the right-hand side.
Since
	\begin{align} 
	&\frac{1}{x-a_{H_i}}\frac{1}{y-b_{H_j}}
	=\frac{(f_{H_j})_y}{(f_{H_j})_x}\frac{1}{a_{H_i}-a_{H_j}}\left(\frac{1}{x-a_{H_i}}-\frac{1}{x-a_{H_j}}\right) \quad (p+1 \le j \le q,~ j\neq i), \\
	&\frac{1}{x-a_{H_i}}\frac{1}{x-a_{H_j}}=
	\begin{cases} 
	-\dfrac{(f_{H_i})_x}{(f_{H_i})_y}\dfrac{1}{y-c_{H_i H_j}}\left(\dfrac{1}{x-a_{H_i}}-\dfrac{1}{x-a_{H_j}}\right) &(1\le j \le p), \\[11pt]
	\dfrac{1}{a_{H_i}-a_{H_j}}\left(\dfrac{1}{x-a_{H_i}}-\dfrac{1}{x-a_{H_j}}\right) & (p+1 \le j \le q,~j \neq i),
	\end{cases} 
	\end{align}
we have 
	\begin{align*}
	\frac{1}{x-a_{H_i}}&\left(\sum_{\substack{p+1\le j \le q \\ H_j \in \mathcal{A}_y,~j \neq i}}
	\frac{A_{H_j}}{y-b_{H_j}}-\frac{(f_{H_i})_y}{(f_{H_i})_x} \sum_{\substack{1\le j \le q \\ j \neq i}}\frac{A_{H_j}}{x-a_{H_j}}\right)\\
	&=\frac{1}{x-a_{H_i}}\sum_{\substack{p+1\le j \le q \\ H_j \in \mathcal{A}_y,~j \neq i}}\frac{A_{H_j}}{y-b_{H_j}}-\frac{(f_{H_i})_y}{(f_{H_i})_x}\frac{1}{x-a_{H_i}}\left(\sum_{1 \le j \le p}\frac{A_{H_j}}{x-a_{H_j}}+\sum_{\substack{p+1 \le j \le q \\ j \neq i}} \frac{A_{H_j}}{x-a_{H_j}}\right) \\
	&=\sum_{1\le j \le p}\frac{A_{H_j}}{y-c_{H_i H_j}}\left(\frac{1}{x-a_{H_i}}-\frac{1}{x-a_{H_j}}\right)\\
	&\quad +\sum_{\substack{p+1 \le j \le q \\H_j \in \mathcal{A}_y,~ j \neq i}}A_{H_j}\left(\frac{(f_{H_j})_y}{(f_{H_j})_x}-\frac{(f_{H_i})_y}{(f_{H_i})_x}\right)\frac{1}{a_{H_i}-a_{H_j}}\left(\frac{1}{x-a_{H_i}}-\frac{1}{x-a_{H_j}}\right). 
	\end{align*}
Here, since $H_i,H_j \in \mathcal{A}_x\cap \mathcal{A}_y$, we have
	\[
	\frac{(f_{H_j})_y}{(f_{H_j})_x}-\frac{(f_{H_i})_y}{(f_{H_i})_x}=
	(x-a_{H_j})_y-(x-a_{H_i})_y=(a_{H_i}-a_{H_j})_y.
	\]
Then, by definition, it holds that 	
	\[
	(a_{H_i}-a_{H_j})_y=
	\begin{cases}
	\ds\frac{a_{H_i}-a_{H_j}}{y-c_{H_iH_j}} & (H_j \in \mathcal{C}_{H_i,y}), \\
	0 & (H_j \notin \mathcal{C}_{H_i,y}).
	\end{cases}
	\]
Therefore, we have
	\begin{equation}\label{Eq_second}
	\begin{aligned}
	\frac{1}{x-a_{H_i}}&\left(\sum_{\substack{p+1\le j \le q \\ H_j \in \mathcal{A}_y,~j \neq i}}
	\frac{A_{H_j}}{y-b_{H_j}}-\frac{(f_{H_i})_y}{(f_{H_i})_x} \sum_{\substack{1\le j \le q \\ j \neq i}}\frac{A_{H_j}}{x-a_{H_j}}\right)\\
	&=\sum_{1\le j \le p}\frac{A_{H_j}}{y-c_{H_i H_j}}\left(\frac{1}{x-a_{H_i}}-\frac{1}{x-a_{H_j}}\right)\\
	&\qquad +\sum_{\substack{p+1 \le j \le q \\H_j \in \mathcal{C}_{H_i,y}
	}}
	\frac{A_{H_j}}{y-c_{H_iH_j}}\left(\frac{1}{x-a_{H_i}}-\frac{1}{x-a_{H_j}}\right)
	\\
	&=\left(\sum_{1\le j \le p}+\sum_{\substack{p+1 \le j \le q \\H_j \in \mathcal{C}_{H_i,y}}}\right)
	\frac{A_{H_j}}{y-c_{H_i H_j}}\left(\frac{1}{x-a_{H_i}}-\frac{1}{x-a_{H_j}}\right).
	\end{aligned}
	\end{equation}
Eventually, we see that the equation \eqref{Eq_Ui_y3} can be expressed as 
	\begin{align*}
	\frac{\p U_i}{\p y}&=\left(S_y(\bm{x})+
	\sum_{\substack{q+1\le j \le r \\ H_j \in \mathcal{A}_y}}
	\frac{A_{H_j}}{y-b_{H_j}}\right)U_i 
	+\sum_{\substack{1 \le j \le q \\H_j \in\mathcal{C}_{H_i,y}}}
	\frac{A_{H_j}}{y-c_{H_i H_j}}\left(U_i-U_j\right)
	\\
	&\qquad
	+\frac{1}{y-b_{H_i}}\sum_{1\le j\le q}(A_{H_j}-\delta_{ij}I_N)U_j.
	\end{align*}
Summarizing the above results, we obtain the following.

\begin{proposition}
\label{Prop_Okubo}
Let $u(\bm{x})$ be a solution of the system \eqref{Eq_Pfaff} satisfying \eqref{Eq_ass_inf}.
Then, for each $y \in x'$, the vector 
$U(\bm{x})={}^t(U_1,U_2,\ldots,U_q)$
defined by \eqref{Eq_U} satisfies the following:
	\begin{itemize}
	\item For $U_i$ such that $H_i \in \mathcal{A}_x\cap\mathcal{A}_y^c$
	(namely, $1\le i \le p$ or $p+1\le i \le q$ with $H_i \notin \mathcal{A}_y$),
	\begin{equation}\label{Eq_Prop_BO1}
	\frac{\p U_i}{\p y}
	=\left(S_y(\bm{x})
	+
	\sum_{\substack{q+1\le j \le r \\ H_j \in \mathcal{A}_y}}
	\frac{A_{H_j}}{y-b_{H_j}}\right)U_i
	+\sum_{\substack{p+1\le j\le q \\ H_{j} \in\mathcal{A}_y} }
	\frac{A_{H_j}}{y-c_{H_i H_j}}(U_i-U_j).
	\end{equation}
	
	\item For $U_i$ such that $H_i \in \mathcal{A}_x\cap\mathcal{A}_y$
	(namely, $p+1\le i\le q$ with $H_i \in \mathcal{A}_y$),
	\begin{equation}\label{Eq_Prop_BO2}
	\begin{aligned}
	\frac{\p U_i}{\p y}&=\left(S_y(\bm{x})+
	\sum_{\substack{q+1\le j \le r \\ H_j \in \mathcal{A}_y}}
	\frac{A_{H_j}}{y-b_{H_j}}\right)U_i 
	+\sum_{\substack{1 \le j \le q \\H_j \in\mathcal{C}_{H_i,y}}}
	\frac{A_{H_j}}{y-c_{H_i H_j}}\left(U_i-U_j\right)
	\\
	&\quad
	+\frac{1}{y-b_{H_i}}\sum_{1\le j\le q}(A_{H_j}-\delta_{ij}I_N)U_j.
	\end{aligned}
	\end{equation}
	\end{itemize}
\end{proposition}
\begin{remark}
Proposition \ref{Prop_Okubo} itself holds even if the assumption \eqref{Eq_ass_diag} is not satisfied. 
\end{remark}

Combining \eqref{Eq_Prop_BO1} and \eqref{Eq_Prop_BO2}, we see that $U(\bm{x})$ satisfies the system of the form
	\begin{equation}\label{Eq_U_y}
	\frac{\p U}{\p y}=
	\left(
	S_y(\bm{x})^{\oplus q}
	+
	\sum_{\substack{p+1 \le i \le r \\ H_i \in \mathcal{A}_y}} \frac{\widetilde{A}_{H_i}}{y-b_{H_i}}
	+
	\sum_{
		\substack{1\le i<j\le q \\ H_j \in  \mathcal{C}_{H_i,y} \\
		c_{H_i H_j} \neq b_{H} \text{ for any $H \in \mathcal{A}_y$}}
		}
		\frac{\widetilde{A}_{H_{ij}}}{y-c_{H_i H_j}}
	\right)U,
	\end{equation}
where 
	\begin{equation}\label{Eq_GyGi}
	\begin{aligned}
	S_y(\bm{x})^{\oplus q}=\begin{pmatrix}
	S_y(\bm{x}) & & &  \\
	& S_y(\bm{x}) & & \\
	& & \ddots &  \\
	& & &  S_y(\bm{x})
	\end{pmatrix}
	=A_y^{\oplus q}+\sum_{z\in(x,x'')}A_{yz}^{\oplus q}z
	\end{aligned}
	\end{equation}
and $\widetilde{A}_{H_i}~(p+1\le i \le q,\,H_i\in\mathcal{A}_y)$ are given in \eqref{Eq_U_x_sch}. 
Then, we see that the systems for $U(\bm{x})$ given by \eqref{Eq_U_x_sch} and \eqref{Eq_U_y} for all $y\in x'$ forms the Pfaffian system of the form
	\begin{equation}\label{Eq_Pfaff_U}
	dU=BO^{x}(\Omega)U, \quad 
	BO^{x}(\Omega):=
	\sum_{i=1}^n S_{x_i}(\bm{x})^{\oplus q}\,dx_i
	+\sum_{H \in \mathcal{\widetilde{A}}}\widetilde{A}_H \, d\log f_H.
	\end{equation}
Here
\begin{equation}\label{def:tilA}
\widetilde{\mathcal{A}}:=\mathcal{A}\cup\{H_{ij}\mid\text{$1\le i<j\le q$ satisfies $H_j \in \mathcal{C}_{H_i,y}$ for some $y \in x'$}\}
\end{equation}
and $H_{ij}$ is the hyperplane defined by the polynomial $f_{H_{ij}}=y-c_{H_iH_j}$. 
We note that, by considering the extension, the singular locus increases in general.

\medskip
\noindent
\textbf{Step 2.} 
We consider the Laplace transform in $x$-direction
	\begin{equation}\label{Eq_op_FL}
		L^x:\left\{
	\begin{array}{ccc}
	\dfrac{\partial }{\partial x} &\,\mapsto\, & x, \\
	x &\,\mapsto\,& -\dfrac{\partial }{\partial x}
	\end{array}\right.
	\end{equation}	
for the Pfaffian system \eqref{Eq_Pfaff_U}, which corresponds to the integral transform 
	\[
	U(\bm{x})~\mapsto~V(\bm{x})=\int_{\Delta} U(t,x')e^{-xt}\,dt
	\]
with a suitable path of integration $\Delta$.
We set $V:=L^{x}(U)$, namely,
\begin{equation}\label{Eq_V_def}
	V={}^t(V_1,V_2,\ldots,V_q), \quad V_i:=L^x(U_i).
	\end{equation}
Then, as in Subsection \ref{subsec:onedef}, the $x$-equation \eqref{Eq_U_xeq_BO} of the system \eqref{Eq_Pfaff_U} is transformed into
	\begin{equation}\label{Eq_xeq_Lap_BO}
	\left(
	x-S_{x}(\bm{x})^{\oplus q}\right)\frac{\partial V}{\partial x}
	=
	-\left(A+\left(x-S_{x}(\bm{x})^{\oplus q}\right)T_{x}(\bm{x})
	\right)V
	\end{equation}
We consider expressing the equation \eqref{Eq_xeq_Lap_BO} as Schlesinger form.
Set 	\begin{equation}\label{eq:Bdef}
	\mathcal{\mathcal{B}}:=\{\hat{H}_1,\ldots,\hat{H}_{\hat{q}}\}\cup(\widetilde{\mathcal{A}}\setminus\mathcal{A}_x).
	\end{equation}
We note that 
	\begin{equation}
		\label{Eq_Bx}
	\mathcal{B}_x=\{\hat{H}_1,\ldots,\hat{H}_{\hat{q}}\}.
	\end{equation}
Then, by multiplying $\left(x-S_x(\bm{x})^{\oplus q})\right)^{-1}$ from the left, the equation \eqref{Eq_xeq_Lap_BO} is transformed as 
	\begin{equation}\label{Eq_xeq_Lap_Sch2}
	\frac{\p V}{\p x}=
	\left(-T_{x}(\bm{x})+\sum_{\hat{H}\in\mathcal{B}_x}\frac{B_{\hat{H}}}{x-a_{\hat{H}}}\right)V
	=
	\left(-T_{x}(\bm{x})+\sum_{j=1}^{\hat{q}}\frac{B_{\hat{H}_j}}{x-a_{\hat{H}_j}}\right)V,
	\end{equation}
where
	\begin{equation}\label{Eq_hat_G}
	B_{\hat{H}_j}=-\begin{pmatrix}
	E_{N_j}A_{H_1} & E_{N_j}A_{H_2}& \cdots &E_{N_j}A_{H_q} \\
	E_{N_j}A_{H_1} & E_{N_j}A_{H_2}& \cdots & E_{N_j}A_{H_q} \\
	\vdots & \vdots &  &\vdots\\
	E_{N_j}A_{H_1}& E_{N_j}A_{H_2}&\cdots &E_{N_j}A_{H_q} 
	\end{pmatrix}.
	\end{equation}
This equation is the same as \eqref{eq:odeL2}.
We proceed to derive the differential equations of $V(\bm{x})$ in the other directions. 
Let us take any $y\in x'$ and shall derive the equation for $\p V_i/\p y$.
First, we consider the case $H_i \in \mathcal{A}_x\cap\mathcal{A}_y^c$.
By applying \eqref{Eq_op_FL} for the equation \eqref{Eq_Prop_BO1}, we have
	\[
	\frac{\p V_i}{\p y}
	=\left(A_{y}+\sum_{z\in x''} A_{yz}z
	+
	\sum_{\substack{q+1\le j \le r \\ H_j \in \mathcal{A}_y}}
	\frac{A_{H_j}}{y-b_{H_j}}\right)V_i-A_{yx}\frac{\p V_i}{\p x}
	+\sum_{\substack{p+1\le j\le q \\ H_{j} \in\mathcal{A}_y} }
	\frac{A_{H_j}}{y-c_{H_i H_j}}(V_i-V_j).
	\]
Here we note that, the polynomials $b_{H_j}$ and $c_{H_iH_j}$ appearing in \eqref{Eq_Prop_BO1} do not depend on $x$.
Since it holds that 
	\[
	\frac{\p V_i}{\p x}=-a_{H_i}V_i-\sum_{k=1}^{\hat{q}}\frac{E_{N_k}}{x-a_{\hat{H}_k}}\sum_{j=1}^qA_{H_j
} V_j
	\]
from the equation \eqref{Eq_xeq_Lap_Sch2}, we obtain
	\begin{align*}
	\frac{\p V_i}{\p y}=&
		\left(A_{y}+
		\sum_{z\in x''}A_{yz} z+a_{H_i}A_{yx}
	+
	\sum_{\substack{q+1\le j \le r \\ H_j \in \mathcal{A}_y}}
	\frac{A_{H_j}}{y-b_{H_j}}
		\right)V_i
	+\sum_{\substack{p+1\le j\le q \\ H_{j} \in\mathcal{A}_y} }
	\frac{A_{H_j}}{y-c_{H_i H_j}}(V_i-V_j)
	 \\
	&\quad +\sum_{k=1}^{\hat{q}}\frac{A_{yx}E_{N_k}}{x-a_{\hat{H}_k}}\sum_{j=1}^q A_{H_j}V_j.
	\end{align*}
Let us focus on the last term.
From \eqref{Eq_A_yA_2}, we have
	\begin{align*}
	A_{yx}E_{N_k}=
	\begin{cases} O &(1\le k \le \hat{p}), \\
	(a_{\hat{H}_k})_y E_{N_k} & 
	(\hat{p}+1\le k \le \hat{q}).
	\end{cases}
	\end{align*}
Moreover, for $\hat{p}+1\le k\leq\hat{q}$ with $(a_{\hat{H}_k})_y\neq 0$, i.e., $\hat{H}_k\in\mathcal{B}_y$, it holds that
	\[
	\frac{A_{yx}E_{N_k}}{x-a_{\hat{H}_k}}=
	\frac{(a_{\hat{H}_k})_yE_{N_k}}{x-a_{\hat{H}_k}}
	=\frac{(f_{\hat{H}_k})_x(a_{\hat{H}_k})_yE_{N_k}}{(f_{\hat{H}_k})_y(y-b_{\hat{H}_k})}
	=-\frac{E_{N_k}}{y-b_{\hat{H}_k}}.
	\]
Therefore, for $H_i \in \mathcal{A}_x\cap\mathcal{A}_y^c$ we have
	\begin{equation}\label{Eq_Vi_1}
	\begin{aligned}
	\frac{\p V_i}{\p y}=&
		\left(A_{y}+
		\sum_{z\in x''}A_{yz} z+a_{H_i}A_{yx}
	+
	\sum_{\substack{q+1\le j \le r \\ H_j \in \mathcal{A}_y}}
	\frac{A_{H_j}}{y-b_{H_j}}
		\right)V_i
	+\sum_{\substack{p+1\le j\le q \\ H_{j} \in\mathcal{A}_y} }
	\frac{A_{H_j}}{y-c_{H_i H_j}}(V_i-V_j)
	 \\
	&\quad -\sum_{\substack{\hat{p}+1\le k\le \hat{q} \\ \hat{H}_k \in \mathcal{B}_y}}
	\frac{1}{y-b_{\hat{H}_k}}\sum_{j=1}^q E_{N_k}A_{H_j}V_j.
	\end{aligned}
	\end{equation}
	Thanks to the assumption \eqref{Eq_ass_inf}, 
	if $A_{xy}\neq O$, then $\mathcal{A}_x\cap\mathcal{A}_{x_j}=\emptyset$ holds for any $x_j \in x'$. 
	This says that $a_{H_i}\in\mathbb{C}$.

Next, we shall derive the equation for $\p V_i/\p y$ with $H_i \in \mathcal{A}_x\cap\mathcal{A}_y$.
In this case, we can assume $A_{xz}=O$ for any $z\in x'$ thanks to the assumption \eqref{Eq_ass_inf}.
	By using \eqref{Eq_BO_xPoly}, the equation \eqref{Eq_Prop_BO2}  for $\p U_i/\p y$ can be expressed as
	\begin{equation}\label{Eq_BO_y_Ui_2}
	\begin{aligned}
	\frac{\p U_i}{\p y}&=\left(A_{y}+
	\sum_{z\in x''} A_{yz}z+
	\sum_{\substack{q+1\le j \le r \\ H_j \in \mathcal{A}_y}}
	\frac{A_{H_j}}{y-b_{H_j}}\right)U_i 
	\\
	&\quad 
	+\sum_{\substack{1 \le j \le q \\H_j \in\mathcal{C}_{H_i,y}}}
	\frac{A_{H_j}}{y-c_{H_i H_j}}\left(U_i-U_j\right)
	+\frac{(f_{H_i})_y}{(f_{H_i})_x}\left(\frac{\p U_i}{\p x}- A_{x}U_i\right).
	\end{aligned}	
	\end{equation}
Since $H_i \in \mathcal{A}_x\cap\mathcal{A}_y$, it holds that
	\[
	\frac{(f_{H_i})_y}{(f_{H_i})_x}
	=-(a_{H_i})_y.
	\]
Hence we have
	\begin{equation}\label{Eq_BO_y_Ui_3}
	\begin{aligned}
	\frac{\p U_i}{\p y}&=\left(A_{y}+
	\sum_{z\in x''} A_{yz}z+
	\sum_{\substack{q+1\le j \le r \\ H_j \in \mathcal{A}_y}}
	\frac{A_{H_j}}{y-b_{H_j}}\right)U_i 
	\\
	&\quad 
	+\sum_{\substack{1 \le j \le q \\H_j \in\mathcal{C}_{H_i,y}}}
	\frac{A_{H_j}}{y-c_{H_i H_j}}\left(U_i-U_j\right)
	-(a_{H_i})_y\left(\frac{\p U_i}{\p x}- A_{x}U_i\right).
	\end{aligned}	
	\end{equation}
By applying the Laplace transform \eqref{Eq_op_FL} for this equation, we have 
	\[
	\begin{aligned}
	\frac{\p V_i}{\p y}&=\left(A_{y}+
	\sum_{z\in x''} A_{yz}z+
	\sum_{\substack{q+1\le j \le r \\ H_j \in \mathcal{A}_y}}
	\frac{A_{H_j}}{y-b_{H_j}}\right)V_i 
	\\
	&\quad 
	+\sum_{\substack{1 \le j \le q \\H_j \in\mathcal{C}_{H_i,y}}}
	\frac{A_{H_j}}{y-c_{H_i H_j}}\left(V_i-V_j\right)
	-(a_{H_i})_y\left(xV_i- A_{x}V_i\right).
	\end{aligned}	
	\]
Combining this and the equation \eqref{Eq_Vi_1}, we obtain the following.

\begin{proposition}
\label{Prop_Laplace}
Suppose that the system \eqref{Eq_Pfaff} satisfies \eqref{Eq_ass_diag} and \eqref{Eq_ass_inf}.
By the transformation \eqref{Eq_op_FL}, the $y$-equation ($y\in x'$) of \eqref{Eq_Pfaff_U} is transformed as follows:
	\begin{itemize}
	\item For $V_i$ such that $H_i \in \mathcal{A}_x\cap\mathcal{A}_y^c$
	(namely, $1\le i \le p$ or $p+1\le i \le q$ with $H_i \notin \mathcal{A}_y$),
	\begin{equation}\label{Eq_Prop_Vi_1}
	\begin{aligned}
	\frac{\p V_i}{\p y}=&
		\left(A_{y}+
		\sum_{z\in x''}A_{yz} z+a_{H_i}A_{yx}
	+
	\sum_{\substack{q+1\le j \le r \\ H_j \in \mathcal{A}_y}}
	\frac{A_{H_j}}{y-b_{H_j}}
		\right)V_i
	\\&\quad +\sum_{\substack{p+1\le j\le q \\ H_{j} \in\mathcal{A}_y} }
	\frac{A_{H_j}}{y-c_{H_i H_j}}(V_i-V_j)
	-\sum_{\substack{\hat{p}+1\le k\le \hat{q} \\ \hat{H}_k \in \mathcal{B}_y}}
	\frac{1}{y-b_{\hat{H}_k}}\sum_{j=1}^q E_{N_k}A_{H_j}V_j.
	\end{aligned}
	\end{equation}

	\item For $V_i$ such that $H_i \in \mathcal{A}_x\cap\mathcal{A}_y$
	(namely, $p+1\le i\le q$ with $H_i \in \mathcal{A}_y$),
	\begin{equation}\label{Eq_Prop_Vi_2}
	\begin{aligned}
	\frac{\p V_i}{\p y}
	&=	\left(A_{y}+
	(a_{H_i})_y A_{x}-(a_{H_i})_y x+\sum_{z\in x''} A_{yz}z+
	\sum_{\substack{q+1\le j \le r \\ H_j \in \mathcal{A}_y}}
	\frac{A_{H_j}}{y-b_{H_j}}\right)V_i 
	\\ 
	&\quad +\sum_{\substack{1 \le j \le q \\H_j \in\mathcal{C}_{H_i,y}}}
	\frac{A_{H_j}}{y-c_{H_i H_j}}\left(V_i-V_j\right).
	\end{aligned}
	\end{equation}
	\end{itemize}
\end{proposition}
Combining \eqref{Eq_xeq_Lap_Sch2}, \eqref{Eq_Prop_Vi_1} and \eqref{Eq_Prop_Vi_2} for all $y \in x'$, we obtain the Pfaffian system 
	\begin{equation}\label{Eq_Pfaff_V}
	dV=\mathcal{L}^x(\Omega)V, \quad 
	\mathcal{L}^x(\Omega) :=
	-\sum_{i=1}^n T_{x_i}(\bm{x})\,dx_i
	+\sum_{H \in \mathcal{B}}B_H \, d\log f_H
	\end{equation}
where $-T_{x_i}(\bm{x})$ is of the form
	\[
	-T_{x_i}(\bm{x})=B_{x_i}+\sum_{\substack{j=1 \\ j\neq i}}^nB_{x_i x_j}x_j,
	\quad 
	(B_{x_i}, B_{x_ix_j}\in\mathrm{Mat}(qN,\mathbb{C})).
	\]
The actual expressions of the $y$-equations of Pfaffian system \eqref{Eq_Pfaff_V} are given by \eqref{Eq_V_yeq} and \eqref{Eq_Pfaff_V_y2}.
We remark that the system \eqref{Eq_Pfaff_V} satisfies the assumption \eqref{Eq_ass_inf}. 
Namely, it holds that 
	\[
	\exists x_j \in x' ~\text{such that}~
	\mathcal{B}_x\cap\mathcal{B}_{x_j}\neq \emptyset
	~
	\Rightarrow
	~
	B_{xz}=O
	\quad
	(\forall z\in x').
	\]
Moreover, if $\mathcal{A}_x\cap\mathcal{A}_{x_j}\neq \emptyset$ for some $x_j \in x'$ in the initial system \eqref{Eq_Pfaff}, then the transformed system \eqref{Eq_Pfaff_V} satisfies $\mathcal{B}_x\cap\mathcal{B}_{z}=\emptyset$ for all $z\in x'$.
On the other hand, if $A_{xx_j}\neq O$ for some $x_j\in x'$ in the initial system \eqref{Eq_Pfaff}, then the transformed system \eqref{Eq_Pfaff_V} satisfies $B_{xz}=O$ for all $z\in x'$.

\begin{remark}\label{rem:integrableV}
If the Pfaffian system \eqref{Eq_Pfaff} is completely integrable, then the Pfaffian system \eqref{Eq_Pfaff_V} is also completely integrable.
\end{remark}

\begin{definition}\label{Def_L}
We call the operation which sends the Pfaffian system \eqref{Eq_Pfaff}  satisfying \eqref{Eq_ass_diag} and \eqref{Eq_ass_inf} to the Pfaffian system \eqref{Eq_Pfaff_V} the \emph{Laplace transform in $x_\ell$-direction} and denote it by $\mathcal{L}^{x_\ell}$.
\end{definition}

\medskip
\noindent
\textbf{Step 3.} 
Let $\mathcal{K}^x$ be a subspace of $(\mathbb{C}^{N})^q$ defined by
\begin{equation}\label{Eq_K}
	\mathcal{K}^x:=\bigoplus_{i=1}^q\Ker A_{H_i}
	=
	\left\{
	\begin{pmatrix}
	v_1 \\
	\vdots\\
	v_q
	\end{pmatrix}
	\in(\mathbb{C}^N)^q
	~\middle\vert~
	v_i \in \Ker A_{H_i}
	~
	(1\le i\le q)
	\right\},
\end{equation}
which is the same as \eqref{eq:odeK}. 
Then, thanks to Lemma \ref{lem:odeK}, we see that this subspace is invariant for the action of the $x$-equation \eqref{Eq_xeq_Lap_Sch2}. 
Namely, it holds that
	\begin{equation}\label{Eq_inv_K}
(\mathcal{L}^x(\Omega))_x \mathcal{K}^x\subset \mathcal{K}^x.
	\end{equation}
For $y$-equations ($y \in x'$), an analog of \cite[Proposition 2.2]{Haraoka2012} holds.
\begin{proposition}
\label{Prop_Inv_subsp}
For any $y\in x'$, the subspace $\mathcal{K}^x$ is invariant for the $y$-equation of the system \eqref{Eq_Pfaff_V}.
Namely,
	\[
	(\mathcal{L}^x(\Omega))_y\mathcal{K}^x \subset \mathcal{K}^x
	\]
holds for any value of $\bm{x}$.
\end{proposition}
\begin{proof}
Take any $v={}^t(v_1,\ldots,v_q)\in\mathcal{K}^x$, and set 
	\[
	w=(\mathcal{L}^x(\Omega))_yv={}^t(w_1,\ldots,w_q).
	\]
First, we shall show $A_{H_i}w_i=0$ for $H_i \in \mathcal{A}_x\cap\mathcal{A}_y^c$ (namely, $1\le i\le p$ or $p+1\le i\le q$ with $H_i \notin\mathcal{A}_y$). 
In this case, from \eqref{Eq_Prop_Vi_1} and $v_j \in \Ker A_{H_j}$, we have
	\begin{align*}
	w_i=&
		\left(A_{y}+
		\sum_{z\in x''}A_{yz} z+A_{yx}a_{H_i}
	+
	\sum_{\substack{q+1\le j \le r \\ H_j \in \mathcal{A}_y}}
	\frac{A_{H_j}}{y-b_{H_j}}
		\right)v_i
	+\sum_{\substack{p+1\le j\le q \\ H_{j} \in\mathcal{A}_y} }
	\frac{A_{H_j}}{y-c_{H_i H_j}}(v_i-v_j)
	 \\
	&\quad -\sum_{\substack{\hat{p}+1\le k\le \hat{q} \\ \hat{H} \in \mathcal{B}_y}}
	\frac{1}{y-b_{\hat{H}_k}}\sum_{j=1}^q E_{N_k}A_{H_j}v_j
	\\
	&=
	\left(A_{y}+
		\sum_{z\in x''}A_{yz} z+A_{yx}a_{H_i}
	+
	\sum_{\substack{q+1\le j \le r \\ H_j \in \mathcal{A}_y}}
	\frac{A_{H_j}}{y-b_{H_j}}
	+\sum_{\substack{p+1\le j\le q \\ H_{j} \in\mathcal{A}_y} }
	\frac{A_{H_j}}{y-c_{H_i H_j}}\right)v_i.
	\end{align*}
Then, we have
	\begin{align*}
	A_{H_i}w_i =
		A_{H_i}\left(A_{y}+
		\sum_{z\in x''}A_{yz} z+A_{yx}a_{H_i}
	+
	\sum_{\substack{q+1\le j \le r \\ H_j \in \mathcal{A}_y}}
	\frac{A_{H_j}}{y-b_{H_j}}
	+\sum_{\substack{p+1\le j\le q \\ H_{j} \in\mathcal{A}_y} }
	\frac{A_{H_j}}{y-c_{H_i H_j}}\right)v_i.
	\end{align*}
Here, it holds that $[A_{H_i},A_y+a_{H_i}A_{xy}]=O$ thanks to \eqref{eq:cond13p} and the condition \eqref{Eq_ass_inf}. 
It also holds that $[A_{H_i},A_{yz}]=O~(z\in x')$ thanks to \eqref{eq:cond14p} (note that, $(a_{H})_{z}A_{xy}=O$ holds under the condition \eqref{Eq_ass_inf}).
The commutativity on the remaining part is already shown by Haraoka (cf. the proof of \cite[Proposition 2.2]{Haraoka2012}). 
Therefore, we have $A_{H_i}w_i=0$ for $H_i \in \mathcal{A}_x\cap\mathcal{A}_y^c$.

Next, we show $A_{H_i}w=0$ with $H_i \in \mathcal{A}_x\cap\mathcal{A}_y$ (namely, $p+1\le i\le q$ with $H_i \in\mathcal{A}_y$). 
As in before, in this case we can assume $A_{xz}=O$ ($z\in x'$), especially $A_{xy}=O$. 
Then, by using \eqref{Eq_Prop_Vi_2}, we have
	\begin{align*}
	w_i
	&=	\left(A_{y}+
	(a_{H_i})_y A_{x}-(a_{H_i})_y x+
	\sum_{z\in x''} A_{yz}z+
	\sum_{\substack{q+1\le j \le r \\ H_j \in \mathcal{A}_y}}
	\frac{A_{H_j}}{y-b_{H_j}}\right)v_i 
	\\ 
	&\quad +\sum_{\substack{1 \le j \le q \\H_j \in \mathcal{C}_{H_i,y}}}
	\frac{A_{H_j}}{y-c_{H_i H_j}}\left(v_i-v_j\right) \\
	&=	\left(A_{y}+
	(a_{H_i})_y A_{x}-(a_{H_i})_y x+
	\sum_{z\in x''} A_{yz}z+
	\sum_{\substack{q+1\le j \le r \\ H_j \in \mathcal{A}_y}}
	\frac{A_{H_j}}{y-b_{H_j}}+\sum_{\substack{1 \le j \le q \\H_j \in \mathcal{C}_{H_i,y}}}
	\frac{A_{H_j}}{y-c_{H_i H_j}}\right)v_i.
	\end{align*}
Then we obtain
	\[
	A_{H_i}w_i=A_{H_i}	\left(A_{y}+
	(a_{H_i})_y A_{x}-(a_{H_i})_y x+
	\sum_{z\in x''} A_{yz}z+
	\sum_{\substack{q+1\le j \le r \\ H_j \in \mathcal{A}_y}}
	\frac{A_{H_j}}{y-b_{H_j}}+\sum_{\substack{1 \le j \le q \\H_j \in \mathcal{C}_{H_i,y}}}
	\frac{A_{H_j}}{y-c_{H_i H_j}}\right)v_i.
	\]
Here, it hold that $[A_{H_i},A_{y}+(a_{H_i})_yA_{x}]=[A_{H_i},A_{yz}]=O$ thanks to \eqref{eq:cond9p} and \eqref{eq:cond10p} (note that, in this case, $A_{xz}=O$ $(z\in x')$ thanks to the assumption \eqref{Eq_ass_inf}).
The commutativity on the remaining part is already shown by Haraoka (cf. the proof of \cite[Proposition 2.2]{Haraoka2012}). 
Therefore, we have $A_{H_i}w_i=0$ for $H_i \in \mathcal{A}_x\cap\mathcal{A}_y$.
As a result, we obtain $w\in\mathcal{K}^x$.
\end{proof}

Proposition \ref{Prop_Inv_subsp} leads that, the Pfaffian system \eqref{Eq_Pfaff_V} induces the Pfaffian system 
	\begin{equation}\label{Eq_Pfaff_v}
	dv=\mathcal{ML}^x(\Omega)v,
	\quad 
	\mathcal{ML}^x(\Omega)=-\sum_{i=1}^n\bar{T}_{x_i}(\bm{x})\,dx_i+\sum_{H \in \mathcal{B}}\bar{B}_H \, d\log f_{H}.
	\end{equation}
on the quotient space $\mathbb{C}^{qN}/\mathcal{K}^x$.
Here $\mathcal{ML}^x(\Omega)$ is the matrix $1$-form that represents the action of $\mathcal{L}^x(\Omega)$ on $\mathbb{C}^{qN}/\mathcal{K}^x$.

\begin{definition}\label{Def_ML}
We call the operation which sends the Pfaffian system \eqref{Eq_Pfaff}  satisfying \eqref{Eq_ass_diag} and \eqref{Eq_ass_inf} to the Pfaffian system \eqref{Eq_Pfaff_v} the \emph{middle Laplace transform in $x_\ell$-direction} and denote it by $\mathcal{ML}^{x_\ell}$.
\end{definition}
We note that, the middle Laplace transform $\mathcal{ML}^{x_\ell}$ changes the rank of Pfaffian systems in general.
Some examples will be given in Section \ref{sec:extwo}.

\subsection{Inverse (middle) Laplace transform}
\label{Subsec_IFL}
In the same way as the previous subsection, we can define the \emph{inverse Laplace transform} and the \emph{inverse middle Laplace transform}.
Let us consider the Pfaffian system \eqref{Eq_Pfaff}
satisfying \eqref{Eq_ass_diag} and \eqref{Eq_ass_inf}.
We take $x_\ell=x$ satisfying \eqref{Eq_ass_diag} and \eqref{Eq_ass_inf}, and suppose \eqref{Eq_A_yA}.
For this system, we consider the Birkhoff-Okubo extension \eqref{Eq_Pfaff_U} and its inverse Laplace transform in $x$-direction
	\begin{equation}\label{Eq_op_invFL}
		L^{-x}:\left\{
	\begin{array}{ccc}
	\dfrac{\partial }{\partial x} &\,\mapsto\, & -x, \\
	x &\,\mapsto\,& \dfrac{\partial }{\partial x},
	\end{array}\right. 
	\end{equation}	
which corresponds to the integral transform
	\[
	U(\bm{x})~\mapsto~W(\bm{x})=\int_{\Delta}U(t,x')e^{xt}\,dt
	\]
with a suitable path of integration $\Delta$.
In the Pfaffian system \eqref{Eq_Pfaff_U}, we set $W:=L^{-x}(U)$. 
Namely,
	\begin{equation}\label{Eq_W_def}
	W={}^t(W_1,W_2,\ldots,W_q), \quad W_i:=L^{-x}(U_i).
	\end{equation}
Then, as in Subsection \ref{subsec:invLone}, the $x$-equation \eqref{Eq_U_xeq_BO} of the system \eqref{Eq_Pfaff_U} is transformed into 
		\begin{equation}\label{Eq_xeq_InvLap_BO}
	\left(x+S_{x}(\bm{x})^{\oplus q}\right)\frac{\p W}{\partial x}
	=
	\left(-A+\left(x+S_{x}(\bm{x})^{\oplus q}\right)T_{x}(\bm{x})\right)W.
	\end{equation}
This equation can be transformed into
	\begin{equation}
	\label{Eq_xeq_InvLap_BO_Sch}
	\frac{\p W}{\p x}=
	\left(T_{x}(\bm{x})+\sum_{j=1}^{\hat{q}} 
	\frac{B_{\hat{H}_j}}{x+a_{\hat{H}_j}}\right)W
	\end{equation}
where $B_{\hat{H}_j}$ is given by \eqref{Eq_hat_G}.
This equation is the same as \eqref{eq:odeinvL}.
On the $y$-equation, in the same way as Proposition \ref{Prop_Laplace}, we can show the following.

\begin{proposition}
\label{Prop_invLaplace}
Suppose that the system \eqref{Eq_Pfaff} satisfies \eqref{Eq_ass_diag} and \eqref{Eq_ass_inf}.
By the transformation \eqref{Eq_op_invFL}, the $y$-equation ($y\in x'$) of \eqref{Eq_Pfaff_U} is transformed as follows:
	\begin{itemize}
	\item For $W_i$ such that $H_i \in \mathcal{A}_x\cap\mathcal{A}_y^c$
	(namely, $1\le i \le p$ or $p+1\le i \le q$ with $H_i \notin \mathcal{A}_y$),
	\begin{equation}\label{Eq_Prop_Wi_1}
	\begin{aligned}
	\frac{\p W_i}{\p y}=&
		\left(A_{y}+
		\sum_{z\in x''}A_{yz} z+a_{H_i}A_{yx}
	+
	\sum_{\substack{q+1\le j \le r \\ H_j \in \mathcal{A}_y}}
	\frac{A_{H_j}}{y-b_{H_j}}
		\right)W_i \\
	&	
	+\sum_{\substack{p+1\le j\le q \\ H_{j} \in\mathcal{A}_y} }
	\frac{A_{H_j}}{y-c_{H_i H_j}}(W_i-W_j)
	 -\sum_{\substack{\hat{p}+1\le k\le \hat{q} \\ \hat{H}_k \in \mathcal{B}_y}}
	\frac{(a_{\hat{H}_k})_y}{x+a_{\hat{H}_k}}\sum_{j=1}^q E_{N_k}A_{H_j}W_j.
	\end{aligned}
	\end{equation}

	\item For $W_i$ such that $H_i \in \mathcal{A}_x\cap\mathcal{A}_y$
	(namely, $p+1\le i\le q$ with $H_i \in \mathcal{A}_y$),
	\begin{equation}\label{Eq_Prop_Wi_2}
	\begin{aligned}
	\frac{\p W_i}{\p y}
	&=	\left(A_{y}+
	(a_{H_i})_y A_{x}+(a_{H_i})_y x+\sum_{z\in x''} A_{yz}z+
	\sum_{\substack{q+1\le j \le r \\ H_j \in \mathcal{A}_y}}
	\frac{A_{H_j}}{y-b_{H_j}}\right)W_i 
	\\ 
	&\quad +\sum_{\substack{1 \le j \le q \\H_j \in\mathcal{C}_{H_i,y}}}
	\frac{A_{H_j}}{y-c_{H_i H_j}}\left(W_i-W_j\right).
	\end{aligned}
	\end{equation}
	\end{itemize}
\end{proposition}
Combining \eqref{Eq_xeq_InvLap_BO_Sch}, \eqref{Eq_Prop_Wi_1} and \eqref{Eq_Prop_Wi_2} for all $y \in x'$, we obtain the Pfaffian system
	\begin{equation}\label{Eq_Pfaff_W}
	dW=\mathcal{L}^{-x}(\Omega) W, \quad 
	\mathcal{L}^{-x}(\Omega):=
	\sum_{i=1}^n T_{x_i}(\bm{x})\,dx_i
	+\sum_{H \in \mathcal{B}^-}B_H \, d\log f_H,
	\end{equation}
where  
	\[
	\mathcal{\mathcal{B}}^-:=\{\sigma(\hat{H}_1),\ldots,\sigma(\hat{H}_{\hat{q}})\}\cup(\widetilde{\mathcal{A}}\setminus\mathcal{A}_x). 
	\]
and $\sigma(H)$ is the hyperplane defined by the polynomial $f_{\sigma(H)}:=(f_{H})_{x}(x+a_{H})$. 
We note that $T_{x_i}(\bm{x})$ is of the form
	\[
	T_{x_i}(\bm{x})=B_{x_i}+\sum_{\substack{j=1 \\ j\neq i}}^nB_{x_i x_j}x_j,
	\quad 
	(B_{x_i}, B_{x_i x_j}\in\mathrm{Mat}(qN,\mathbb{C})).
	\]
and the system \eqref{Eq_Pfaff_W} satisfies the assumption \eqref{Eq_ass_inf}. 
Namely, it holds that 
	\[
	\exists x_j \in x' ~\text{such that}~
	\mathcal{B}_x^-\cap\mathcal{B}_{x_j}^-\neq \emptyset
	~
	\Rightarrow
	~
	B_{xz}=O
	\quad
	(\forall z\in x').
	\]
Moreover, if $\mathcal{A}_x\cap\mathcal{A}_{x_j}\neq \emptyset$ for some $x_j \in x'$ in the initial system \eqref{Eq_Pfaff}, then the transformed system \eqref{Eq_Pfaff_W} satisfies $\mathcal{B}_x^-\cap\mathcal{B}_{z}^-=\emptyset$ for all $z \in x'$.
On the other hand, if $A_{xx_j}\neq O$ for some $x_j\in x'$ in the initial system \eqref{Eq_Pfaff}, then the transformed system \eqref{Eq_Pfaff_W} satisfies $B_{xz}=O$ for all $z\in x'$.

\begin{remark}\label{rem:integrableW}
If the Pfaffian system \eqref{Eq_Pfaff} is completely integrable, then the Pfaffian system \eqref{Eq_Pfaff_W} is also completely integrable.
\end{remark}

\begin{definition}\label{Def_invL}
 We call the operation which sends the Pfaffian system \eqref{Eq_Pfaff}  satisfying \eqref{Eq_ass_diag} and \eqref{Eq_ass_inf} to the Pfaffian system \eqref{Eq_Pfaff_W} the \emph{inverse Laplace transform in $x_\ell$-direction} and denote it by $\mathcal{L}^{-x_\ell}$.
\end{definition}
The subspace $\mathcal{K}^x$ is invariant for the system \eqref{Eq_Pfaff_W}, which can be shown in the same way with the previous subsection.
Therefore, the Pfaffian system \eqref{Eq_Pfaff_W} induces the Pfaffian system 
	\begin{equation}\label{Eq_Pfaff_w}
	dw=\mathcal{ML}^{-x}(\Omega)w,
	\quad 
	\mathcal{ML}^{-x}(\Omega)=\sum_{i=1}^n\bar{T}_{x_i}(\bm{x})\,dx_i+\sum_{H \in \mathcal{B}^-}\bar{B}_H d\log f_{H}.
	\end{equation}
on the quotient space $\mathbb{C}^{qN}/\mathcal{K}^x$.
Here $\mathcal{ML}^{-x}(\Omega)$ is the matrix $1$-form that represents the action of $\mathcal{L}^{-x}(\Omega)$ on $\mathbb{C}^{qN}/\mathcal{K}^x$.

\begin{definition}\label{Def_invML}
We call the operation which sends the Pfaffian system \eqref{Eq_Pfaff}  satisfying \eqref{Eq_ass_diag} and \eqref{Eq_ass_inf}  to the Pfaffian system \eqref{Eq_Pfaff_w} the \emph{inverse middle Laplace transform in $x_\ell$-direction} and denote it by $\mathcal{ML}^{-x_\ell}$.
\end{definition}
%

\subsection{Fundamental properties of the middle Laplace transform}

	We summarize some fundamental properties of the (inverse) middle Laplace transform without proofs.
	All the proofs will be given in Section \ref{Sec_Inv}.

	\begin{definition}\label{Def_irr}
	Let $N\ge1$. 
	For a Pfaffian system \eqref{Eq_Pfaff}, or equivalently its corresponding $1$-form $\Omega$, we use the following terminology. 
	\begin{itemize}
	\item We say that it is \emph{irreducible in $x_\ell$-direction} if the $x_\ell$-equation of $du=\Omega u$ is irreducible in the sense of Definition \ref{def:odeirred}. 
	Namely,  there is no $\langle A_{x_\ell},\{A_{x_\ell x_j}\}_{j\neq \ell},\{A_{H}\}_{H\in\mathcal{A}_{x_\ell}}\rangle$-invariant subspace  except $\mathbb{C}^N$ and $\{0\}$. 
	\item
	We say that it is \emph{exceptional in $x_{\ell}$-direction} if the $x_\ell$-equation of $du=\Omega u$ is exceptional in the sense of Definition \ref{def:odeirred}. Namely, $N=1$ and $A_H=0$ for all $H\in\calA_{x_\ell}$.
	Otherwise, we say it is \emph{non-exceptional in $x_{\ell}$-direction}.
	\end{itemize}
	\end{definition}
	\begin{theorem}
	\label{Thm_main}
	If a Pfaffian system \eqref{Eq_Pfaff} satisfying \eqref{Eq_ass_diag} and \eqref{Eq_ass_inf} is irreducible and non-exceptional in $x_\ell$-direction, then the following hold.
	\begin{enumerate}[(i)]
	\item $\mathcal{ML}^{-x_\ell}\circ\mathcal{ML}^{x_\ell}(\Omega)\sim \Omega$ and $\mathcal{ML}^{x_\ell}\circ\mathcal{ML}^{-x_\ell}(\Omega)\sim \Omega$.
	\item The middle Laplace transformed system \eqref{Eq_Pfaff_v} and the inverse middle Laplace transformed system \eqref{Eq_Pfaff_w} are also irreducible and non-exceptional in $x_\ell$-direction.
	\end{enumerate}
Here, the equivalence relation $\sim$ on $1$-forms is defined by 
	\begin{equation}\label{Eq_equiv}
	\Omega \sim \Omega' 
	~\overset{\mathrm{def}}{\iff}~
	\exists P \in \mathrm{GL}(N,\mathbb{C}) \text{ such that }
	\Omega'=P^{-1}\Omega P.
	\end{equation}
	\end{theorem}
We remark that, in terms of Pfaffian systems, statement (i) can be written as 
	\[
	\mathcal{ML}^{-x_\ell}\circ\mathcal{ML}^{x_\ell}=\mathcal{ML}^{x_\ell}\circ\mathcal{ML}^{-x_\ell}=\mathrm{id.}
	\]
	Here, the above symbol $``="$ is understood up to constant gauge equivalence; that is, the equations obtained by $\mathcal{ML}^{-x_\ell}\circ\mathcal{ML}^{x_\ell}$ and $\mathcal{ML}^{x_\ell}\circ\mathcal{ML}^{-x_\ell}$ coincide with the original one modulo a gauge transformation by some constant matrix $P \in \mathrm{GL}(N, \mathbb{C})$.
	We also remark that statement (ii) is essentially the same as Theorem \ref{thm:odemain} (ii).

\subsection{Formulation of the middle Laplace transform as a map}\label{Subsec_mapping}

Since the (inverse) middle Laplace transformed systems \eqref{Eq_Pfaff_v} and \eqref{Eq_Pfaff_w} depend on the choice of the basis of $\mathbb{C}^{qN}/\mathcal{K}$, the (inverse) middle Laplace transform operation cannot be regarded as a map. 
However, they induce the map of quotient spaces consisting of isomorphic classes of Pfaffian systems as follows.
Fix any two hyperplane arrangements $\mathcal{S}_{x_\ell}$ and $\mathcal{T}_{x_\ell}$ consisting of hyperplanes whose defining polynomial depends on $x_\ell$.
Then, we define the set $\mathfrak{P}_{\mathcal{T}_{x_\ell},\mathcal{S}_{x_\ell}}$ of the $1$-forms of the form of \eqref{Eq_Pfaff} by
	\begin{equation}
	\begin{aligned}
\mathfrak{P}_{\mathcal{T}_{x_\ell},\mathcal{S}_{x_\ell}}:=&\left\{
	\Omega=\sum_{i=1}^nS_{x_i}(\bm{x})dx_i
	+\sum_{H \in \mathcal{A}}A_H \, d\log f_H
	 ~\middle\vert~
	\begin{array}{l}
	A_{x_i}, A_{x_ix_j}, A_H \in \mathrm{Mat}(N,\mathbb{C}),\,N\in\mathbb{Z}_{\ge0}, \\
	\text{$\Omega$ satisfies \eqref{Eq_xixj}, \eqref{Eq_integrable}, \eqref{Eq_ass_diag}, \eqref{Eq_ass_inf},} \\
	\mathcal{A}_{x_\ell}=\mathcal{T}_{x_\ell}$ \text{~and~} $\mathcal{B}_{x_\ell}\subset\mathcal{S}_{x_\ell}
	\end{array}
	\right\}.
\end{aligned}
	\end{equation}
	Here the arrangement $\mathcal{B}_{x_\ell}$ is defined by \eqref{Eq_Bx} when $x=x_{\ell}$.
We note that the size of the coefficient matrices of $1$-forms in $\mathfrak{P}_{\mathcal{T}_{x_\ell},\mathcal{S}_{x_\ell}}$ is not fixed. 
We often identify a $1$-form $\Omega\in \mathfrak{P}_{\mathcal{T}_{x_\ell},\mathcal{S}_{x_\ell}}$ with the Pfaffian system $du=\Omega u$. 
We define the quotient space $\mathfrak{M}_{\mathcal{T}_{x_\ell},\mathcal{S}_{x_\ell}}$ of $1$-forms by
	\[
	\begin{aligned}
	\mathfrak{M}_{\mathcal{T}_{x_\ell},\mathcal{S}_{x_\ell}}:=\mathfrak{P}_{\mathcal{T}_{x_\ell},\mathcal{S}_{x_\ell}}/\sim.
	\end{aligned}
	\]
Here the equivalence relation $\sim$ is defined by \eqref{Eq_equiv}.
Then, the middle Laplace transform $\mathcal{ML}^{x_\ell}$ and the inverse middle Laplace transform $\mathcal{ML}^{-x_\ell}$ can be regarded as maps
	\begin{equation}\label{Eq_ML_as_map}
	\mathcal{ML}^{x_\ell}: \mathfrak{M}_{\mathcal{T}_{x_\ell},\mathcal{S}_{x_\ell}}
	\to\mathfrak{M}_{\mathcal{S}_{x_\ell},-\mathcal{T}_{x_\ell}}, 
	\quad 
	\mathcal{ML}^{-x_\ell}: \mathfrak{M}_{\mathcal{T}_{x_\ell},\mathcal{S}_{x_\ell}}
	\to\mathfrak{M}_{-\mathcal{S}_{x_\ell},\mathcal{T}_{x_\ell}}.
	\end{equation}
Here $-\mathcal{T}_{x_\ell}$ denotes the hyperplane arrangement $\{\sigma(H)\mid H\in\mathcal{T}_{x_\ell}\}$. 
For example, if 
	\[
	\mathcal{T}_{x_1}=\{\{x_1=0\}, \{x_1-1=0\}, \{x_1-x_2=0\}\},
	\]
then
	\[
	-\mathcal{T}_{x_1}=\{\{x_1=0\}, \{x_1+1=0\},\{x_1+x_2=0\}\}.
	\]
In this setting, Theorem \ref{Thm_main} can be rephrased concisely. 
Let $\mathfrak{P}_{\mathcal{T}_{x_\ell},\mathcal{S}_{x_\ell}}^{\rm \rm irr,ne}$ be the subset of $\mathfrak{P}_{\mathcal{T}_{x_\ell},\mathcal{S}_{x_\ell}}$ consisting of $1$-forms which are irreducible and non-exceptional in $x_\ell$-direction, and set
	\[
	\mathfrak{M}_{\mathcal{T}_{x_\ell},\mathcal{S}_{x_\ell}}^{\rm irr,ne}:=\mathfrak{P}_{\mathcal{T}_{x_\ell},\mathcal{S}_{x_\ell}}^{\rm irr,ne}/\sim.
	\]
Then, Theorem \ref{Thm_main} states that the maps 
	\[
	\mathcal{ML}^{x_\ell}: \mathfrak{M}_{\mathcal{T}_{x_\ell},\mathcal{S}_{x_\ell}}^{\rm irr,ne}
	\to\mathfrak{M}_{\mathcal{S}_{x_\ell},-\mathcal{T}_{x_\ell}}^{\rm irr,ne}, 
	\quad 
	\mathcal{ML}^{-x_\ell}: \mathfrak{M}_{\mathcal{T}_{x_\ell},\mathcal{S}_{x_\ell}}^{\rm irr,ne}
	\to\mathfrak{M}_{-\mathcal{S}_{x_\ell},\mathcal{T}_{x_\ell}}^{\rm irr,ne}
	\]
are both bijective.

\section{Categorical interpretation (several variables case)}
\label{Sec_Categorical}

We give a categorical interpretation of the (middle) Laplace transform with several variables, which is useful to prove Theorem \ref{Thm_main}.

\subsection{Category of meromorphic connections}

We set $X=(\mathbb{P}^1)^n$. 
Let $\mathcal{O}_{X}$ be the sheaf of holomorphic functions on $X$. 
We set $\mathcal{M}_X^1$ be the sheaf of meromorphic $1$-forms on $X$. 
Let $\mathscr{C}_{X}$ be the category of meromorphic connections on $X$. 
That is, 
		\begin{equation}\label{Eq_merom}
	\mathrm{Ob}(\mathscr{C}_X) = \left\{ (\mathcal{O}_X\otimes \mathcal{V},\nabla)~ \middle|~
	 \begin{array}{l}
\text{$\mathcal{V}$ is a finite dimensional $\mathbb{C}$-vector space and} \\
\text{$\nabla:\mathcal{O}_X\otimes \mathcal{V}\to\mathcal{M}_X^1 \otimes_{\mathcal{O}_{X}} (\mathcal{O}_X\otimes \mathcal{V})$}
\\
\text{is a $\mathbb{C}$-linear map satisfying } \\
\nabla(hv)=dh\otimes_{\mathcal{O}_{X}}   v+ h \nabla(v) ~ (h\in\mathcal{O}_{X},\,v \in \mathcal{O}_{X}\otimes \mathcal{V})
 	\end{array}
 \right\},
	\end{equation}
where we write $\otimes_{\mathbb{C}}=\otimes$ for short.
A linear map $f:\mathcal{V}\to\mathcal{V}'$ is called a morphism of (meromorphic) connections from $(\mathcal{O}_X\otimes\mathcal{V},\nabla)$ to $(\mathcal{O}_X\otimes\mathcal{V}',\nabla')$ if the diagram 
		\[
		\begin{tikzcd}
		\mathcal{O}_X\otimes \mathcal{V} \arrow[r,"\nabla",""]
		\arrow[d,"\varphi" ']
		&\mathcal{M}_X^1\otimes_{\mathcal{O}_{X}} (\mathcal{O}_X\otimes \mathcal{V})\arrow[d,""," \mathrm{id.}\otimes_{\mathcal{O}_{X}}\varphi"]\\
		\mathcal{O}_X\otimes \mathcal{V}'  \arrow[r,"\nabla'",""]
		&\mathcal{M}_X^1\otimes_{\mathcal{O}_{X}} (\mathcal{O}_X\otimes \mathcal{V}')
	\end{tikzcd}
	\quad 
	\]
is commutative.
Here $\varphi:\mathcal{O}_{X}\otimes\mathcal{V}\to\mathcal{O}_{X}\otimes\mathcal{V}'$ is a morphism of $\mathcal{O}_{X}$-modules obtained through the isomorphism 
	\begin{align*}
	f\in \mathrm{Hom}_{\mathbb{C}}(\mathcal{V},\mathcal{V}') 
	&\cong \mathrm{Hom}_{\mathbb{C}}(\mathbb{C}^N,\mathbb{C}^{N'}) \quad (\mathcal{V}\cong\mathbb{C}^N,~\mathcal{V}'\cong\mathbb{C}^{N'}) \\
	&\cong \mathrm{Hom}_{\mathcal{O}_{X}(X)}(\mathcal{O}_{X}(X)^{N},\mathcal{O}_{X}(X)^{N'}) \\
	&\cong \mathrm{Hom}_{\mathcal{O}_{X}}(\mathcal{O}_{X}^N,\mathcal{O}_{X}^{N'}) \\
	&\cong\mathrm{Hom}_{\mathcal{O}_{X}}(\mathcal{O}_{X}\otimes\mathcal{V},\mathcal{O}_{X}\otimes \mathcal{V}')\ni\varphi.
	\end{align*}
In the above, we used the fact $\mathcal{O}_{X}(X)=\mathbb{C}$ which follows from the compactness of $X=(\mathbb{P}^1)^n$.

A linear map $\nabla$ of $(\mathcal{O}_X\otimes\mathcal{V},\nabla)$ can be expressed as $\nabla=d-\Omega$, where $\Omega$ is an $\mathrm{End}_{\mathbb{C}}(\mathcal{V})$-valued meromorphic $1$-form. 
Then, we identify $\mathscr{C}_{X}$ with the category of pairs $(\mathcal{V},\Omega)$ consisting of a finite dimensional $\mathbb{C}$-vector space $\mathcal{V}$ and an $\mathrm{End}_{\mathbb{C}}(\mathcal{V})$-valued meromorphic $1$-form $\Omega$. 
The morphisms $(\mathcal{V},\Omega)\to(\mathcal{V}',\Omega')$ in $\mathscr{C}_{X}$ are linear maps $f:\mathcal{V}\to\mathcal{V}'$ satisfying $\Omega'f=f\Omega$.

We shall also explain the correspondence of meromorphic connections and differential equations.
For a connection $(\mathcal{V}, \Omega)\in\mathscr{C}_{X}$, we fix an isomorphism (a basis) $\mathcal{V}\cong \mathbb{C}^N$. 
Then the $1$-form $\Omega$ can be expressed as
	\[
	\Omega=\sum_{i=1}^n A_{x_i}(\bm{x})\,dx_i, 
	\]
where $A_{x_i}(\bm{x})\in\mathrm{Mat}(N,\mathbb{C}(\bm{x}))$. 
Therefore, the connection $(\mathcal{V},\Omega)$ can be identified with the $\mathrm{GL}(N,\mathbb{C})$-conjugacy class of the Pfaffian system $du=\Omega u$.
We call this matrix $A_{x_i}(\bm{x})$ the coefficient matrix of $\Omega$ in $x_i$-direction. 
For another connection $(\mathcal{V}',\Omega')\in\mathscr{C}_{X}$, we fix an isomorphism $\mathcal{V}'\cong\mathbb{C}^{N'}$ and let $\{B_{x_i}(\bm{x})\}_{i=1}^n$ be the coefficient matrices of $\Omega'$. 
Then, the linear map $f:\mathcal{V}\to\mathcal{V}'$ is to be a morphism of connections $f:(\mathcal{V},\Omega)\to(\mathcal{V}',\Omega')$, if and only if the matrix representation $F\in\mathrm{Mat}(N'\times N,\mathbb{C})$ of $f$ with respect to the above two basis satisfies
	\[
	FA_{x_i}(\bm{x})=B_{x_i}(\bm{x})F \qquad (i=1,2,\ldots,n).
	\]
In particular, if the morphism $f:(\mathcal{V},\Omega)\to(\mathcal{V}',\Omega')$ is an isomorphism of vector spaces, then the corresponding matrix representation $F$ is invertible. 
This means that the Pfaffian system $du=\Omega u$ can be transformed into $dv=\Omega' v$ by the gauge transformation $v=Fu$.

\subsection{Subcategories of $\mathscr{C}_X$}

Here we introduce some subcategories of $\mathscr{C}_X$. 
Firstly, let $\mathscr{C}_{X}^{\text{Int}}$ be the full subcategory of $\mathscr{C}_{X}$ consisting of flat connections. That is, 
\begin{equation}
\mathrm{Ob}(\mathscr{C}_{X}^{\text{Int}}):=\{(\mathcal{O}_X\otimes \mathcal{V},\nabla)\,\vert\, \nabla\circ\nabla=0\}=\{(\calV,\Omega)\,\vert\,d\Omega-\Omega\wedge\Omega=0\}.
\end{equation}
In the following, as in the previous section, take any $\ell$ $(1\le \ell\le n)$ and set $x=x_{\ell}$. 

\subsubsection{Category of Pfaffian systems with logarithmic singularities along hyperplane arrangements}

 Let $\mathcal{S}_{x}$ and $\mathcal{T}_{x}$ be two hyperplane arrangements consisting of hyperplanes whose defining polynomial depends on $x$.
 We set 
	\begin{equation}\label{Eq_conn_hs}
	\mathcal{T}_{x}=\{H_1,\ldots, H_p,\ldots,H_q\},
	\quad 
	\mathcal{S}_{x}=\{\hat{H}_1,\ldots,\hat{H}_{\hat{p}},\ldots,\hat{H}_{\hat{q}}\},
	\end{equation}
where $f_{H_j} (j=1,\ldots,p)$ and $f_{\hat{H}_j} (j=1,\ldots,\hat{p})$ depend only on $x$, and  $f_{H_j} (j=p+1,\ldots,q)$ and $f_{\hat{H}_j} (j=\hat{p}+1,\ldots,\hat{q})$ depend on $x$ and other variables. 
Then, we define the full subcategory $\mathscr{P}_{x}=\mathscr{P}_{x}(\mathcal{T}_{x},\mathcal{S}_{x})$ as
	\[
	\mathrm{Ob}(\mathscr{P}_{x}(\mathcal{T}_{x},\mathcal{S}_{x})):=\{(\mathcal{V},\Omega)\in\mathrm{Ob}(\mathscr{C}_X) \mid
	\text{$\Omega$ satisfies the condition (P)}\}.
	\]
Here the condition (P) is given as follows: there exists a hyperplane arrangement $\mathcal{A}$ satisfying $\mathcal{A}_x=\mathcal{T}_x$ such that the $1$-form $\Omega$ can be expressed as
	\begin{equation}\label{Eq_conn_Omega}
	\Omega=\sum_{i=1}^nA_{x_i}(\bm{x})\,dx_i=\sum_{i=1}^nS_{x_i}(\bm{x})\,dx_i+\sum_{H\in\mathcal{A}}A_Hd\log f_H,
	\end{equation}
where $\mathcal{A}_x=\mathcal{T}_x$, $A_H\in \mathrm{End}_{\mathbb{C}}(\mathcal{V})$ and $S_{x_i}(\bm{x})$ are expressed as
	\begin{equation}\label{Eq_S}
	S_{x_i}(\bm{x})=A_{x_i}+\sum_{\substack{j=1 \\ j \neq i}}^nA_{x_ix_j} x_j,
	\quad 
	A_{x_i},A_{x_ix_j}\in\mathrm{End}_{\mathbb{C}}(\mathcal{V})
	\end{equation}
and satisfy 
\begin{enumerate}[(i)]
	 \item $\ds A_{x_ix_j}=A_{x_jx_i}~\left(\Leftrightarrow~\frac{\p S_{x_i}}{\p x_j}=\frac{\p S_{x_j}}{\p x_i}\right)$ \quad $(i\neq j)$, 

	\item $A_{x}$ and $A_{xz}\,(z \in x')$ are diagonalizable,

	\item If there exists $x_j \in x'$ such that $\mathcal{A}_{x} \cap \mathcal{A}_{x_j} \neq \emptyset$, 
	then 
	$A_{x z}=O$ for all $z\in x'$,

	\item $\ds\{\bm{x}\in\mathbb{C}^n\mid\det(x-S_{x}(\bm{x}))=0\}\subset\bigcup_{i=1}^{\hat{q}}\hat{H}_{i}$.
	
\end{enumerate}
The conditions (i), (ii) and (iii) correspond to \eqref{Eq_xixj},  \eqref{Eq_ass_diag} and \eqref{Eq_ass_inf}, respectively. 
Then, the category $\mathscr{P}_{x}(\mathcal{T}_{x},\mathcal{S}_{x})$ can be regarded as the category of Pfaffian systems of the form \eqref{Eq_Pfaff} with \eqref{Eq_xixj},  \eqref{Eq_ass_diag} and \eqref{Eq_ass_inf}.
Let $\mathscr{P}_{x}^{\text{Int}}$ be the full subcategory of $\mathscr{P}_{x}$ consisting of flat connections. 
Both $\mathscr{P}_{x}(\mathcal{T}_{x},\mathcal{S}_{x})$ and $\mathscr{P}_{x}^{\text{Int}}(\mathcal{T}_{x},\mathcal{S}_{x})$ are abelian categories.
\begin{remark}\label{Rem_equiv_conn}
Two objects $(\mathcal{V},\Omega),(\mathcal{V}',\Omega')\in\mathscr{P}_{x_\ell}(\mathcal{T}_{x_\ell},\mathcal{S}_{x_\ell})$ are isomorphic if and only if $\mathcal{V}\cong\mathcal{V}'$ as $\mathbb{C}$-vector space and $\Omega\sim\Omega'$ (in the sense of \eqref{Eq_equiv}) holds with respect to any basis of $\mathcal{V}\cong\mathbb{C}^N$ and $\mathcal{V}'\cong \mathbb{C}^{N}$.
In this case, we write $(\mathcal{V},\Omega)\sim(\mathcal{V}',\Omega')$ by abuse of notation.
\end{remark}

\subsubsection{Category of Pfaffian systems of Birkhoff-Okubo normal form}

We define the subcategory $\mathscr{B}_{x}=\mathscr{B}_{x}(\mathcal{T}_{x},\mathcal{S}_{x})$ of $\mathscr{P}_{x}$. 
First, we define $\mathrm{Ob}(\mathscr{B}_{x})$ as
	\[
	\mathrm{Ob}(\mathscr{B}_{x})=
	\left\{(\mathcal{V},\Omega)\in\mathrm{Ob}(\mathscr{P}_{x})
	\,
	\middle \vert
	\,
	\begin{array}{l}
	\text{There exists an isomorphism }\mathcal{V}\cong \mathbb{C}^N 
	\\
	\text{(i.e., a choice of basis) for some $N\in\mathbb{Z}_{\ge0}$}
	\\ \text{such that $\Omega$ satisfies the condition (B)}
	\end{array}
	\right\}.
	\]
Here the condition (B) is given as follows: there exists a decomposition of $N$
	\[
	N=N_1+\cdots+N_{q} \quad (N_j\ge0,~ j=1,2,\ldots,q)
	\]
and a constant matrix $A\in\mathrm{Mat}(N,\mathbb{C})$ such that the coefficient matrix $A_{x}(\bm{x})$ of $\Omega$ can be expressed as
	\begin{equation}\label{Eq_Cat_Bcoef}
	A_{x}(\bm{x})=S_{x}(\bm{x})+(x-T_{x}(\bm{x}))^{-1}A,
	\end{equation}
where $S_{x}(\bm{x})$ is a diagonal matrix and 
\begin{align}\label{eq:catT}
	T_{x}(\bm{x})&=\begin{pmatrix}
		a_{H_1} I_{N_1} & & & & & \\
		& \ddots & & & & \\
		& & a_{H_p}I_{N_{p}} & & &\\
		& & & a_{H_{p+1}}(x')I_{N_{p+1}} & & \\
		& & & & \ddots & \\
		& & & & & a_{H_{q}}(x')I_{N_{q}}
		\end{pmatrix}.
	\end{align}
If $N_i = 0$, we omit the corresponding diagonal block $a_{H_{i}}(x') I_{N_i}$ in the block diagonal matrix.
For ($\calV,\Omega)\in\mathrm{Ob}(\scrP_{x})$ with $\mathcal V=\{0\}$ (corresponding to the case $N=0$), condition (B) is understood to be satisfied, with all matrices appearing therein interpreted as empty matrices.

For $(\mathcal{V},\Omega),(\mathcal{V}',\Omega')\in\mathrm{Ob}(\mathscr{B}_{x})$, we define
	\[
	\mathrm{Hom}_{\mathscr{B}_{x}}((\mathcal{V},\Omega),(\mathcal{V}',\Omega')):=\{f \in \mathrm{Hom}_{\mathscr{P}_{x}}((\mathcal{V},\Omega),(\mathcal{V}',\Omega'))
	\mid
	f \text{ satisfies (Bmor)}
	\},
	\]
where the condition (Bmor) is the following: 
for the basis $\mathcal{V}\cong\mathbb{C}^N$ and $\mathcal{V}'\cong\mathbb{C}^{N'}$ satisfying the condition (B), we express the coefficient matrices of $\Omega$ and $\Omega'$ in $x$-direction as \eqref{Eq_Cat_Bcoef} and 
	\[
	A'_{x}(\bm{x})=S'_{x}(\bm{x})+(x-T'_{x}(\bm{x}))^{-1}A',
	\]
respectively. 
Then, the matrix representation $F\in\mathrm{Mat}(N'\times N,\mathbb{C})$ of $f$ satisfies
	\begin{equation}\label{Eq_cat_B_mor}
	\begin{aligned}
	&FT_{x}(\bm{x})=T'_{x}(\bm{x})F, \\
	&FS_{x}(\bm{x})=S'_{x}(\bm{x})F, \\
	&FA=A'F.
	\end{aligned}
	\end{equation}
If $N=0$ or $N'=0$, condition (Bmor) is understood to be  automatically satisfied, with the usual convention on empty matrices.

\subsection{Functors}
\label{Subsec_Func}

Here we introduce three constructions corresponding to the three steps of the middle Laplace transform: extension to the Birkhoff-Okubo normal form, Laplace transform, and projection onto the quotient space.

\subsubsection{Birkhoff-Okubo extension functor}
\label{Subsubsec_BOf}

We define the functor $BO^{x}:\mathscr{P}_{x}^{\text{Int}}(\mathcal{T}_{x},\mathcal{S}_{x}) \to \mathscr{B}_{x}(\mathcal{T}_{x},\mathcal{S}_{x})$.
This corresponds to the extension of a given Pfaffian system to the Birkhoff-Okubo normal form.

For $(\mathcal{V},\Omega)\in\mathscr{P}_{x}^{\text{Int}}(\mathcal{T}_{x},\mathcal{S}_{x})$, we fix an isomorphism $\mathcal{V}\cong\mathbb{C}^N$ such that the coefficient matrix of $\Omega$ in $x$-direction
	\begin{equation}\label{Eq_conn_Axi}
	A_{x}(\bm{x})=S_{x}(\bm{x})+\sum_{i=1}^q\frac{A_{H_i}}{x-a_{H_i}}
	\end{equation}
satisfies 	
	\begin{equation}\label{Eq_cat_S}
S_{x}(\bm{x})=\begin{pmatrix}
		a_{\hat{H}_1} I_{\hat{N}_1} & & & & & \\
		& \ddots & & & & \\
		& & a_{\hat{H}_{\hat{p}}} I_{\hat{N}_{\hat{p}}} & & &\\
		& & & a_{\hat{H}_{\hat{p}+1}}(x') I_{\hat{N}_{\hat{p}+1}} & & \\
		& & & & \ddots & \\
		& & & & & a_{\hat{H}_{\hat{q}}}(x') I_{\hat{N}_{\hat{q}}}
		\end{pmatrix}.
	\end{equation}
Here $N=\hat{N}_1+\cdots+\hat{N}_{\hat{q}}$ $(\hat{N}_j\ge0)$ is a decomposition of $N$. 
If $\hat{N}_j = 0$, we omit the corresponding diagonal block $a_{\hat{H}_{j}}(x') I_{\hat{N}_j}$ in the block diagonal matrix.
Then, we define the $\mathrm{End}_{\mathbb{C}}(\mathcal{V}^q)$-valued $1$-form $BO^x(\Omega)$ by the matrix representation
	\[
	BO^x(\Omega):=\sum_{i=1}^n\widetilde{A}_{x_i}(\bm{x})\,dx_i,
	\]
by \eqref{Eq_Pfaff_U}. 
Namely, $\widetilde{A}_{x}(\bm{x})$ is defined by the right-hand side of \eqref{Eq_U_x_sch}, that is,
	\begin{equation}
	\widetilde{A}_{x}(\bm{x})=
	S_{x}(\bm{x})^{\oplus q}
	+
	(x-T_{x}(\bm{x}))^{-1}
	\widetilde{A}
	\end{equation}
where $S_{x}(\bm{x})^{\oplus q}$ and $T_{x}(\bm{x})$ are given by \eqref{Eq_Stil} and \eqref{Eq_Ttil} respectively and 
	\[
	\widetilde{A}=A-I_{qN}.
	\]
The other coefficient matrices $\widetilde{A}_{x_i}(x)\,(i\neq \ell)$ are defined by the coefficients of the right-hand side of \eqref{Eq_U_y}. 
Then we define 
	\[
	BO^{x}(\mathcal{V},\Omega):=(\mathcal{V}^q,BO^x(\Omega)).
	\]
By definition, it holds that $BO^{x}(\mathcal{V},\Omega)\in\mathrm{Ob}(\mathscr{B}_{x})$.

For $f\in\mathrm{Hom}_{\mathscr{P}_{x}^{\text{Int}}}((\mathcal{V},\Omega),(\mathcal{V}',\Omega'))$, we define the linear map $BO^{x}(f):\mathcal{V}^q\to(\mathcal{V}')^q$ by
	\begin{equation}\label{Eq_BO_mor}
	BO^{x}(f):=f^{\oplus q}=
	\begin{pmatrix}
	f  \\
	 \vdots \\
	 f
	\end{pmatrix}.
	\end{equation}
Then it is easy to check that the map $BO^{x}(f)$ is a morphism from $BO^{x}(\mathcal{V},\Omega)$ to $BO^{x}(\mathcal{V}',\Omega')$ in $\mathscr{B}_{x}(\mathcal{T}_{x},\mathcal{S}_{x})$ (cf. Lemma \ref{lem:BOmor}).

	\begin{definition}
	We call the functor $BO^{x_\ell}:\mathscr{P}_{x_\ell}^{\text{Int}}(\mathcal{T}_{x_\ell},\mathcal{S}_{x_\ell}) \to \mathscr{B}_{x_\ell}(\mathcal{T}_{x_\ell},\mathcal{S}_{x_\ell})$ the \emph{Birkhoff-Okubo extension functor in $x_\ell$-direction}.
	\end{definition}

\subsubsection{Laplace transform functor}
\label{Subsec_LTF}

We first introduce the operation $L^{x}$, which corresponds to the Laplace transform for Birkhoff-Okubo normal forms.
On objects, it sends a connection in $\mathscr{B}_{x}(\mathcal{T}_{x},\mathcal{S}_{x})$ to a connection in $\mathscr{B}_{x}(\mathcal{S}_{x},-\mathcal{T}_{x})$. 
Here we recall that $-\mathcal{T}_{x}=\{\sigma(H)\mid H\in\mathcal{T}_{x}\}$ and $\sigma(H)$ is the hyperplane defined by the polynomial $f_{\sigma(H)}:=(f_{H})_{x}(x+a_{H})$. 

To describe the operation $L^{x}$, we consider the Laplace transform of the linear Pfaffian system whose $x$-equation is in Birkhoff-Okubo normal form. 
For $(\mathcal{V},\Omega)\in\mathscr{B}_{x}(\mathcal{T}_{x},\mathcal{S}_{x})$, we fix an isomorphism (a basis) $\mathcal{V}\cong\mathbb{C}^{N}$ satisfying the condition (B); that is, we assume the coefficient matrix $A_{x}(\bm{x})$ of $\Omega$ in $x$-direction is given in the form \eqref{Eq_Cat_Bcoef}. 
Then, let us consider the Pfaffian system
	\begin{equation}\label{Eq_app_Pfaff}
	du=\Omega u
	\end{equation}
and the Laplace transform in $x$-direction.
We shall prepare some notation. 
From the definition, the coefficient matrix $A_{x}(\bm{x})$ can be expressed as
	\begin{equation}\label{Eq_App_Cat_Bcoef}
	A_{x}(\bm{x})
	=S_{x}(\bm{x})+(x-T_{x}(\bm{x}))^{-1}A
	=S_{x}(\bm{x})+\sum_{i=1}^q\frac{A_{H_i}}{x-a_{H_i}},
	\end{equation}
where $A$ is a constant matrix, $T_{x}(\bm{x})$ is given by \eqref{eq:catT}, and $S_{x}(\bm{x})$ is a diagonal matrix.
We divide the matrices $S_{x}(\bm{x})$ and $A$ as
	\[
	S_{x}(\bm{x})=\begin{pmatrix}
		S^x_1(\bm{x}) & & & & & \\
		& \ddots & & & & \\
		& & S^x_{p}(\bm{x}) & & &\\
		& & & S^x_{p+1}(\bm{x}) & & \\
		& & & & \ddots & \\
		& & & & & S^x_{q}(\bm{x})
		\end{pmatrix},
	\quad 
	A=\begin{pmatrix}
	A_{11}& A_{12}& \cdots & A_{1q}\\
	A_{21}& A_{22}& \cdots & A_{2q}\\
	\vdots & \vdots & &\vdots \\
	A_{q1} & A_{q2} & \cdots & A_{qq}
	\end{pmatrix}.
	\]
where $S_j^x(\bm{x})$ is an $N_j \times N_j$ diagonal matrix ($j=1\ldots,q$) and $A_{ij}$ is an $N_i\times N_j$ matrix ($1\le i,j \le q$).
Then, the matrix $A_{H_i}$ ($1\le i\le q$) can be expressed as
	\[
	A_{H_i}=
	\begin{pmatrix}
	 & O & \\
	A_{i1}& \cdots & A_{iq}\\
	 & O &   
	\end{pmatrix}
	(i.
	\]
We take any $y\in x'=\{x_1,\ldots,x_n\}\setminus\{x\}$ and consider the $y$-equation 
	\begin{equation}\label{eq:funcLyeq}
	\frac{\p u}{\p y}=A_{y}(\bm{x})=\left(S_{y}(\bm{x})+\sum_{H\in \mathcal{A}_y}\frac{A_{H}}{y-b_{H}}\right)u
	\end{equation}
of \eqref{Eq_app_Pfaff}.
We also divide the matrices $S_{y}(\bm{x})$ and $\{A_{H}\}_{H\in\mathcal{A}_y}$ as
	\[
	S_{y}(\bm{x})=\begin{pmatrix}
		S^y_{11}(\bm{x}) &S^y_{12}(\bm{x}) & \cdots & S^y_{1q}(\bm{x}) \\
		S^y_{21}(\bm{x}) &S^y_{22}(\bm{x}) & \cdots & S^y_{2q}(\bm{x}) \\
		\vdots &\vdots&  & \vdots \\
		S^y_{q1}(\bm{x}) &S^y_{q2}(\bm{x}) & \cdots & S^y_{qq}(\bm{x}) \\
		\end{pmatrix},
	\quad 
	A_{H}=\begin{pmatrix}
	A^H_{11}& A^H_{12}& \cdots & A^H_{1q}\\
	A^H_{21}& A^H_{22}& \cdots & A^H_{2q}\\
	\vdots & \vdots & &\vdots \\
	A^H_{q1} & A^H_{q2} & \cdots & A^H_{qq}
	\end{pmatrix}
	\]
where $S^{y}_{ij}(\bm{x})$ and $A_{ij}^{H}$ are $N_i\times N_j$ matrices ($1\le i,j \le q$). 
Lastly, we also divide the vector $u$ as 
	\[
	u={}^t(u_1,\dots,u_q),
	\] 
where the size of $u_i$ is $N_i$ ($1\le i\le q$).
Then, we consider applying the Laplace transform in $x$-direction \eqref{Eq_op_FL} for the Pfaffian system \eqref{Eq_app_Pfaff}.
Set
	\begin{equation}\label{Eq_app_divide_v}
	v={}^t(v_1,\ldots,v_q),\quad 
	v_i=L^{x}(u_i).
	\end{equation}
Then, the $x$-equation of the Pfaffian system is transformed into 
	\[
	\frac{\p v}{\p x}=\widehat{A}_{x}(\bm{x})v,
	\]
where $\widehat{A}_{x}(\bm{x})$ is given by
	\begin{equation}\label{Eq_Lfunc_coef_x}
	\widehat{A}_{x}(\bm{x}):=-T_{x}(\bm{x})-(x-S_{x}(\bm{x}))^{-1}(A+I).
	\end{equation}
On the transformation of $y$-equation \eqref{eq:funcLyeq} ($y\in x'$), in a similar way to Proposition \ref{Prop_Laplace}, we can show the following. 
	\begin{proposition}
	By the Laplace transform in $x$-direction \eqref{Eq_op_FL}, the $y$-equation \eqref{eq:funcLyeq} ($y\in x'$) of \eqref{Eq_app_Pfaff} is transformed into
	\begin{equation}\label{Eq_app_prop}
	\begin{aligned}
	\frac{\p v_i}{\p y}=&\frac{\p}{\p y}\left[(x-S^x_i(\bm{x}))(x-a_{H_i})\right]v_i+\sum_{j=1}^q S^y_{ij}v_j+\frac{\p S^x_i}{\p y}(x-S^x_i(\bm{x}))^{-1}\sum_{k=1}^qA'_{ik}v_k
	\\ &+\sum_{H\in\mathcal{A}_x^c\cap\mathcal{A}_y}\frac{1}{y-b_H}\sum_{j=1}^qA^H_{ij}v_j \quad (i=1,\ldots,q).
	\end{aligned}
	\end{equation}
Here we set 
	\[
	A+I=\begin{pmatrix}
	A_{11}'& A_{12}'& \cdots & A_{1q}'\\
	A_{21}'& A_{22}'& \cdots & A_{2q}'\\
	\vdots & \vdots & &\vdots \\
	A_{q1}' & A_{q2}' & \cdots & A_{qq}'
	\end{pmatrix}.
	\]
	\end{proposition}
Summarizing the equations \eqref{Eq_app_prop}, we have the $y$-equation 
		\begin{equation}\label{Eq_app_yeq}
	\frac{\p v}{\p y}=\widehat{A}_{y}(\bm{x})v
	\end{equation}
of the Laplace transform of \eqref{Eq_app_Pfaff} in $x$-direction.
\begin{remark}
The inverse Laplace transform is omitted, as it can be considered in the same way as the Laplace transform.
\end{remark}

We are now ready to define the operation $L^{x}$. 
For $(\mathcal{V},\Omega)\in\mathscr{B}_{x}(\mathcal{T}_{x},\mathcal{S}_{x})$, we fix an isomorphism $\mathcal{V}\cong\mathbb{C}^{N}$ satisfying the condition (B); that is, we assume the coefficient matrix $A_{x}(\bm{x})$ of $\Omega$ in $x$-direction is given in the form \eqref{Eq_Cat_Bcoef}. 
For this connection, we set the $\mathrm{End}_{\mathbb{C}}(\mathcal{V})$-valued $1$-form $L^x(\Omega)$ by the matrix representation
	\begin{equation}\label{Eq_Lfunc_coef}
	L^x(\Omega)=\sum_{i=1}^n\widehat{A}_{x_i}(\bm{x})\,dx_i,
	\end{equation}
where $\widehat{A}_{x_\ell}(\bm{x})=\widehat{A}_x(\bm{x})$ is given by \eqref{Eq_Lfunc_coef_x} and the other coefficient matrices $\widehat{A}_{x_i}(\bm{x})\,(i\neq \ell)$ are defined by the coefficient matrices of the right-hand side of \eqref{Eq_app_yeq}. 
Then, we define
	\[
	L^{x}(\mathcal{V},\Omega):=(\mathcal{V},L^x(\Omega)).
	\]
Note that $(\mathcal{V},L^x(\Omega))$ is an object of $\mathscr{B}_{x}(\mathcal{S}_{x},-\mathcal{T}_{x})$.

For $f\in\mathrm{Hom}_{\mathscr{B}_{x}(\mathcal{T}_{x},\mathcal{S}_{x})}((\mathcal{V},\Omega),(\mathcal{V}',\Omega'))$, we define the linear map $L^{x}(f):\mathcal{V}\to\mathcal{V}'$ by
	\[
	L^{x}(f):=f.
	\]
Then it is easy to check that the map $L^{x}(f)$ is a morphism from $L^{x}(\mathcal{V},\Omega)$ to $L^{x}(\mathcal{V}',\Omega')$ in $\mathscr{B}_{x}(\mathcal{S}_{x},-\mathcal{T}_{x})$.

We define the operation $L^{-x}$ in the same way.
That is, 
	\[
	L^{-x}(\mathcal{V},\Omega):=(\mathcal{V},L^{-x}(\Omega)), 
	\quad 
	L^{-x}(f):=f,
	\]
where $L^{-x}(\Omega)$ is the $\mathrm{End}_{\mathbb{C}}(\mathcal{V})$-valued $1$-form defined by the matrix representation corresponding to the inverse Laplace transform of $du=\Omega u$ in $x$-direction.

\smallskip

Now we define the functors corresponding to the (inverse) Laplace transform for linear Pfaffian systems in Definitions \ref{Def_L} and \ref{Def_invL}.
Let $I:\mathscr{B}_{x}\to\mathscr{P}_{x}$ denote the inclusion functor. 
Then, from Remarks \ref{rem:integrableV} and \ref{rem:integrableW}, we can verify that the connections 
	$I\circ L^{x}\circ BO^{x}(\calV,\Omega)$ and $I\circ L^{-x}\circ BO^{x}(\calV,\Omega)$ are both flat for any $(\calV,\Omega)\in\scrP^{\text{Int}}_{x}(\calT_{x},\calS_{x})$.
Therefore, the compositions $I\circ L^{x}\circ BO^{x}$ and $I\circ L^{-x}\circ BO^{x}$ define the functors
\[
\mathscr{P}^{\text{Int}}_{x}(\mathcal{T}_{x},\mathcal{S}_{x})\to\mathscr{P}^{\text{Int}}_{x}(\mathcal{S}_{x},-\mathcal{T}_{x}),
\quad 
\mathscr{P}^{\text{Int}}_{x}(\mathcal{T}_{x},\mathcal{S}_{x})\to\mathscr{P}^{\text{Int}}_{x}(-\mathcal{S}_{x},\mathcal{T}_{x}),
\]
respectively. 
Then, we define the (inverse) Laplace transform functor as follows.

	\begin{definition}
We define the \emph{Laplace transform functor in $x$-direction} $\mathcal{L}^{x}: \mathscr{P}_{x}^{\text{Int}}(\mathcal{T}_{x},\mathcal{S}_{x})\to \mathscr{P}_{x}^{\text{Int}}(\mathcal{S}_{x},-\mathcal{T}_{x})$ by
	\[
	\mathcal{L}^{x}:=I\circ L^{x}\circ BO^{x}.
	\]
For $(\mathcal{V},\Omega)\in\mathscr{P}_{x}^{\text{Int}}(\mathcal{T}_{x},\mathcal{S}_{x})$, we write
	\[
	\mathcal{L}^{x}(\mathcal{V},\Omega)=(\mathcal{V}^q,\mathcal{L}^{x}(\Omega)).
	\]
Similarly, we define the \emph{inverse Laplace transform functor in $x$-direction}
$\mathcal{L}^{-x}: \mathscr{P}_{x}^{\text{Int}}(\mathcal{T}_{x},\mathcal{S}_{x})\to \mathscr{P}_{x}^{\text{Int}}(-\mathcal{S}_{x},\mathcal{T}_{x})$ by
	\[
	\mathcal{L}^{-x}:=I\circ L^{-x}\circ BO^{x}.
	\]
For $(\mathcal{V},\Omega)\in\mathscr{P}_{x}^{\text{Int}}(\mathcal{T}_{x},\mathcal{S}_{x})$, we write
	\[
	\mathcal{L}^{-x}(\mathcal{V},\Omega)=(\mathcal{V}^q,\mathcal{L}^{-x}(\Omega)).
	\]
	\end{definition}

	\begin{remark}\label{Rem_L_Exact}
\begin{enumerate}[(i)]
	\item The $1$-forms $\mathcal{L}^{x}(\Omega)$ and $\mathcal{L}^{-x}(\Omega)$ are given by \eqref{Eq_Pfaff_V} and \eqref{Eq_Pfaff_W}, respectively. 
	\item The functors $\mathcal{L}^{\pm x}$ are both exact.
This can be verified in the same way as in Proposition \ref{prop:odefuncLexact} (the one variable case).
\end{enumerate}
\end{remark}

\subsubsection{Middle Laplace transform functor}

For $(\mathcal{V},\Omega)\in\mathscr{P}_{x}^{\text{Int}}(\mathcal{T}_{x},\mathcal{S}_{x})$ with \eqref{Eq_conn_Omega}, we set the subspace $\mathcal{K}_{\mathcal{V}}^{x}$ of $\mathcal{V}^q$ as
	\begin{equation}\label{Eq:Ktwo}
	\mathcal{K}_{\mathcal{V}}^{x}:=\bigoplus_{i=1}^{q}\Ker A_{H_i},
	\end{equation}
which corresponds to \eqref{Eq_K}. 
The subspace $\mathcal{K}_{\mathcal{V}}^x$ is invariant under the action of the $1$-form $\mathcal{L}^x(\Omega)$, which can be checked similarly in the previous section. 
Hence, we can take the restriction and the quotient:
	\[
	\mathcal{K}^x(\mathcal{V},\Omega):=(\mathcal{K}_{\mathcal{V}}^x,\mathcal{L}^x(\Omega)\vert_{\mathcal{K}_{\mathcal{V}}^{x}}), 
	\quad 
	\mathcal{ML}^x(\mathcal{V},\Omega):=(\mathcal{V}^q/\mathcal{K}_{\mathcal{V}}^x,\mathcal{L}^x(\Omega)\vert_{\mathcal{V}^q/\mathcal{K}_{\mathcal{V}}^{x}}). 
	\]
Since $\mathcal{K}_{\mathcal{V}}^x$ is also invariant under the action of the $1$-form $\mathcal{L}^{-x}(\Omega)$, we can take the restriction and the quotient:
	\[
	\mathcal{K}^{-x}(\mathcal{V},\Omega):=(\mathcal{K}_{\mathcal{V}}^x,\mathcal{L}^{-x}(\Omega)\vert_{\mathcal{K}_{\mathcal{V}}^{x}}), 
	\quad 
	\mathcal{ML}^{-x}(\mathcal{V},\Omega):=(\mathcal{V}^q/\mathcal{K}_{\mathcal{V}}^x,\mathcal{L}^{-x}(\Omega)\vert_{\mathcal{V}^q/\mathcal{K}_{\mathcal{V}}^{x}}). 
	\]
We note that $\mathcal{K}^{x}$ and $\mathcal{ML}^{x}$ (resp. $\mathcal{K}^{-x}$ and $\mathcal{ML}^{-x}$) are the functors from $\mathscr{P}_{x}^{\text{Int}}(\mathcal{T}_{x},\mathcal{S}_x)$ to $\mathscr{P}_x^{\text{Int}}(\mathcal{S}_x,-\mathcal{T}_x)$ (resp. $\mathscr{P}_x^{\text{Int}}(-\mathcal{S}_x,\mathcal{T}_x)$).
	\begin{definition}
	We call the functors $\mathcal{ML}^{x_\ell}$ and $\mathcal{ML}^{-x_\ell}$ the \emph{middle Laplace transform functor in $x_\ell$-direction} and the \emph{inverse middle Laplace transform functor in $x_\ell$-direction}, respectively.
	\end{definition}
		\begin{remark}\label{Rem_L_Exact_2}
The $1$-forms $\mathcal{L}^{x_{\ell}}(\Omega)\vert_{\mathcal{V}^q/\mathcal{K}_{\mathcal{V}}^{x}}$ and $\mathcal{L}^{-x_{\ell}}(\Omega)\vert_{\mathcal{V}^q/\mathcal{K}_{\mathcal{V}}^{x}}$ are given by $\mathcal{ML}^{x_\ell}(\Omega)$ and $\mathcal{ML}^{-x_\ell}(\Omega)$ in \eqref{Eq_Pfaff_v} and \eqref{Eq_Pfaff_w}, respectively. 
Hence, we write $\mathcal{L}^{x_{\ell}}(\Omega)\vert_{\mathcal{V}^q/\mathcal{K}_{\mathcal{V}}^{x}}=\mathcal{ML}^{x_\ell}(\Omega)$ and $\mathcal{L}^{-x_{\ell}}(\Omega)\vert_{\mathcal{V}^q/\mathcal{K}_{\mathcal{V}}^{x}}=\mathcal{ML}^{-x_\ell}(\Omega)$ by abuse of notation.
\end{remark}
	\begin{remark}
	$\mathcal{ML}^{x_\ell}$ can be interpreted as follows.
	Since the inclusion map $\iota:\mathcal{K}_{\mathcal{V}}^{x_\ell}\hookrightarrow \mathcal{V}^q$ induces a morphism in $\mathscr{P}_{x_\ell}^{\text{\rm Int}}(\mathcal{S}_{x_\ell},-\mathcal{T}_{x_\ell})$, then we obtain an exact sequence 
	\begin{equation}\label{Eq_exact_ML}
		\begin{tikzcd}
		0\arrow[r]&\mathcal{K}^{x_\ell}(\mathcal{V},\Omega)\arrow[r,"\iota",hookrightarrow]
		&\mathcal{L}^{x_\ell}(\mathcal{V},\Omega) \arrow[r]
		&\Coker \iota\arrow[r]&0
		\end{tikzcd}
	\end{equation}
	in $\mathscr{P}_{x_\ell}^\text{\rm Int}(\mathcal{S}_{x_\ell},-\mathcal{T}_{x_\ell})$. 
	Then, we have $\mathcal{ML}^{x_\ell}(\mathcal{V},\Omega)=\Coker \iota$.
	Similarly, we obtain an exact sequence
	\begin{equation}\label{Eq_exact_invML}
		\begin{tikzcd}
		0\arrow[r]&\mathcal{K}^{-x_\ell}(\mathcal{V},\Omega)\arrow[r,"\iota",hookrightarrow]
		&\mathcal{L}^{-x_\ell}(\mathcal{V},\Omega) \arrow[r]
		&\Coker \iota\arrow[r]&0
		\end{tikzcd}
	\end{equation}
	in $\mathscr{P}_{x_\ell}^{\text{\rm Int}}(-\mathcal{S}_{x_\ell},\mathcal{T}_{x_\ell})$. 
	Then we have $\mathcal{ML}^{-x_\ell}(\mathcal{V},\Omega)=\Coker \iota$.
	\end{remark}

\section{Inversion formula and irreducibility (several variables case)}
\label{Sec_Inv}
Retain the notation of the previous section.
In this section, we show Theorem \ref{Thm_main}.
First, we introduce the notions of irreducibility and exceptionality for connections.
	\begin{definition}\label{Def_conn_irr}
	Let $(\mathcal{V},\Omega)\in\mathscr{P}_{x_\ell}^{\text{\rm Int}}(\mathcal{T}_{x_\ell},\mathcal{S}_{x_\ell})$ be a connection with $\calV\neq\{0\}$, whose $1$-form $\Omega$ is of the form \eqref{Eq_conn_Omega}.
	\begin{itemize}
	\item We say that $(\mathcal{V},\Omega)$ is \emph{irreducible in $x_\ell$-direction} if there is no $\langle A_{x_\ell},\{A_{x_\ell x_j}\}_{j\neq \ell},\{A_{H}\}_{H\in\mathcal{A}_{x_\ell}}\rangle$-invariant subspace of $\calV$ except $\mathcal{V}$ and $\{0\}$. 
	\item We say that $(\mathcal{V},\Omega)$ is \emph{exceptional in $x_{\ell}$-direction} if $\dim \calV=1$ and $A_{H}=0$ for all $H\in\calA_{x_{\ell}}$. 
	Otherwise, we say it is non-exceptional in \(x_\ell\)-direction.
	\end{itemize}
	\end{definition}
Clearly, a connection $(\mathcal{V},\Omega)\in\mathscr{P}_{x_\ell}^{\text{\rm Int}}(\mathcal{T}_{x_\ell},\mathcal{S}_{x_\ell})$ is irreducible and non-exceptional in $x_\ell$-direction if and only if the corresponding Pfaffian system $du=\Omega u$ is irreducible and non-exceptional in $x_\ell$-direction (Definition \ref{Def_irr}).
	\begin{remark}
	As in Remark \ref{rem:odeDR}, if a connection $(\mathcal{V},\Omega)\in\mathscr{P}_{x_\ell}^{\text{\rm Int}}(\mathcal{T}_{x_\ell},\mathcal{S}_{x_\ell})$ with the $1$-form $\Omega$ of the form \eqref{Eq_conn_Omega} is irreducible and non-exceptional in $x_\ell$-direction, then it holds that 
	\begin{align}
	&\bigcap_{H\in\mathcal{A}_{x_\ell}}\Ker A_{H} \cap \Ker (A_{x_\ell}+c_{x_\ell})\cap\bigcap_{j\neq \ell} \Ker (A_{x_\ell x_j}+c_{x_\ell x_j})=\{0\}	\quad (\forall c_{x_\ell}, c_{x_\ell x_j}\in\mathbb{C})\tag{$\star_\ell$},\label{Eq_DR_star1val}\\
	&\sum_{H\in\mathcal{A}_{x_\ell}} \im A_{H}+\im (A_{x_\ell}+c_{x_\ell})+\sum_{j\neq \ell}\im (A_{x_\ell x_j}+c_{x_\ell x_j})=\mathcal{V} \quad (\forall c_{x_\ell}, c_{x_\ell x_j}\in\mathbb{C}).	\label{Eq_DR_star2val}\tag{$\star\star_\ell$}
	\end{align}
	\end{remark}

Then, Theorem \ref{Thm_main} can be rephrased as follows. 
\begin{theorem}\label{Thm_conn}
Suppose that $(\mathcal{V},\Omega)\in\mathscr{P}_{x_\ell}^{\text{\rm Int}}(\mathcal{T}_{x_\ell},\mathcal{S}_{x_\ell})$ is irreducible and non-exceptional in $x_\ell$-direction. 
Then the following hold.
	\begin{enumerate}[(i)]
	\item $\mathcal{ML}^{-x_\ell}\circ\mathcal{ML}^{x_\ell}(\mathcal{V},\Omega)\sim(\mathcal{V},\Omega)$ and 
	$\mathcal{ML}^{x_\ell}\circ\mathcal{ML}^{-x_\ell}(\mathcal{V},\Omega)\sim(\mathcal{V},\Omega)$. 

	\item Both $\mathcal{ML}^{x_\ell}(\mathcal{V},\Omega)$ and $\mathcal{ML}^{-x_\ell}(\mathcal{V},\Omega)$ are irreducible and non-exceptional in $x_\ell$-direction.
		\end{enumerate}
	Here the symbol $\sim$ means isomorphic as connections (see Remark \ref{Rem_equiv_conn}).
\end{theorem}
We note that the above statement (ii) is essentially the same as the statement of Theorem \ref{thm:odeconn} (ii). 
Therefore, in the following, we shall only show the inversion formula (Theorem \ref{Thm_conn} (i)).

\subsection{A proof of inversion formula}\label{Subsec_inv_form}
This subsection is devoted to proving the inversion formula (Theorem \ref{Thm_conn} (i)).
We only show the statement $\mathcal{ML}^{-x_\ell}\circ\mathcal{ML}^{x_\ell}(\mathcal{V},\Omega)\sim(\mathcal{V},\Omega)$. 
Take any $\ell$ ($1\le \ell\le n)$ and write $x_\ell=x$. 
We set $\mathcal{T}_{x}$ and $\mathcal{S}_{x}$ as in \eqref{Eq_conn_hs}. 
Suppose that $(\mathcal{V},\Omega)\in \mathscr{P}_x^{\text{Int}}(\mathcal{T}_x,\mathcal{S}_x)$ consisting of an $N$-dimensional vector space $\mathcal{V}$ and a $1$-form of the form \eqref{Eq_conn_Omega} is irreducible and non-exceptional in $x$-direction. 
Let $\calS_{x}'$ be the subarrangement of $\mathcal{S}_{x}$ such that 
	\[
	\{\bm{x}\in\mathbb{C}^n\mid\det(x-S_{x}(\bm{x}))=0\}=\bigcup_{\hat{H}\in \calS'_{x}}\hat{H}.
	\] 
Then we have $(\calV,\Omega)\in\scrP_x^{\text{Int}}(\calT_{x},\calS_{x}')$.
Since $\calS_{x}'\subseteq\calS_{x}$, the category $\scrP_x^{\text{Int}}(\calT_{x},\calS_{x}')$ is a full subcategory of $\scrP_x^{\text{Int}}(\calT_{x},\calS_{x})$. 
Therefore, for the proof of the inversion formula, it suffices to prove the assertion after replacing $\calS_{x}$ by $\calS_{x}'$.
After this replacement, using the same notation $\calS_{x}=\{\hat{H}_{1},\ldots,\hat{H}_{\hat{q}}\}$, we may assume in \eqref{Eq_cat_S} that $\hat{N}_j\ge1$ for all $j=1,2,\ldots,\hat{q}$.

Our goal is to show the existence of the isomorphism $\mathcal{ML}^{-x}(\mathcal{ML}^{x}(\mathcal{V},\Omega))\to(\mathcal{V},\Omega)$ as connections. 
The strategy of the proof is essentially the same as in the one variable case. Nevertheless, for completeness, we provide a brief explanation.

First, we shall describe $(\mathcal{V},\Omega)$ as a quotient connection of $\mathcal{L}^{-x}(\mathcal{L}^{x}(\mathcal{V},\Omega))$. 
In the following, we fix an isomorphism (a basis) $\mathcal{V}\cong\mathbb{C}^N$ as in Section \ref{Subsubsec_BOf}, i.e., the coefficient matrix $A_{x}(\bm{x})\in\mathrm{Mat}(N,\mathbb{C}(\bm{x}))$ of $\Omega$ in $x$-direction is given by \eqref{Eq_conn_Axi} with \eqref{Eq_cat_S}.
Then, by definition of
	\begin{equation}\label{Eq_LL}
	\mathcal{L}^{-x}(\mathcal{L}^{x}(\mathcal{V},\Omega))=((\mathcal{V}^q)^{\hat{q}},\mathcal{L}^{-x}(\mathcal{L}^{x}(\Omega))),
	\end{equation}
we can fix an isomorphism (a basis) $(\mathcal{V}^q)^{\hat{q}}\cong(\mathbb{C}^{Nq})^{\hat{q}}$ such that the coefficient matrices of $\mathcal{L}^{-x}(\mathcal{L}^{x}(\Omega))$ are obtained by iterating the procedures explained in Sections \ref{subsec:deftwo} and \ref{Subsec_IFL} for $\Omega$ (The actual form of the coefficient matrices will be given later). 

Then, we define a linear map $\varphi:(\mathcal{V}^q)^{\hat{q}}\to\mathcal{V}$ as follows:
for $\bm{v}={}^t(v_1,\ldots,v_{\hat{q}})\in(\mathcal{V}^q)^{\hat{q}}$ with $v_i={}^t(v_1^i,\ldots,v_q^i)\in\mathcal{V}^q$, we define
	\begin{equation}\label{Eq_varphi}
	\varphi(\bm{v}):=\sum_{i=1}^{\hat{q}}E_{N_i} \sum_{j=1}^q A_{H_j}v_j^i.
	\end{equation}
This map is an analog of the linear map \eqref{eq:odevarphi} introduced in the one variable case.
Then, a natural generalization of Proposition \ref{prop:odevarphimor} holds.
\begin{proposition}
\label{Prop_varphi_mor}
	The linear map \eqref{Eq_varphi} defines a morphism of connections 
	\begin{equation}\label{eq:varphiasmor}
	\varphi: \mathcal{L}^{-x}(\mathcal{L}^{x}(\mathcal{V},\Omega))\to(\mathcal{V},\Omega).
	\end{equation}
\end{proposition}
The proof of this proposition is a bit lengthy, so we postpone it to Section \ref{Subsec_varphi_mor}. 
By using \eqref{Eq_DR_star2val}, in the same way as Lemma \ref{lem:odevarphi}, we can show the following.
\begin{lemma}\label{Lem_varphi}
The linear map \eqref{Eq_varphi} satisfies 
	\[
	\Ker \varphi=\mathcal{K}', 
	\quad 
	\im \varphi=\mathcal{V}.
	\]
Here 
	\begin{equation}\label{Eq_K'}
	\mathcal{K}':=
	\left\{
	\begin{pmatrix} 
	v_1 \\
	\vdots \\
	v_{\hat{q}}
	\end{pmatrix}
	\in (\mathcal{V}^q)^{\hat{q}}
	~\middle \vert~
	\begin{array}{l} 
	v_i={}^t(v_1^i,\ldots,v_q^i)\in\mathcal{V}^q
	\text{ satisfies}
	\\\ds
	E_{N_i}\sum_{j=1}^q A_{H_j}v_j^i=0
	\quad
	(i=1,\ldots,\hat{q})
	\end{array}
	\right\}.
	\end{equation}
\end{lemma}
As a consequence of this lemma and Proposition \ref{Prop_varphi_mor}, we have the following. 
\begin{corollary}
The connection $\mathcal{K}'(\mathcal{V},\Omega):=(\mathcal{K}',\mathcal{L}^{-x}(\mathcal{L}^{x}(\Omega))\vert_{\mathcal{K}'})$ is a subconnection of $\mathcal{L}^{-x}(\mathcal{L}^{x}(\mathcal{V},\Omega))$, 
and the morphism \eqref{eq:varphiasmor} induces the isomorphism
	\begin{equation}\label{Eq_mor_phi}
	\bar{\varphi}: \mathcal{L}^{-x}(\mathcal{L}^{x}(\mathcal{V},\Omega))/\mathcal{K}'(\mathcal{V},\Omega)\stackrel{\sim}{\longrightarrow}(\mathcal{V},\Omega).
	\end{equation}
\end{corollary}
Here we note that
	\[
	\mathcal{L}^{-x}(\mathcal{L}^{x}(\mathcal{V},\Omega))/\mathcal{K}'(\mathcal{V},\Omega)=
	\left((\mathcal{V}^q)^{\hat{q}}/\mathcal{K}',
	\mathcal{L}^{-x}(\mathcal{L}^{x}(\Omega))\Big\vert_{(\mathcal{V}^q)^{\hat{q}}/\mathcal{K}'}\right).
	\]

To show $\mathcal{L}^{-x}(\mathcal{L}^{x}(\mathcal{V},\Omega))/\mathcal{K}'(\mathcal{V},\Omega)\sim\mathcal{ML}^{-x}(\mathcal{ML}^{x}(\mathcal{V},\Omega))$, we prepare a lemma concerning the expression of $\mathcal{K}'$.
This can be shown in the same way as in Lemma \ref{lem:odeK'}, by using \eqref{Eq_DR_star1val}.
\begin{lemma}\label{Lem_K'}
The vector space $\mathcal{K}'$ in \eqref{Eq_K'} can be expressed as
	\begin{equation}\label{Eq_K'2}
	\mathcal{K}'
	=\left\{
	\begin{pmatrix} 
	v_1 \\
	\vdots \\
	v_{\hat{q}}
	\end{pmatrix}
	\in (\mathcal{V}^q)^{\hat{q}}
	~\middle \vert~
	B_{\hat{H}_i}v_i \in \mathcal{K}_{\mathcal{V}}^{x}	
	\quad (i=1,\ldots,\hat{q})
	\right\}
	\end{equation}
where $B_{\hat{H}_i}$ and $\mathcal{K}_{\mathcal{V}}^{x}$ are given by \eqref{Eq_hat_G} and \eqref{Eq:Ktwo}, respectively.
\end{lemma}

Let us show $\mathcal{L}^{-x}(\mathcal{L}^{x}(\mathcal{V},\Omega))/\mathcal{K}'(\mathcal{V},\Omega)\sim\mathcal{ML}^{-x}(\mathcal{ML}^{x}(\mathcal{V},\Omega))$.
To do this, we first construct a morphism from $\mathcal{L}^{-x}(\mathcal{L}^x(\mathcal{V},\Omega))$ to $\mathcal{L}^{-x}(\mathcal{ML}^{x}(\mathcal{V},\Omega))$.
Let $f_{*}:\mathcal{L}^x(\mathcal{V},\Omega)\to\mathcal{ML}^{x}(\mathcal{V},\Omega)$ be a quotient morphism. 
It is nothing but the quotient map from $\mathcal{V}^q$ to $\mathcal{V}^q/\mathcal{K}_{\mathcal{V}}^{x}$ as a linear map. 
Then, by definition, we see that the morphism
$f:=\mathcal{L}^{-x}(f_{*})$ from $\mathcal{L}^{-x}(\mathcal{L}^x(\mathcal{V},\Omega))$ to $\mathcal{L}^{-x}(\mathcal{ML}^{x}(\mathcal{V},\Omega))$ is given by the linear map
	\begin{equation}\label{Eq_mor_f}
	\begin{array}{ccccc}
	f:&(\mathcal{V}^q)^{\hat{q}}& \to & (\mathcal{V}^q/\mathcal{K}_{\mathcal{V}}^{x})^{\hat{q}}& \\
	& \rotatebox{90}{$\in$}  & & \rotatebox{90}{$\in$} & \\
	& \begin{pmatrix}v_1 \\
	\vdots \\
	v_{\hat{q}}
	\end{pmatrix} & \mapsto & 
	\begin{pmatrix}\bar{v}_1 \\
	\vdots \\
	\bar{v}_{\hat{q}}
	\end{pmatrix}, &\bar{v}_i=v_i+\mathcal{K}^x_{\mathcal{V}}.
	\end{array}
	\end{equation}
We shall construct a morphism $f'$ from $\mathcal{L}^{-x}(\mathcal{L}^{x}(\mathcal{V},\Omega))/\mathcal{K}'(\mathcal{V},\Omega)$ to $\mathcal{ML}^{-x}(\mathcal{ML}^{x}(\mathcal{V},\Omega))$. 
Let $\pi:\mathcal{L}^{-x}(\mathcal{ML}^{x}(\mathcal{V},\Omega))\to\mathcal{ML}^{-x}(\mathcal{ML}^{x}(\mathcal{V},\Omega))$ be the quotient morphism.
As a linear map, the quotient morphism $\pi$ is nothing but 
\[
\pi:(\mathcal{V}^q/\mathcal{K}_{\mathcal{V}}^{x})^{\hat{q}} \to (\mathcal{V}^q/\mathcal{K}_{\mathcal{V}}^{x})^{\hat{q}}\left/\bigoplus_{i=1}^{\hat{q}}\Ker \bar{B}_{\hat{H}_i}, \right.
\]
where $\bar{B}_{\hat{H}_i}$ are the coefficient matrices of $\mathcal{ML}^{x}(\Omega)$ in \eqref{Eq_Pfaff_v}.
Therefore, the composition $\pi\circ f:\mathcal{L}^{-x}(\mathcal{L}^{x}(\mathcal{V},\Omega))\to\mathcal{ML}^{-x}(\mathcal{ML}^{x}(\mathcal{V},\Omega))$ is given by the linear map
	\[
	\begin{array}{ccccc}
	\pi\circ f:&(\mathcal{V}^q)^{\hat{q}}& \to & \ds(\mathcal{V}^q/\mathcal{K}_{\mathcal{V}}^{x})^{\hat{q}}\left/\bigoplus_{i=1}^{\hat{q}}\Ker \bar{B}_{\hat{H}_i} \right.& \\
	& \rotatebox{90}{$\in$}  & & \rotatebox{90}{$\in$} & \\
	& \begin{pmatrix}
	v_1 \\ 
	\vdots \\
	v_{\hat{q}}
	\end{pmatrix} & \mapsto & 
	\begin{pmatrix}
	\bar{v}_1+\Ker \bar{B}_{\hat{H}_1} \\
	\vdots \\
	\bar{v}_{\hat{q}}+\Ker \bar{B}_{\hat{H}_{\hat{q}}}
	\end{pmatrix}, & 
	\bar{v}_i=v_i+\mathcal{K}_{\mathcal{V}}^{x}.
	\end{array}
	\]
For $\bm{v},\bm{w} \in (\mathcal{V}^q)^{\hat{q}}$ satisfying $\bm{v}-\bm{w}\in\mathcal{K}'$, it can be easily checked
	\[
	(\pi\circ f)(\bm{v})=(\pi\circ f)(\bm{w})
	\]
by using Lemma \ref{Lem_K'}.
Hence, the linear map $\pi\circ f$ induces the linear map
	\begin{equation}\label{Eq_f'}
	\begin{array}{ccccc}
	f':&(\mathcal{V}^q)^{\hat{q}}/\mathcal{K}'& \to &\ds (\mathcal{V}^q/\mathcal{K}_{\mathcal{V}}^{x})^{\hat{q}}\left/\bigoplus_{i=1}^{\hat{q}}\Ker \bar{B}_{\hat{H}_i} \right.& \\
	& \rotatebox{90}{$\in$}  & & \rotatebox{90}{$\in$} & \\
	& \begin{pmatrix}
	v_1 \\ 
	\vdots \\
	v_{\hat{q}}
	\end{pmatrix}+\mathcal{K}' & \mapsto & 
	\begin{pmatrix}
	\bar{v}_1+\Ker \bar{B}_{\hat{H}_1} \\
	\vdots \\
	\bar{v}_{\hat{q}}+\Ker \bar{B}_{\hat{H}_{\hat{q}}}
	\end{pmatrix}, & 
	\bar{v}_i=v_i+\mathcal{K}_{\mathcal{V}}^{x}.
	\end{array}
	\end{equation}
	%
We can easily check that the linear map $f'$ becomes a morphism from $\mathcal{L}^{-x}(\mathcal{L}^{x}(\mathcal{V},\Omega))/\mathcal{K}'(\mathcal{V},\Omega)$ to $\mathcal{ML}^{-x}(\mathcal{ML}^{x}(\mathcal{V},\Omega))$ and the diagram
	\begin{equation}\label{Eq_com_MLM}
		\begin{tikzcd}
		\mathcal{L}^{-x}(\mathcal{L}^{x}(\mathcal{V},\Omega)) \arrow[r,"\pi'"]
		\arrow[d,"f "']
		&\mathcal{L}^{-x}(\mathcal{L}^{x}(\mathcal{V},\Omega))/\mathcal{K}'(\mathcal{V},\Omega)\arrow[d,"f'"]\\
		 \mathcal{L}^{-x}(\mathcal{ML}^x(\mathcal{V},\Omega))  \arrow[r,"\pi"]
		&\mathcal{ML}^{-x}(\mathcal{ML}^{x}(\mathcal{V},\Omega))
	\end{tikzcd}
	\end{equation}
is commutative.
Here $\pi'$ denotes the quotient morphism.
Then, in the same way as Proposition \ref{Prop_bijc}, we have the following.
\begin{proposition}\label{Prop:bijctwo}
The linear map $f'$ in \eqref{Eq_f'} is bijective.
\end{proposition}
As a consequence, we have an isomorphism of connections
	\begin{equation}
	\bar{\varphi}\circ (f')^{-1}:\mathcal{ML}^{-x}(\mathcal{ML}^{x}(\mathcal{V},\Omega))\stackrel{\sim}{\longrightarrow}(\mathcal{V},\Omega),
	\end{equation}
which shows $\mathcal{ML}^{-x}\circ\mathcal{ML}^{x}(\mathcal{V},\Omega)\sim(\mathcal{V},\Omega)$, as desired.

\subsection{Proof of Proposition \ref{Prop_varphi_mor}}\label{Subsec_varphi_mor}

This subsection is devoted to the proof of Proposition \ref{Prop_varphi_mor}. 
For the connection $\mathcal{L}^{-x}(\mathcal{L}^{x}(\mathcal{V},\Omega))$, let
	\begin{equation}\label{Eq_LinvL_mat}
	\mathcal{L}^{-x}(\mathcal{L}^{x}(\Omega))=\sum_{i=1}^n G_{x_i}(\bm{x})\,dx_i
	\end{equation}
be the matrix representation of $\mathcal{L}^{-x}(\mathcal{L}^{x}(\Omega))$ with respect to the fixed basis $(\mathcal{V}^q)^{\hat{q}}\cong(\mathbb{C}^{Nq})^{\hat{q}}$ at the beginning of Section \ref{Subsec_inv_form}.
Then, our goal is to show
	\begin{equation}\label{Eq_comm_inv}
	\varphi\circ G_{x_i}(\bm{x})=A_{x_i}(\bm{x})\circ\varphi 
	\quad (i=1,2,\ldots,n).
	\end{equation}
By iterating the Laplace transform $\mathcal{L}^{x}$ and the inverse Laplace transform $\mathcal{L}^{-x}$ for the $1$-form $\Omega$, we have
	\begin{equation}\label{Eq_Ue_x}
	G_{x}(\bm{x})=G_x+
	\sum_{z\in x'}G_{xz}z
	+\sum_{i=1}^q\frac{G_{H_i}}{x-a_{H_i}},
	\end{equation}
where 
	\begin{align}
	&G_x=
	\begin{pmatrix}
		a_{\hat{H}_1}I_{qN} & & & & &  \\
		& \ddots& &  & &  \\
		& & a_{\hat{H}_{\hat{p}}} I_{qN} & &  &  \\
		& & & a_{\hat{H}_{\hat{p}+1}}(0) I_{qN} &  &  \\
		& & & & \ddots &  \\
		& & & & &  a_{\hat{H}_{\hat{q}}}(0) I_{qN}
	\end{pmatrix},
	\label{Eq_X_x} 
	\\
	&G_{xz}=
	\begin{pmatrix}
		O_{qN} & & & & &  \\
		& \ddots& &  & &  \\
		& & O_{qN} & &  &  \\
		& & & (a_{\hat{H}_{\hat{p}+1}})_z I_{qN} &  &  \\
		& & & & \ddots &  \\
		& & & & &  (a_{\hat{H}_{\hat{q}}})_z I_{qN}
	\end{pmatrix},
	\label{Eq_X_xz}
	\\
	&G_{H_i}
	=\begin{pmatrix}
		G_1^i & G_2^i& \cdots & G_{\hat{q}}^i \\
		G_1^i & G_2^i& \cdots & G_{\hat{q}}^i \\
		& & \vdots & \\
		G_1^i & G_2^i& \cdots & G_{\hat{q}}^i
	\end{pmatrix} 
	,\quad 
	G_{j}^{i}=
	\begin{pmatrix}
	O_N & O_N & \cdots   &O_N \\
	\vdots & \cdots & \cdots &  \vdots \\
	E_{N_j}A_{H_1} & E_{N_j}A_{H_2} & \cdots & E_{N_j}A_{H_q}  \\
	\vdots & \cdots & \cdots &  \vdots \\
	O_N & O_N & \cdots  &O_N \\	
	\end{pmatrix}(\,i.
	\label{Eq_XH_i}
	\end{align}
Then, by a similar calculation with the proof of Proposition \ref{prop:odevarphimor}, we have the following.
\begin{proposition}
\label{Prop_x_compati}
The linear map \eqref{Eq_varphi} satisfies
	\begin{enumerate}[(i)]
	\item $\varphi\circ G_x=A_x\circ\varphi$,

	\item $\varphi\circ G_{xz}=A_{xz}\circ \varphi \quad (z\in x')$,
	
	\item $\varphi\circ G_{H_i}=A_{H_i}\circ\varphi \quad (1\le i\le q)$.
	\end{enumerate}
\end{proposition}
Therefore, we see that \eqref{Eq_comm_inv} with $i=\ell$ holds.

\medskip

We proceed to consider the commutativity on other variables: 
we choose any $y \in x'$ and shall show that
	\begin{equation}
	\varphi\circ G_y(\bm{x})=A_y(\bm{x})\circ\varphi.
	\end{equation}
The expression of $G_y(\bm{x})$ is quite different whether there exists an $x_k \in x'$ such that $\mathcal{A}_x\cap\mathcal{A}_{x_k}\neq \emptyset$ or not.
Hence we divide the proof into two cases:
	\begin{enumerate}[\bf {Case} 1:]
	\item There exists an $x_k \in x'$ such that $\mathcal{A}_x\cap\mathcal{A}_{x_k}\neq \emptyset$. 
	\item $\mathcal{A}_x\cap\mathcal{A}_{x_k}= \emptyset$ for all $x_k\in x'$.
	\end{enumerate}

\subsubsection{Case 1}

In this case, from the assumption \eqref{Eq_ass_inf}, we can assume $A_{xz}=O$ for $z\in x'$ (Namely, $\hat{p}=\hat{q}$).
Then, the $y$-equation of \eqref{Eq_Pfaff_V} is given by 
	\begin{equation}\label{Eq_V_yeq}
	\frac{\p V}{\p y}=\left(
	B_y+xB_{yx}+\sum_{z\in x''} B_{yz} z+\sum_{\substack{q+1\le j\le r \\ H_j \in\mathcal{A}_y}}\frac{B_{H_j}}{y-b_{H_j}}+\sum_{\substack{1\le k<l\le q \\ H_l \in \mathcal{C}_{H_k,y}}}\frac{B_{H_{kl}}}{y-c_{H_kH_l}}
	\right)V,
	\end{equation}
where 
	\begin{align}
	B_y&=\begin{pmatrix}
	A_y & & & & & \\
	& \ddots & & & & \\
	& & A_y & & & \\
	& & & A_y+(a_{H_{p+1}})_y A_x & & \\
	& & & & \ddots & \\
	& & & & & A_y+(a_{H_{q}})_y A_x
	\end{pmatrix}, 
	\\
	B_{yx}&=-\begin{pmatrix}
	O_N & & & & & \\
	& \ddots & & & & \\
	& & O_N & & & \\
	& & & (a_{H_{p+1}})_y I_N  & & \\
	& & & & \ddots & \\
	& & & & & (a_{H_{q}})_y I_N
	\end{pmatrix}
	\\
	B_{yz}&=\begin{pmatrix}
	A_{yz} & &  \\
	& \ddots &  \\
	& & A_{yz} 
	\end{pmatrix},
	\quad 
	B_{H_j}=\begin{pmatrix}
	A_{H_j} & & \\
	& \ddots & \\
	& & A_{H_j}
	\end{pmatrix}, \label{Eq_Ghatj}
	\\
	&\begin{blockarray}{ccccccc}
	\begin{block}{c(ccccc)c}
		&   &  &  &  &  &   \\
	 &  &  &  &  &  &   \\
	&   & A_{H_l} &  & -A_{H_l} &  & (k \\
	B_{H_{kl}}= &  &  &  &  &  &   \\
	&   & -A_{H_k} &  & A_{H_k} &  &(l   \\
	 &  &  &  &  &  &   \\
	   	 &  &  &  &  &  &   \\
	\end{block}\\[-20pt]
	& & \rotatebox{90}{$)$} & &\rotatebox{90}{$)$} & & \\[-3pt]
	&  & k & &l & &
	\end{blockarray}. \label{Eq_Ghatkl}
	\end{align}
Now we set 
	\begin{align*}
	\left.\Lambda_0=
	\left\{(k,l)\,\middle \vert\,
	\begin{array}{l}
	1\le k<l\le q,~
	H_l \in \mathcal{C}_{H_k,y} \\
	c_{H_kH_l} \neq b_{H_j} \text{ for all }q+1\le j\le r \text{ such that }H_j \in\mathcal{A}_y
	\end{array}
	\right\}\middle/\approx \right.,
	\end{align*}
where 
	\[
	(k,l) \approx (k',l') \quad \overset{\mathrm{def}}{\Leftrightarrow} \quad c_{H_kH_l}=c_{H_{k'}H_{l'}}.
	\]
Then, the equation \eqref{Eq_V_yeq} can be expressed as 
	\begin{equation}\label{Eq_V_yeq_2}
	\frac{\p V}{\p y}=\left(
	B_y+xB_{yx}+\sum_{z\in x''} B_{yz} z+\sum_{\substack{q+1\le j\le r \\ H_j \in\mathcal{A}_y}}\frac{C_{H_j}}{y-b_{H_j}}
	+\sum_{[(k,l)]\in\Lambda_0}\frac{C_{[(k,l)]}}{y-c_{H_kH_l}}
	\right)V,
	\end{equation}
where 
	\[
	C_{H_j}=B_{H_j}+\sum_{(k,l)\in\Lambda_j}B_{H_{kl}}, 
	\quad \Lambda_{j}=
	\left\{(k,l)~\middle \vert~
	\begin{array}{l}
	1\le k<l\le q,~
	H_l \in \mathcal{C}_{H_k,y} \\
	c_{H_kH_l}=b_{H_j} 
	\end{array}
	\right\}
	\]
for $q+1\le j \le r$ and 
	\[
	C_{[(k,l)]}=\sum_{(k',l')\in[(k,l)]}B_{H_{k'l'}}
	\]
for $[(k,l)]\in\Lambda_0$.

By applying Proposition \ref{Prop_invLaplace}, we obtain the coefficient matrix of \eqref{Eq_LinvL_mat} in $y$-direction: 
	\[
	G_{y}(\bm{x})=G_y+\sum_{z\in x''}G_{yz}z+\sum_{\substack{p+1\le k \le q\\ H_k\in \mathcal{A}_y}}\frac{G_{H_k}}{y-b_{H_k}}+\sum_{\substack{q+1\le j\le r \\ H_j \in\mathcal{A}_y}}\frac{ F _{H_j}}{y-b_{H_j}}
	+\sum_{[(k,l)]\in\Lambda_0}\frac{ F _{[(k,l)]}}{y-c_{H_kH_l}}
	\]
where $G_{H_k}$ $(p+1\le k \le q)$ are given by \eqref{Eq_XH_i} and
	\begin{equation}
	\begin{aligned}\label{Eq_Utilde_y_coef}
	&G_y=\begin{pmatrix}
	B_y+a_{\hat{H}_1}B_{yx} & & \\
	& \ddots &  \\
	& & B_y+a_{\hat{H}_{\hat{p}}}B_{yx}  \\
	\end{pmatrix}, 
	\quad 
	G_{yz}=
	\begin{pmatrix}
		B_{yz} & &  \\
		& \ddots& \\
		& & B_{yz} \\
		\end{pmatrix},
	\\
	& F _{H_j}=\begin{pmatrix}
	C_{H_j} & & \\
	& \ddots & \\
	& & C_{H_j}
	\end{pmatrix},
	\quad 
	 F _{[(k,l)]}=\begin{pmatrix}
	C_{[(k,l)]} & & \\
	& \ddots & \\
	& & C_{[(k,l)]}
	\end{pmatrix}.
	\end{aligned}
	\end{equation}
Since we have already shown 
	\[
	\varphi\circ G_{H_k}=A_{H_k}\circ \varphi \quad (p+1\le k \le q),
	\]
it suffices to show
	\begin{align}\label{Eq_compati_1}
	&\varphi\circ G_y
	=A_y\circ\varphi,
	\\
	&\varphi\circ\left(\sum_{z\in x''}G_{yz} z\right)
	=\left(\sum_{z\in x''}A_{yz}z\right)\circ\varphi, \label{Eq_compati_2}
	\\
	&\varphi\circ F _{H_j}=A_{H_j}\circ \varphi \quad (q+1\le j \le r \text{ such that }H_j \in\mathcal{A}_y),\label{Eq_compati_3}
	\\
	&\varphi\circ F _{[(k,l)]}=0 \quad ([(k,l)]\in \Lambda_0).
	\label{Eq_compati_4}
	\end{align}
Let $\bm{v}={}^t(v_1,\ldots,v_{\hat{p}}) \in \mathcal{L}^{-x}(\mathcal{L}^x(\mathcal{V}))$ with $v_i={}^t(v_1^{i},\ldots,v_q^{i}) \in \mathcal{V}^q$ and
	\[
	 \bm{w}=\begin{pmatrix}
	w_1 \\
	\vdots \\
	w_{\hat{p}}
	\end{pmatrix}
	:=G_{y}\bm{v}=
	\begin{pmatrix}
	(B_y+a_{\hat{H}_1}B_{yx})v_1 \\
	\vdots \\
	(B_y+a_{\hat{H}_{\hat{q}}}B_{yx})v_{\hat{p}}
	\end{pmatrix}.
	\]
Then, we have
	\begin{align*}
	w_i&=(B_y+a_{\hat{H}_i}B_{yx})v_i
	 \\
	 &=\begin{pmatrix}
	A_y & & & & & \\
	& \ddots & & & & \\
	& & A_y & & &  \\
	& & & A_y+(a_{H_{p+1}})_y(A_x-a_{\hat{H}_i}) & & \\
	& & & & \ddots & \\
	& & & & & A_y+(a_{H_{q}})_y(A_x-a_{\hat{H}_i})
	\end{pmatrix}
	\begin{pmatrix}
	v_1^{i} \\
	v_2^{i} \\
	\vdots \\
	v_{q}^{i}
	\end{pmatrix}.
	\end{align*}
Therefore we have
	\begin{align*}
	(\varphi\circ G_{y})(\bm{v})&=\varphi(\bm{w})
	=\sum_{i=1}^{\hat{p}}E_{N_i}
	\left(\sum_{j=1}^pA_{H_j}A_y v_j^i
	+
	\sum_{j=p+1}^q A_{H_j}\{A_y+(a_{H_j})_y(A_x-a_{\hat{H}_i})\}v_j^i
	\right).
	\end{align*}
From the conditions \eqref{eq:cond13p} and \eqref{eq:cond9p}, we see that $[A_{H_j},A_{y}]=O$ $(j=1,\ldots,p)$ and $[A_{H_j},A_y+(a_{H_j})_yA_x]=O$ $(j=p+1,\ldots, q)$, respectively. 
Then, we have
	\[
	\begin{aligned}
(\varphi\circ G_{y})(\bm{v})
	&=\sum_{i=1}^{\hat{p}}E_{N_i}
	\left(\sum_{j=1}^pA_yA_{H_j} v_j^i
	+
	\sum_{j=p+1}^q \{A_y+(a_{H_j})_y(A_x-a_{\hat{H}_i})\}A_{H_j}v_j^i
	\right) \\
	&=\sum_{i=1}^{\hat{p}}E_{N_i}
	\left(\sum_{j=1}^qA_yA_{H_j} v_j^i
	+
	\sum_{j=p+1}^q (a_{H_j})_y(A_x-a_{\hat{H}_i})A_{H_j}v_j^i
	\right).
	\end{aligned}
	\]
The condition $[A_{x},A_{y}]=O$ coming from \eqref{eq:cond3p} and 
	\[
	A_{x}=\begin{pmatrix}
	a_{\hat{H}_1}I_{N_1} & & \\
	& \ddots & \\
	& & a_{\hat{H}_{\hat{p}}} I_{N_{\hat{p}}}
	\end{pmatrix}
	=\sum_{i=1}^{\hat{p}}a_{\hat{H}_i} E_{N_i}
	\]
imply $[E_{N_i},A_y]=O$ for $i=1,\ldots,\hat{p}$.  
Hence, we obtain 
	\[
(\varphi\circ G_{y})(\bm{v})
	=A_y\sum_{i=1}^{\hat{p}}E_{N_i}
	\sum_{j=1}^qA_{H_j} v_j^i
	+
	\sum_{i=1}^{\hat{p}}E_{N_i}\sum_{j=p+1}^q (a_{H_j})_y(A_x-a_{\hat{H}_i})A_{H_j}v_j^i.
	\]
Now we notice that the $i$-th diagonal block of $A_x-a_{\hat{H}_i}$ is $O_{N_i}$, and hence $E_{N_i}(A_x-a_{\hat{H}_i})=O$.
Therefore, we obtain 
	\[
	(\varphi\circ G_y)(\bm{v})
	=A_y\sum_{i=1}^{\hat{p}}E_{N_i}
	\sum_{j=1}^qA_{H_j} v_j^i
	=(A_y\circ \varphi)(\bm{v}),
	\]
which shows \eqref{Eq_compati_1}.

We shall show \eqref{Eq_compati_2}. 
By a direct calculation, we obtain 
	\[
	(\varphi\circ G_{yz})(\bm{v})
	=\sum_{i=1}^{\hat{p}}E_{N_i}\sum_{j=1}^q A_{H_j}A_{yz}v_j^i
	\]
for $z \neq x,y$.
From the conditions \eqref{eq:cond10p} and \eqref{eq:cond14p}, we have $[A_{H_j},A_{yz}]=O$ for $1\le j\le q$.
Furthermore, the condition $[A_x,A_{yz}]=O$ coming from \eqref{eq:cond4p} implies $[E_{N_i},A_{yz}]=O$ for $1\le i\le \hat{p}$. 
Combining them, we obtain 
	\[
	(\varphi\circ G_{yz})(\bm{v})
	=A_{yz}\sum_{i=1}^{\hat{p}}E_{N_i}\sum_{j=1}^q A_{H_j}v_j^i
	=(A_{yz}\circ\varphi)(\bm{v}),
	\]
which shows \eqref{Eq_compati_2}.	

We proceed to show \eqref{Eq_compati_3}. 
By a direct calculation, we have
	\[
	 F _{H_j}\bm{v}=
	\begin{pmatrix}
	C_{H_j}v_1 \\
	\vdots 
	\\
	C_{H_j}v_{\hat{q}}
	\end{pmatrix}
	\]
and 
	\[
	C_{H_j}v_i=
	\left(B_{H_j}+\sum_{(k,l)\in\Lambda_j}B_{H_{kl}}\right)v_i
	=
	\begin{pmatrix}
	A_{H_j}v_{1}^{i} \\
	\vdots \\
	A_{H_j}v_{q}^{i}
	\end{pmatrix}
	+\sum_{(k,l)\in\Lambda_j}
\begin{blockarray}{cc}
\begin{block}{(c)c}
	0_{N} &\\
	\vdots & \\
	A_{H_l}(v_{k}^{i}-v_{l}^{i}) & (k\\
	\vdots &\\
	A_{H_k}(-v_{k}^{i}+v_{l}^{i}) & (l\\
	\vdots &\\
	0_{N} & \\
\end{block}
\end{blockarray}
	=:\begin{pmatrix}
	w_{1}^{i} \\
	\vdots \\
	w_{q}^{i}
	\end{pmatrix}.
	\]
Then, we obtain
	\begin{equation}\label{Eq_varphi_Ftilde}
	\begin{aligned}
	(\varphi\circ F _{H_j})(\bm{v})&=
	\sum_{i=1}^{\hat{p}}E_{N_i}\sum_{m=1}^q A_{H_m}w_m^i
	\\
	&
	=\sum_{i=1}^{\hat{p}}E_{N_i}\sum_{m=1}^qA_{H_m}\left\{
	A_{H_j}v_{m}^{i}+\sum_{(m,l)\in\Lambda_j}A_{H_l}(v_m^i-v_l^i)
	+\sum_{(k,m)\in\Lambda_j}A_{H_k}(-v_k^i+v_m^i)	\right\}
	\\
	&
	=\sum_{i=1}^{\hat{p}}\sum_{m=1}^q\Big[E_{N_i}A_{H_m}
	A_{H_j}v_{m}^{i}  \\
 	&\qquad \left.+E_{N_i}\left\{\sum_{(m,l)\in\Lambda_j}A_{H_m}A_{H_l}(v_m^i-v_l^i)
	+\sum_{(k,m)\in\Lambda_j}A_{H_m}A_{H_k}(-v_k^i+v_m^i)
	\right\}\right].
	\end{aligned}
	\end{equation}
To proceed with the calculation, let us recall the Pfaffian system \eqref{Eq_Pfaff_V}. 
From the completely integrable condition for \eqref{Eq_Pfaff_V}, we have
	\[
	[B_{\hat{H}_i},C_{H_j}]=O \quad (1\le i\le \hat{p},~\, q+1\le j \le r \text{ such that }H_j\in\mathcal{A}_y).
	\]
By using \eqref{Eq_hat_G}, \eqref{Eq_Ghatj} and \eqref{Eq_Ghatkl}, this condition can be expressed as
	\begin{align*}
	&\begin{pmatrix}
	[E_{N_i}A_{H_1},A_{H_j}] & [E_{N_i}A_{H_2},A_{H_j}] & \cdots & [E_{N_i}A_{H_q},A_{H_j}] \\
	& & \vdots & \\
	[E_{N_i}A_{H_1},A_{H_j}] & [E_{N_i}A_{H_2},A_{H_j}] & \cdots & [E_{N_i}A_{H_q},A_{H_j}]
	\end{pmatrix} 
	\\
	&\begin{blockarray}{cccccccc}
	\begin{block}{c(ccccccc)}
	& O& \cdots & E_{N_i}[A_{H_k},A_{H_l}] & \cdots & E_{N_i}[A_{H_l},A_{H_k}] & \cdots & O   \\
+\ds\sum_{(k,l)\in \Lambda_j} &	 \vdots &  & \vdots  &  & \vdots  &  &  \vdots  \\[-11pt]
	&O& \cdots & E_{N_i}[A_{H_k},A_{H_l}] & \cdots & E_{N_i}[A_{H_l},A_{H_k}] & \cdots & O   \\
	\end{block}\\[-10pt]
	&& & \rotatebox{90}{$)$} & &\rotatebox{90}{$)$} & & \\[-3pt]
	&&  & k & &l & &
	\end{blockarray} 
	\\[-10pt]
	&=O.
	\end{align*}
Picking up the $m$-th column ($1\le m\le q$), we obtain
	\begin{equation}\label{Eq_EAA1}
	[E_{N_i}A_{H_m},A_{H_j}]+E_{N_i}\left(\sum_{(m,l)\in \Lambda_j}[A_{H_m},A_{H_l}]+
	\sum_{(k,m)\in\Lambda_j}[A_{H_m},A_{H_k}]\right)=O
	\end{equation}
for $1\le i\le \hat{p}$ and $q+1\le j \le r$. 
By using this, we see that \eqref{Eq_varphi_Ftilde} is reduced to
	\begin{align*}
	(\varphi\circ F _{H_j})(\bm{v})=&(A_{H_j}\circ\varphi)(v),
	\end{align*}
which shows \eqref{Eq_compati_3}. 

Lastly, we shall show \eqref{Eq_compati_4}. 
Take $[(k,l)] \in \Lambda_0$ and consider
	\[
	F _{[(k,l)]}\bm{v}=
	\begin{pmatrix}
	C_{[(k,l)]}v_1 \\
	\vdots \\
	C_{[(k,l)]}v_{\hat{p}}
	\end{pmatrix},
	\quad 
	C_{[(k,l)]}v_i=\sum_{(k',l')\in[(k,l)]}
	\begin{blockarray}{cc}
\begin{block}{(c)c}
	0_{N} &\\
	\vdots & \\
	A_{H_{l'}}(v_{k'}^{i}-v_{l'}^{i}) & (k'\\
	\vdots &\\
	A_{H_{k'}}(-v_{k'}^{i}+v_{l'}^{i}) & (l'\\
	\vdots &\\
	0_{N} & \\
\end{block}
\end{blockarray}=:
	\begin{pmatrix}
	w_1^i \\
	\vdots \\
	w_q^i
	\end{pmatrix}.
	\]
Then we have
	\begin{equation}\label{Eq_varphi_Gtilde}
	\begin{aligned}
	(\varphi\circ F _{[(k,l)]})(\bm{v})
	&=\sum_{i=1}^{\hat{p}}E_{N_i}\sum_{m=1}^q A_{H_m}w_m^i \\
	&=\sum_{i=1}^{\hat{p}}E_{N_i}\sum_{m=1}^{q}A_{H_m}\left(\sum_{(m,l')\in[(k,l)]}A_{H_{l'}}(v_m^i-v_{l'}^i)-\sum_{(k',m)\in[(k,l)]}A_{H_{k'}}(v_{k'}^i-v_{m}^i)\right).
	\end{aligned}
	\end{equation}
Here, we focus on the second term on the right-hand side:
	\[
	\begin{aligned}
	\sum_{m=1}^qA_{H_m}\sum_{(k',m)\in[(k,l)]}A_{H_{k'}}(v_{k'}^i-v_{m}^i)
	&=\sum_{\substack{1\le m\le q \\ (k',m)\in[(k,l)]}}A_{H_m}A_{H_{k'}}(v_{k'}^i-v_{m}^i)
	\\
	&=\sum_{\substack{1\le m\le q \\ 1\le k'< m \\ c_{H_{k'}H_m}=c_{H_{k}H_l}}}A_{H_m}A_{H_{k'}}(v_{k'}^i-v_{m}^i)
	\\
	&=\sum_{\substack{1\le k'\le q \\ k'<m\le q \\ c_{H_{k'}H_m}=c_{H_{k}H_l}}}A_{H_m}A_{H_{k'}}(v_{k'}^i-v_{m}^i).
	\end{aligned}
	\]
Since $c_{H_{k'}H_m}=c_{H_mH_{k'}}$ holds, we can exchange $k'\leftrightarrow m$ in this sum and obtain
	\[
	\begin{aligned}
	\sum_{m=1}^qA_{H_m}\sum_{(k',m)\in[(k,m)]}A_{H_{k'}}(v_{k'}^i-v_{m}^i)
	&=\sum_{\substack{1\le m\le q \\ m< k'\le q \\ c_{H_mH_{k'}}=c_{H_{k}H_{l}}}}A_{H_{k'}}A_{H_{m}}(v_{m}^i-v_{k'}^i)
	\\
	&=\sum_{m=1}^{q}\sum_{(m,k')\in[(k,l)]}A_{H_{k'}}A_{H_{m}}(v_{m}^i-v_{k'}^i).
	\end{aligned}	
	\]
Substituting this into \eqref{Eq_varphi_Gtilde}, we have
	\begin{equation}
	\label{Eq_varphi_Gtilde2}
	(\varphi\circ F _{[(k,l)]})(\bm{v})
	=\sum_{i=1}^{\hat{p}}E_{N_i}\sum_{m=1}^{q}\sum_{(m,l')\in[(k,l)]}[A_{H_m},A_{H_{l'}}](v_m^i-v_{l'}^i).
	\end{equation}
From the completely integrable condition of the system \eqref{Eq_Pfaff_V}, it holds that 
	\[
	\sum_{(k',l')\in[(k,l)]}[B_{\hat{H}_i},B_{H_{k' l'}}]=O
	\quad (1\le i \le \hat{p}).
	\]
From \eqref{Eq_hat_G} and \eqref{Eq_Ghatkl}, we have
	\begin{align*}
	&\begin{blockarray}{cccccccc}
	\begin{block}{c(ccccccc)}
	& O& \cdots & E_{N_i}[A_{H_{l'}},A_{H_{k'}}] & \cdots & E_{N_i}[A_{H_{k'}},A_{H_{l'}}] & \cdots & O   \\
\ds\sum_{(k',l')\in[(k,l)]}[B_{\hat{H}_i},B_{H_{k'l'}}]=\sum_{(k',l')\in[(k,l)]} &	 \vdots &  & \vdots  &  & \vdots  &  &  \vdots  \\[-11pt]
	&O& \cdots & E_{N_i}[A_{H_{l'}},A_{H_{k'}}] & \cdots & E_{N_i}[A_{H_{k'}},A_{H_{l'}}] & \cdots & O   \\
	\end{block}\\[-10pt]
	&& & \rotatebox{90}{$)$} & &\rotatebox{90}{$)$} & & \\[-3pt]
	&&  & k' & &l' & &
	\end{blockarray} 
	\\[-10pt]
	&=O.
	\end{align*}
Then, by picking up the $m$-th column, we have 
	\begin{equation}
	\label{Eq_Case1_lem}
	\sum_{(m,l')\in[(k,l)]}E_{N_i}[A_{H_{l'}},A_{H_{m}}]+\sum_{(k',m)\in[(k,l)]}E_{N_i}[A_{H_{k'}},A_{H_m}]=O
	\end{equation}
for $1\le i\le \hat{p}$ and $1\le m\le q$.
Here \eqref{Eq_varphi_Gtilde2} can be expressed as
\begin{align*}
	(\varphi\circ F _{[(k,l)]})(\bm{v})
	&=\sum_{i=1}^{\hat{p}}E_{N_i}\left\{
	\sum_{m=1}^{q} \sum_{(m,l')\in[(k,l)]}[A_{H_m},A_{H_{l'}}]v_m^i
	-\sum_{m=1}^{q}\sum_{(m,l')\in[(k,l)]}[A_{H_m},A_{H_{l'}}]v_{l'}^i
	\right\}.
	\end{align*}
By using \eqref{Eq_Case1_lem} for the first summation on $m$, we obtain
	\begin{align*}
	(\varphi\circ F _{[(k,l)]})(\bm{v})
	=\sum_{i=1}^{\hat{p}}E_{N_i}\left\{
	\sum_{m=1}^{q} \sum_{(k',m)\in[(k,l)]}[A_{H_{k'}},A_{H_m}]v_m^i
	-\sum_{m=1}^{q}\sum_{(m,l')\in[(k,l)]}[A_{H_m},A_{H_{l'}}]v_{l'}^i
	\right\}.
	\end{align*}
As in the above, we can change the order of summation and exchange the index $k'\leftrightarrow m$. 
Then we have
	\begin{align*}
	\sum_{m=1}^{q} \sum_{(k',m)\in[(k,l)]}[A_{H_{k'}},A_{H_m}]v_m^i
	&=\sum_{m=1}^q \sum_{(m,k')\in[(k,l)]}[A_{H_m},A_{H_{k'}}]v_{k'}^i.
	\end{align*}
Hence, we have
	\begin{align*}
	(\varphi\circ F _{[(k,l)]})(\bm{v})
		&=\sum_{i=1}^{\hat{p}}E_{N_i}\left\{
	\sum_{m=1}^q \sum_{(m,k')\in[(k,l)]}E_{N_i}[A_{H_m},A_{H_{k'}}]v_{k'}^i
	-\sum_{m=1}^{q}\sum_{(m,l')\in[(k,l)]}[A_{H_m},A_{H_{l'}}]v_{l'}^i
	\right\}
	\\
	&=0.
	\end{align*}

\subsubsection{Case 2}
In this case, the polynomial $f_{H}$ with $H\in\mathcal{A}_x$ does not depend on other variables. 
Namely, it holds that $a_{H}\in \mathbb{C}$ for $H\in \mathcal{A}_x$, i.e., $p=q$. 
In particular, since the polynomial $f_H$ with $H\in\mathcal{A}_y$ does not depend on $x$, it holds that $b_{H} \in \mathbb{C}[x'']$ for $H \in \mathcal{A}_{y}$.

Then, the $x$-equation \eqref{Eq_xeq_Lap_Sch2} of the Pfaffian system \eqref{Eq_Pfaff_V} is written 
	\begin{equation}\label{Eq_Pfaff_V_x2}
	\frac{\p V}{\p x}=\left(B_x+\sum_{i=1}^{\hat{q}}\frac{B_{\hat{H}_i}}{x-a_{\hat{H}_i}}\right)V
	\end{equation}
with 
	\begin{equation}
	\label{Eq_Case2_xeqcoef}
	B_x=
	-\begin{pmatrix}
	a_{H_1} & & \\
	& \ddots & \\
	& & a_{H_p}
	\end{pmatrix},
	\quad
	B_{\hat{H}_i}=
	-\begin{pmatrix}
	E_{N_i}A_{H_1} & \cdots & E_{N_i}A_{H_p} \\
	& \vdots & \\
	E_{N_i}A_{H_1} & \cdots & E_{N_i}A_{H_p}
	\end{pmatrix}
	\quad (1\le i \le \hat{q}).
	\end{equation}
Here we note that $a_{\hat{H}_1},\ldots,a_{\hat{H}_p} \in \mathbb{C}$ and $a_{\hat{H}_{\hat{p}+1}},\ldots,a_{\hat{H}_{\hat{q}}}\in\mathbb{C}[x']$.
From Proposition \ref{Prop_Laplace}, the $y$-equation of \eqref{Eq_Pfaff_V} is given by
	\begin{equation}
	\label{Eq_Pfaff_V_y2}
	\frac{\p V}{\p y}
	=\left(
	B_y+\sum_{z\in x''}B_{yz}z+\sum_{\substack{q+1\le j \le r \\ H_j \in \mathcal{A}_y}}\frac{B_{H_j}}{y-b_{H_j}}+\sum_{\substack{\hat{p}+1\le k\le \hat{q} \\ \hat{H}_k \in \mathcal{B}_y}}\frac{B_{\hat{H}_k}}{y-b_{\hat{H}_k}}
	\right)V,
	\end{equation}
where 
	\begin{align}
	&B_y=\begin{pmatrix}
	A_y+a_{H_1}A_{xy} & & \\
	&\ddots & \\
	& & A_y+a_{H_p}A_{xy}
	\end{pmatrix}, 
	\quad 
	B_{yz}=\begin{pmatrix}
	A_{yz} & & \\
	&\ddots & \\
	& & A_{yz}
	\end{pmatrix},
	\\
	&B_{H_j}=\begin{pmatrix}
	A_{H_j} & & \\
	& \ddots & \\
	& & A_{H_j}
	\end{pmatrix}
	\quad (q+1\le j \le r \text{ with } H_j \in \mathcal{A}_y). \label{Eq_Case2_yeqcoef}
	\end{align}
Let us consider the inverse Laplace transform.
Applying Proposition \ref{Prop_invLaplace} to the Pfaffian system \eqref{Eq_Pfaff_V}, we obtain the coefficient matrix of \eqref{Eq_LinvL_mat} in $y$-direction: 
	\begin{equation}\label{Eq_Ue_y2}
	G_{y}(\bm{x})=G_y+G_{yx}x+\sum_{z\in x''}G_{yz}z+\sum_{\substack{q+1\le j \le r \\ H_j \in \mathcal{A}_y}}\frac{G_{H_j}}{y-b_{H_j}}
	+\sum_{\substack{1\le k<l\le \hat{q} \\ \hat{H}_l \in \mathcal{C}_{\hat{H}_k,y}}}\frac{G_{\hat{H}_{kl}}}{y-c_{\hat{H}_k\hat{H}_l}},
	\end{equation}
where 
	\begin{align}
	&G_y=
	\begin{pmatrix}
	B_y & & & & &  \\
	& \ddots & & & &  \\
	& & B_y & & &  \\
	& & & B_y+(a_{\hat{H}_{\hat{p}+1}})_yB_x & &  \\
	& & & & \ddots & \\
	& & & & &  B_y+(a_{\hat{H}_{\hat{q}}})_yB_x
	\end{pmatrix},
	\\
	&G_{yx}=\begin{pmatrix}
	O_{pN} & & & & & \\
	& \ddots & & & & \\
	& & O_{pN} & & & \\
	& & & (a_{\hat{H}_{\hat{p}+1}})_y I_{pN}  & & \\
	& & & & \ddots & \\
	& & & & & (a_{\hat{H}_{q}})_y I_{pN}
	\end{pmatrix},
	\label{Eq_widetildeGyx2}
	\\
	&G_{yz}=
	\begin{pmatrix}
	B_{yz} & & \\
	& \ddots & \\
	& & B_{yz}
	\end{pmatrix},
	\qquad 
	G_{H_j}=
	\begin{pmatrix}
	G_{H_j} & & \\
	& \ddots  & \\
	& & G_{H_j}
	\end{pmatrix},
	\\
	&\begin{blockarray}{ccccccc}
	\begin{block}{c(ccccc)c}
		&   &  &  &  &  &   \\
	 &  &  &  &  &  &   \\
	&   & B_{\hat{H}_l} &  & -B_{\hat{H}_l} &  & (k \\
	G_{\hat{H}_{kl}}= &  &  &  &  &  &   \\
	&   & -B_{\hat{H}_k} &  & B_{\hat{H}_k} &  &(l   \\
	 &  &  &  &  &  &   \\
	   	 &  &  &  &  &  &   \\
	\end{block}\\[-20pt]
	& & \rotatebox{90}{$)$} & &\rotatebox{90}{$)$} & & \\[-3pt]
	&  & k & &l & &
	\end{blockarray}. \label{Eq_Ghatkl2}
	\end{align}
Here we set 
	\[
	\mathcal{C}_{\hat{H}_{k},y}:=\{\hat{H}\in\mathcal{B}_x\setminus\{ \hat{H}_k\} \mid (a_{\hat{H}}-a_{\hat{H}_k})_y \neq 0\}
	\]
by abuse of notation. 
We note that \eqref{Eq_widetildeGyx2} is the same as \eqref{Eq_X_xz} with $z=y$.
As in Case 1, let us rearrange the equation \eqref{Eq_Ue_y2}.
By abuse of notation, we set
	\[
	\Lambda_{j}:=\left\{(k,l)
	~\middle\vert\,
	\begin{array}{l}
	1\le k<l\le \hat{q},~
	\hat{H}_l \in \mathcal{C}_{\hat{H}_k,y} \\
	c_{\hat{H}_k\hat{H}_l}=b_{H_j} 
	\end{array}
	\right\}.
	\]
Then we have
	\begin{equation}\label{Eq_Ue_y2_2}
	G_{y}(\bm{x})=G_y+G_{yx}x+\sum_{z\in x''}G_{yz}z+\sum_{\substack{q+1\le j \le r \\ H_j \in \mathcal{A}_y}}\frac{ F _{H_j}}{y-b_{H_j}}
	+\sum_{\substack{1\le k<l\le \hat{q} \\ \hat{H}_l \in \mathcal{C}_{\hat{H}_k,y} \\ c_{\hat{H}_k\hat{H}_l}\neq b_{H_j} \text{ for all } q+1\le j\le r \\\text{ such that } H_j \in \mathcal{A}_y}}\frac{G_{\hat{H}_{kl}}}{y-c_{\hat{H}_k\hat{H}_l}},
	\end{equation}
where 
	\begin{equation}\label{Eq_Ftilde_j}
	 F _{H_j}=G_{H_j}+\sum_{(k,l)\in\Lambda_j}G_{\tilde{H}_{kl}}.
	\end{equation}
Moreover, we set
	\[\left.
	\Lambda':=\left\{(k,l)~\middle\vert ~
	\begin{array}{l}
	1\le k<l\le \hat{q},~
	\hat{H}_l \in \mathcal{C}_{\hat{H}_k,y} \\
	c_{\hat{H}_k\hat{H}_l} \neq b_{H_j} \text{ for all }q+1\le j\le r \text{ such that }H_j \in\mathcal{A}_y
	\end{array}
	\right\} \middle/ \approx \right. ,
	\]
where 
	\[
	(k,l) \approx (k',l') \quad \overset{\mathrm{def}}{\Leftrightarrow} \quad c_{\hat{H}_k\hat{H}_l}=c_{\hat{H}_k'\hat{H}_l'}.
	\]
For each $[(k,l)] \in \Lambda'$, we take a representative $(k,l)$ satisfying
	\begin{equation}
	\label{Eq_Case2_rep}
	(k,l) \approx (k,l')\quad \Rightarrow \quad l\le l'.
	\end{equation}
For such $(k,l)$, we set 
	\begin{equation}
	\label{Eq_Case2_FHkl}
	 F _{\hat{H}_{kl}}=\sum_{\substack{k<l'\le \hat{q} \\ \hat{H}_{l'}\in\mathcal{C}_{\hat{H}_k,y} \\ c_{\hat{H}_{k}\hat{H}_{l'}}=c_{\hat{H}_{k}\hat{H}_l}}}G_{\hat{H}_{kl'}}
	\end{equation}
and define 
	\[
	 F _{[(k,l)]}=\sum_{\substack{k=1\\ (k,l)\in[(k,l)] \text{ with }\eqref{Eq_Case2_rep}}}^{\hat{q}} F _{\hat{H}_{kl}}.
	\]
Then we can express
	\begin{align*}
	\sum_{\substack{1\le k<l\le \hat{q} \\ \hat{H}_l \in \mathcal{C}_{\hat{H}_k,y} \\ c_{\hat{H}_k\hat{H}_l}\neq b_{H_j} \text{ for all } q+1\le j\le r \\\text{ such that } H_j \in \mathcal{A}_y}}\frac{G_{\hat{H}_{kl}}}{y-c_{\hat{H}_k\hat{H}_l}}
	&=\sum_{[(k,l)]\in\Lambda'}\frac{ F _{[(k,l)]}}{y-c_{\hat{H}_k\hat{H}_l}}.
	\end{align*}
Then, the $y$-equation \eqref{Eq_Ue_y2_2} can be expressed as 
	\begin{equation}\label{Eq_Ue_y3}
	G_{y}(\bm{x})=G_y+G_{yx}x+\sum_{z\in x''}G_{yz}z+\sum_{\substack{q+1\le j \le r \\ H_j \in \mathcal{A}_y}}\frac{ F_{H_j}}{y-b_{H_j}}
	+\sum_{[(k,l)]\in\Lambda'}\frac{ F _{[(k,l)]}}{y-c_{\hat{H}_k\hat{H}_l}}.
	\end{equation}
Since $\varphi\circ G_{yx}
	=A_{yx}\circ\varphi$
	has been already shown in Proposition \ref{Prop_x_compati} (note that $A_{xy}=A_{yx}$), it suffices to show 
	\begin{align}
	&\varphi\circ G_y
	=A_y\circ\varphi,
	\label{Eq_compati_5}
	\\
	&\varphi\circ\left(\sum_{z\in x''}G_{yz} z\right)
	=\left(\sum_{z\in x''}A_{yz}z\right)\circ\varphi, \label{Eq_compati_6}
	\\
	&\varphi\circ F _{H_j}=A_{H_j}\circ \varphi \quad (q+1\le j \le r \text{ such that }H_j \in\mathcal{A}_y),\label{Eq_compati_7}
	\\
	&\varphi\circ F _{[(k,l)]}=0 \quad ([(k,l)]\in \Lambda').
	\label{Eq_compati_8}
	\end{align}
We shall show \eqref{Eq_compati_5}.
Let $\bm{v}={}^t(v_1,\ldots,v_{\hat{q}}) \in \mathcal{L}^{-x}(\mathcal{L}^x(\mathcal{V}))$ with $v_i={}^t(v_1^{i},\ldots,v_p^{i}) \in \mathcal{V}^p$
and 
	\[
	 \bm{w}=\begin{pmatrix}
	w_1 \\
	\vdots \\
	w_{\hat{p}}\\
	w_{\hat{p}+1}\\
	\vdots \\
	w_{\hat{q}}
	\end{pmatrix}
	:=G_{y}\bm{v}=
	\begin{pmatrix}
	B_yv_1 \\
	\vdots \\
	B_yv_{\hat{p}} \\
	\{B_y+(a_{\hat{H}_{\hat{p}+1}})_yB_{x}\}v_{\hat{p}+1}\\
	\vdots \\
	\{B_y+(a_{\hat{H}_{\hat{q}}})_yB_{x}\}v_{\hat{q}}
	\end{pmatrix}.
	\]
Then, we have
	\begin{align*}
	w_i&=B_yv_i
	=\begin{pmatrix}
	A_y+a_{H_1}A_{xy} & & \\
	&\ddots & \\
	& & A_y+a_{H_p}A_{xy}
	\end{pmatrix}
	\begin{pmatrix}
	v_1^{i} \\
	\vdots \\
	v_{p}^{i}
	\end{pmatrix}
	\end{align*}
	for $1\le i\le \hat{p}$ and
	\begin{align*}
	w_i&=\{B_y+(a_{\hat{H}_i})_yB_{x}\}v_i
	 \\
	 &
	=\begin{pmatrix}
	A_y+a_{H_1}\{A_{xy}-(a_{\hat{H}_i})_y\} & & &  \\
	&A_y+a_{H_2}\{A_{xy}-(a_{\hat{H}_i})_y\}  & &  \\
	& &\ddots &  \\
	& & & A_y+a_{H_p}\{A_{xy}-(a_{\hat{H}_i})_y\}    \\
	\end{pmatrix}
	\begin{pmatrix}
	v_1^{i} \\
	v_2^{i} \\
	\vdots \\
	v_{p}^{i}
	\end{pmatrix}
	\end{align*}
	for $\hat{p}+1\le i\le \hat{q}$.
Hence, we have 
	\begin{align*}
	\varphi(\bm{w})=\sum_{i=1}^{\hat{p}}E_{N_i}\sum_{j=1}^p A_{H_j}(A_y+a_{H_j}A_{xy})v_j^i
	+\sum_{i=\hat{p}+1}^{\hat{q}}E_{N_i}\sum_{j=1}^p A_{H_j}\{A_y+a_{H_j}(A_{xy}-(a_{\hat{H}_i})_y)\}v_j^i.
	\end{align*}
By \eqref{eq:cond13p} and $a_{H_j}\in\bbC$, it holds that $[A_{H_j},A_y+a_{H_j}A_{xy}]=O$ for $1\le j\le p$. 
Then we have
	\[
	\varphi(\bm{w})=\sum_{i=1}^{\hat{p}}E_{N_i}\sum_{j=1}^p (A_y+a_{H_j}A_{xy})A_{H_j}v_j^i
	+\sum_{i=\hat{p}+1}^{\hat{q}}E_{N_i}\sum_{j=1}^p \{A_y+a_{H_j}(A_{xy}-(a_{\hat{H}_i})_y)\}A_{H_j}v_j^i.
	\]
From
	\begin{equation}\label{eq:expressionAxy}
	A_{xy}=\begin{pmatrix}
	O_{N_1} & & & & &\\
	& \ddots & & & & \\
	& & O_{N_{\hat{p}}} & & &\\
	& & & (a_{\hat{H}_{\hat{p}+1}})_yI_{N_{\hat{p}+1}} & & \\
	& & & & \ddots & \\
	& & & & & (a_{\hat{H}_{\hat{q}}})_yI_{N_{\hat{q}}}
	\end{pmatrix},
	\end{equation}
we have $E_{N_i}A_{xy}=O$ for $1\le i\le\hat{p}$ and $E_{N_i}\{A_{xy}-(a_{\hat{H}_i})_y\}=O$ for $\hat{p}+1\le i\le \hat{q}$.
Hence, we obtain
	\[
	\varphi(\bm{w})=\sum_{i=1}^{\hat{q}}E_{N_i}\sum_{j=1}^p A_yA_{H_j}v_j^i.
	\]
Here we prepare a lemma. 
\begin{lemma}\label{lem:comAxAy}
Under the assumption \eqref{Eq_ass_diag} (hence $A_{x}$ and $A_{xy}$ are diagonalizable), it holds that
\[
[A_{y},A_{x}]=O, \quad [A_{y},A_{xz}]=O~~~~~\,(z\in x'').
\]
\end{lemma}
\begin{proof}
Firstly, we shall show $[A_{y},A_{x}]=O$. 
From \eqref{eq:cond3p}, we have
\begin{equation}\label{eq:commAxAy}
[A_y,A_x]
=\left[A_{xy},\ \sum_{H\in\mathcal A_x^{c}\cap\mathcal A_y}A_H-\sum_{H\in\mathcal A_x}A_H\right].
\end{equation}
Define the linear map $\mathrm{ad}_{A_{xy}}:\mathrm{Mat}(N,\mathbb C)\to\mathrm{Mat}(N,\mathbb C)$ by
$\mathrm{ad}_{A_{xy}}(A)=[A_{xy},A]$.
Then \eqref{eq:commAxAy} implies $[A_{y},A_{x}]\in \mathrm{Im}(\mathrm{ad}_{A_{xy}})$.
On the other hand, by the Jacobi identity,
\begin{equation}\label{eq:commAxAy2}
[A_{xy},[A_{y},A_{x}]]
=[[A_{x},A_{xy}],A_y]+[A_x,[A_{y},A_{xy}]].
\end{equation}
By \eqref{eq:cond1p}, we see that $[A_{x},A_{xy}]=O$. 
Moreover, by applying \eqref{eq:cond1p} also to the interchanged pair $(y,x)$ and using $A_{xy}=A_{yx}$, we have $[A_{y},A_{xy}]=O$. 
Then, the right-hand side of \eqref{eq:commAxAy2} vanishes and hence
$[A_{y},A_{x}]\in \Ker(\mathrm{ad}_{A_{xy}})$.
Therefore,
\[
[A_{y},A_{x}]\in \Ker(\mathrm{ad}_{A_{xy}})\cap \mathrm{Im}(\mathrm{ad}_{A_{xy}}).
\]
We claim that $\Ker(\mathrm{ad}_{A_{xy}})\cap \mathrm{Im}(\mathrm{ad}_{A_{xy}})=\{O\}$ under the assumption.
Let $E_{ij}=E_{ij}^{(N)}$ denote the standard matrix unit. 
For clearly, we denote \eqref{eq:expressionAxy} as $A_{xy}=\mathrm{diag}[\lambda_1,\ldots,\lambda_N]$ specifically for this proof.
Then we have
\[
\mathrm{ad}_{A_{xy}}(E_{ij})
=(\lambda_{i}-\lambda_{j})E_{ij}.
\]
Take $A\in \Ker(\mathrm{ad}_{A_{xy}})\cap\mathrm{Im}(\mathrm{ad}_{A_{xy}})$ and write
\[
A=\sum_{1\le i,j\le N}a_{ij}E_{ij}\quad (a_{ij}\in\mathbb{C}).
\]
From $A\in \Ker(\mathrm{ad}_{A_{xy}})$, we have
\[
\mathrm{ad}_{A_{xy}}(A)=\sum_{1\le i,j\le N}(\lambda_i-\lambda_j)a_{ij}E_{ij}=O
\]
and hence $(\lambda_{i}-\lambda_{j})a_{ij}=0$. 
Therefore we have $a_{ij}=0$ for $(i,j)$ such that $\lambda_i\neq\lambda_j$. 
On the other hand, since $A\in\mathrm{Im}(\mathrm{ad}_{A_{xy}})$, there exists $M=\sum_{1\le i,j\le N}m_{ij}E_{ij}\in\mathrm{Mat}(N,\mathbb{C})$ such that $\mathrm{ad}_{A_{xy}}(M)=A$. 
This implies
\[
\sum_{1\le i,j\le N}(\lambda_i-\lambda_j)m_{ij}E_{ij}=\sum_{1\le i,j\le N}a_{ij}E_{ij}
\] 
and hence $(\lambda_{i}-\lambda_{j})m_{ij}=a_{ij}$. 
Therefore we have $a_{ij}=0$ for $(i,j)$ such that $\lambda_i=\lambda_j$. 
Thus $a_{ij}=0$ for all $(i,j)$, i.e., $A=O$. 
This proves the claim and follows that $[A_{y},A_{x}]=O$. 

Next we shall show $[A_{y},A_{xz}]=O$. 
From \eqref{eq:cond4p}, we have
\[
[A_{y},A_{xz}]=[A_{x},A_{yz}]
\]
for any $z\in x''$. 
This means that $[A_{y},A_{xz}]\in\mathrm{Im}(\mathrm{ad}_{A_{x}})$. 
On the other hand, by the Jacobi identity,
\begin{equation}\label{eq:commAxAy3}
[A_{x},[A_{y},A_{xz}]]=[A_{xz},[A_{y},A_{x}]]+[A_{y},[A_{x},A_{xz}]].
\end{equation}
Here we recall that $[A_{y},A_{x}]=O$, which we proved above. 
Moreover, applying \eqref{eq:cond1p} to the pair $(x,z)$, we have $[A_{x},A_{xz}]=O$. 
Therefore, the right-hand side of \eqref{eq:commAxAy3} vanishes. 
This means that $[A_{y},A_{xz}]\in \Ker(\mathrm{ad}_{A_{x}})$. 
Therefore, $[A_{y},A_{xz}]\in\Ker(\mathrm{ad}_{A_{x}})\cap\mathrm{Im}(\mathrm{ad}_{A_{x}})$.
Since $A_{x}$ is diagonal, the same argument as above shows that $\Ker(\mathrm{ad}_{A_{x}})\cap\mathrm{Im}(\mathrm{ad}_{A_{x}})=\{O\}$.
\end{proof}
Lemma \ref{lem:comAxAy} and \eqref{Eq_A_yA_2} imply $[A_y,E_{N_i}]=O$ for $1\le i \le \hat{q}$. Hence we have
	\[
	(\varphi\circ G_{y})(\bm{v})=\varphi(\bm{w})=A_y\sum_{i=1}^{\hat{q}}E_{N_i}\sum_{j=1}^p A_{H_j}v_j^i=(A_y\circ \varphi)(\bm{v}),
	\]
as desired. 

Next, we shall show \eqref{Eq_compati_6}.
For $z \neq x,y$, we have
	\[
	(\varphi\circ G_{yz})(\bm{v})=\sum_{i=1}^{\hat{q}}\sum_{j=1}^p E_{N_i}A_{H_j}A_{yz}v_j^i.
	\]
Here, from the completely integrable conditions \eqref{eq:cond10p} and \eqref{eq:cond14p} for \eqref{Eq_Pfaff_V}, we have $[B_{yz},B_{\hat{H}_i}]=O$ ($1\le i\le \hat{q})$. 
This implies $[E_{N_i}A_{H_j},A_{yz}]=O$ for $1\le i \le \hat{q}$ and $1\le j \le p$. 
Therefore we have
	\[
	(\varphi\circ G_{yz})(\bm{v})=A_{yz}\sum_{i=1}^{\hat{q}}\sum_{j=1}^p E_{N_i}A_{H_j}v_j^i
	=(A_{yz}\circ\varphi)(\bm{v}),
	\]
as desired.

Next, we shall show \eqref{Eq_compati_7}. 
We write $ F _{H_j}\bm{v}={}^t(w_1,\ldots,w_{\hat{q}})$ and $w_i={}^t(w_1^{i},\ldots,w_p^{i})\in\mathcal{V}^p$. 
From \eqref{Eq_Ftilde_j}, we have 
	\[
	w_i=B_{H_j}v_i+\sum_{(i,l)\in\Lambda_j}B_{\hat{H}_l}(v_i-v_l)+\sum_{(k,i)\in\Lambda_j}B_{\hat{H}_k}(-v_k+v_i)
	\]
for $1\le i\le \hat{q}$ and hence
	\[
	w_m^i=A_{H_j}v_m^i+\sum_{(i,l)\in\Lambda_j}\sum_{n=1}^pE_{N_l}A_{H_n}(v_n^l-v_n^i)+\sum_{(k,i)\in\Lambda_j}\sum_{n=1}^p E_{N_k}A_{H_n}(v_n^k-v_n^i)
	\]
for $1\le m\le p$.
Therefore we obtain 
	\begin{equation}
	\label{Eq_varphi_widehatF}
	\begin{aligned}
	\varphi( F _{H_j}\bm{v})&=\sum_{i=1}^{\hat{q}}E_{N_i}\sum_{m=1}^p A_{H_m}w_m^i 
	\\
	&=\sum_{i=1}^{\hat{q}}\sum_{m=1}^p E_{N_i}A_{H_m}
	\left\{A_{H_j}v_m^i+\sum_{(i,l)\in\Lambda_j}\sum_{n=1}^pE_{N_l}A_{H_n}(v_n^l-v_n^i)\right.
	\\
	&\qquad\qquad 
	\left.+\sum_{(k,i)\in\Lambda_j}\sum_{n=1}^p E_{N_k}A_{H_n}(v_n^k-v_n^i)\right\}.
	\end{aligned}
	\end{equation}
To proceed with the calculation, we give a lemma which comes from the completely integrable condition for \eqref{Eq_Pfaff_V}.
\begin{lemma}
	\begin{equation}
	\label{Eq_Case2_lem1}
	[A_{H_j},E_{N_i}A_{H_m}]+\sum_{\substack{\hat{p}+1 \le l\le \hat{q}, 
	\\ c_{\hat{H}_i\hat{H}_l}=b_{H_j}}}\sum_{n=1}^p \left(E_{N_i}A_{H_n}E_{N_l}-E_{N_l}A_{H_n}E_{N_i}\right)A_{H_m}=O
	\end{equation}
for $1 \le i \le \hat{p}$, $q+1\le j \le r$ with $H_j\in\calA_{y}$, and $1\le m \le p$.
	\begin{equation}
	\label{Eq_Case2_lem2}
	[A_{H_j},E_{N_i}A_{H_m}]+\sum_{\substack{ 
	 \hat{H}_l \in \mathcal{C}_{\hat{H}_i,y}\\ c_{\hat{H}_i\hat{H}_l}=b_{H_j}}}\sum_{n=1}^p \left(E_{N_i}A_{H_n}E_{N_l}-E_{N_l}A_{H_n}E_{N_i}\right)A_{H_m}=O
	\end{equation}
for $\hat{p}+1 \le i \le \hat{q}$, $q+1\le j \le r$ with $H_j\in\calA_{y}$, and $1\le m\le p$.
\end{lemma}
	\begin{proof}
	Using \eqref{Eq_Pfaff_V_x2}, \eqref{Eq_Pfaff_V_y2}, we compute the completely integrable condition for \eqref{Eq_Pfaff_V}. 
	Then, from the condition \eqref{eq:cond15p} we have
	\[
	\left[B_{\hat{H}_i},B_{H_j}+\sum_{\substack{\hat{p}+1\le l \le \hat{q} \\ c_{\hat{H}_i\hat{H}_l}=b_{H_j}}}B_{\hat{H}_l}
	\right]=O
	\]
	for $1 \le i \le \hat{p}$, $q+1\le j \le r$ with $H_j\in\calA_{y}$.
	Substituting \eqref{Eq_Case2_xeqcoef} and \eqref{Eq_Case2_yeqcoef} into this identity and then extracting the $m$-th block column ($1\le m \le p$), we obtain \eqref{Eq_Case2_lem1}. 
	
	Also, from the condition \eqref{eq:cond11p}, we have
	\[
	\left[
	B_{\hat{H}_i},B_{H_j}+\sum_{\substack{
	 \hat{H}_l \in \mathcal{C}_{\hat{H}_i,y}\\ c_{\hat{H}_i\hat{H}_l}=b_{H_j}}}B_{\hat{H}_l}
	 \right]=O
	\]
	for $\hat{p}+1 \le i \le \hat{q}$, $q+1\le j \le r$ with $H_j\in\calA_{y}$. 
	Substituting \eqref{Eq_Case2_xeqcoef} and \eqref{Eq_Case2_yeqcoef} into this identity and then extracting the $m$-th block column ($1\le m \le p$), we obtain \eqref{Eq_Case2_lem2}. 
	\end{proof}
To use this lemma, we divide the sum \eqref{Eq_varphi_widehatF} on $i$ into $1\le i\le \hat{p}$ and $\hat{p}+1\le i\le \hat{q}$.
Set 
	\begin{align*}
	\bm{A}&=\sum_{i=1}^{\hat{p}}\sum_{m=1}^p E_{N_i}A_{H_m}
	\left\{A_{H_j}v_m^i+\sum_{(i,l)\in\Lambda_j}\sum_{n=1}^pE_{N_l}A_{H_n}(v_n^l-v_n^i)+\sum_{(k,i)\in\Lambda_j}\sum_{n=1}^p E_{N_k}A_{H_n}(v_n^k-v_n^i)\right\},
	\\
	\bm{B}&=\sum_{i=\hat{p}+1}^{\hat{q}}\sum_{m=1}^p E_{N_i}A_{H_m}
	\left\{A_{H_j}v_m^i+\sum_{(i,l)\in\Lambda_j}\sum_{n=1}^pE_{N_l}A_{H_n}(v_n^l-v_n^i)+\sum_{(k,i)\in\Lambda_j}\sum_{n=1}^p E_{N_k}A_{H_n}(v_n^k-v_n^i)\right\}.
	\end{align*}
We focus on $\bm{A}$. 
For $1\le i\le \hat{p}$, it holds that $\hat{H}_l \in\mathcal{C}_{\hat{H}_i,y}$ if and only if $\hat{p}+1\le l\le \hat{q}$.
Therefore
	\begin{align*}
	&\{(i,l)\in\Lambda_j\mid 1\le i \le \hat{p}\}
	=\{(i,l)\mid 1\le i\le \hat{p},\,\hat{p}+1\le l\le \hat{q},\,c_{\hat{H}_i\hat{H}_l}=b_{H_j}\},
	\\
	&\{(k,i)\in\Lambda_j \mid 1\le i \le \hat{p}\}=\emptyset.
	\end{align*}
Then, we have
	\[
	\bm{A}=\sum_{i=1}^{\hat{p}}\sum_{m=1}^p E_{N_i}A_{H_m}\left\{A_{H_j}v_m^i+\sum_{\substack{\hat{p}+1\le l \le \hat{q} \\ c_{\hat{H}_i\hat{H}_l}=b_{H_j}}}\sum_{n=1}^pE_{N_l}A_{H_n}(v_n^l-v_n^i)\right\}.
	\]
Applying \eqref{Eq_Case2_lem1} to $E_{N_i}A_{H_m}A_{H_j}$, we have
	\begin{align*}
	\bm{A}&=\sum_{i=1}^{\hat{p}}\sum_{m=1}^{p}\left\{A_{H_j}E_{N_i}A_{H_m}+\sum_{\substack{\hat{p}+1\le l \le \hat{q} \\ c_{\hat{H}_i\hat{H}_l}=b_{H_j}}}\sum_{n=1}^{p}(E_{N_i}A_{H_n}E_{N_l}-E_{N_l}A_{H_n}E_{N_i})A_{H_m}\right\}v_m^i \\
	&\quad +\sum_{i=1}^{\hat{p}}\sum_{m=1}^{p}E_{N_i}A_{H_m}\sum_{\substack{\hat{p}+1\le l \le \hat{q} \\ c_{\hat{H}_i\hat{H}_l}=b_{H_j}}}\sum_{n=1}^{p}E_{N_l}A_{H_n}(v_n^l-v_n^i).
	\end{align*}
	Then we obtain
	\begin{equation}
	\label{Eq_Case2_A}
	\bm{A}= A_{H_j}\sum_{i=1}^{\hat{p}} E_{N_i}\sum_{m=1}^pE_{N_i}A_{H_m}v_m^i+ R_{\bm{A}},
	\end{equation}
where 
	\begin{align*}
	R_{\bm{A}}&=\sum_{i=1}^{\hat{p}}\sum_{m=1}^p\sum_{\substack{\hat{p}+1\le l \le \hat{q} \\ c_{\hat{H}_i\hat{H}_l}=b_{H_j}}}\sum_{n=1}^p(E_{N_i}A_{H_n}E_{N_l}-E_{N_l}A_{H_n}E_{N_i})A_{H_m}v_m^i
	\\
	&\quad +\sum_{i=1}^{\hat{p}}\sum_{m=1}^p\sum_{\substack{\hat{p}+1\le l \le \hat{q} \\ c_{\hat{H}_i\hat{H}_l}=b_{H_j}}}\sum_{n=1}^pE_{N_i}A_{H_m}E_{N_l}A_{H_n}(v_n^l-v_n^i).
	\end{align*}
Here we can exchange the running index $m \leftrightarrow n$ in the second term, which gives 
	\begin{equation}
	\label{Eq_Case2_RA}
	\begin{aligned}
		R_{\bm{A}}&=\sum_{i=1}^{\hat{p}}\sum_{m=1}^p\sum_{\substack{\hat{p}+1\le l \le \hat{q} \\ c_{\hat{H}_i\hat{H}_l}=b_{H_j}}}\sum_{n=1}^p(E_{N_i}A_{H_n}E_{N_l}-E_{N_l}A_{H_n}E_{N_i})A_{H_m}v_m^i
	\\
	&\quad +\sum_{i=1}^{\hat{p}}\sum_{n=1}^p\sum_{\substack{\hat{p}+1\le l \le \hat{q} \\ c_{\hat{H}_i\hat{H}_l}=b_{H_j}}}\sum_{m=1}^pE_{N_i}A_{H_n}E_{N_l}A_{H_m}(v_m^l-v_m^i)
	\\
	&=\sum_{i=1}^{\hat{p}}\sum_{m=1}^p\sum_{\substack{\hat{p}+1\le l \le \hat{q} \\ c_{\hat{H}_i\hat{H}_l}=b_{H_j}}}\sum_{n=1}^p(E_{N_i}A_{H_n}E_{N_l}A_{H_m}v_m^l-E_{N_l}A_{H_n}E_{N_i}A_{H_m}v_m^i).
	\end{aligned}
	\end{equation}
A similar consideration and \eqref{Eq_Case2_lem2} give
	\begin{equation}
	\label{Eq_Case2_B}
	\bm{B}=A_{H_j}\sum_{i=\hat{p}+1}^{\hat{q}}E_{N_i}\sum_{m=1}^{p}A_{H_m}v_{m}^{i}+R_{\bm{B}},
	\end{equation}
where 
	\[
	R_{\bm{B}}=\sum_{i=1}^{\hat{q}}\sum_{m=1}^p \sum_{\substack{\hat{p}+1\le l \le \hat{q} \\ c_{\hat{H}_i\hat{H}_l}=b_{H_j}}}\sum_{n=1}^p(E_{N_l}A_{H_n}E_{N_i}A_{H_m}v_m^i-E_{N_i}A_{H_n}E_{N_l}A_{H_m}v_m^l).
	\]
From \eqref{Eq_Case2_RA}, we obtain
	\begin{align*}
	R_{\bm{A}}+R_{\bm{B}}&=\sum_{i=\hat{p}+1}^{\hat{q}}\sum_{m=1}^p \sum_{\substack{\hat{p}+1\le l \le \hat{q} \\ c_{\hat{H}_i\hat{H}_l}=b_{H_j}}}\sum_{n=1}^p(E_{N_l}A_{H_n}E_{N_i}A_{H_m}v_m^i-E_{N_i}A_{H_n}E_{N_l}A_{H_m}v_m^l)
	\\
	&=\sum_{i=\hat{p}+1}^{\hat{q}}\sum_{m=1}^p \sum_{\substack{\hat{p}+1\le l \le \hat{q} \\ c_{\hat{H}_i\hat{H}_l}=b_{H_j}}}\sum_{n=1}^pE_{N_l}A_{H_n}E_{N_i}A_{H_m}v_m^i
	\\
	&\quad -\sum_{i=\hat{p}+1}^{\hat{q}}\sum_{m=1}^p \sum_{\substack{\hat{p}+1\le l \le \hat{q} \\ c_{\hat{H}_i\hat{H}_l}=b_{H_j}}}\sum_{n=1}^pE_{N_i}A_{H_n}E_{N_l}A_{H_m}v_m^l.
	\end{align*}
By exchanging the running index $i \leftrightarrow l$ in the second term, we have $R_{\bm{A}}+R_{\bm{B}}=O$. 
Then, from \eqref{Eq_Case2_A} and \eqref{Eq_Case2_B} we obtain
	\[
	(\varphi\circ F _{H_j})(\bm{v})=\bm{A}+\bm{B}=A_{H_j}\sum_{i=1}^{\hat{q}}E_{N_i}\sum_{m=1}^{p}A_{H_m}v_{m}^{i}=(A_{H_j}\circ \varphi)(\bm{v}),
	\]
which shows \eqref{Eq_compati_7}.

Lastly, we shall show \eqref{Eq_compati_8}. 
Take $[(k,l)]\in\Lambda'$ and a representative $(k,l)$ satisfying \eqref{Eq_Case2_rep}. 
Then, from \eqref{Eq_Case2_FHkl}, we have
	\begin{align*}
	 F _{\hat{H}_{kl}}\bm{v}&=\sum_{\substack{k<l'\le \hat{q} \\ \hat{H}_{l'}\in\mathcal{C}_{\hat{H}_k,y} \\ c_{\hat{H}_{k}\hat{H}_{l'}}=c_{\hat{H}_{k}\hat{H}_l}}}G_{\hat{H}_{kl'}}\bm{v} 
	%
	%
=\sum_{\substack{k<l'\le \hat{q} \\ \hat{H}_{l'}\in\mathcal{C}_{\hat{H}_k,y} \\ c_{\hat{H}_{k}\hat{H}_{l'}}=c_{\hat{H}_{k}\hat{H}_l}}}
	\begin{blockarray}{cc}
\begin{block}{(c)c}
	0_{pN} &\\
	\vdots & \\
	B_{\hat{H}_{l'}}(v_{k}-v_{l'}) & (k\\
	\vdots &\\
	B_{\hat{H}_{k}}(-v_{k}+v_{l'}) & (l'\\
	\vdots &\\
	0_{pN} & \\
\end{block}
\end{blockarray}=:
	\begin{pmatrix}
	w_1\\
	\vdots\\
	w_k \\
	\vdots \\
	w_{l'}\\
	\vdots \\
	w_{\hat{q}}
	\end{pmatrix}.
	\end{align*}
Namely,
\[
w_k=\sum_{\substack{k<l'\le \hat{q} \\ \hat{H}_{l'}\in\mathcal{C}_{\hat{H}_k,y} \\ c_{\hat{H}_{k}\hat{H}_{l'}}=c_{\hat{H}_{k}\hat{H}_l}}}B_{\hat{H}_{l'}}(v_k-v_{l'})
\]
and 
\[
w_{l'}=\begin{cases}
B_{\hat{H}_k}(v_{l'}-v_k) & k<l'\le \hat{q} \text{ with }\hat{H}_{l'}\in\mathcal{C}_{\hat{H}_k,y} \text{ and }c_{\hat{H}_{k}\hat{H}_{l'}}=c_{\hat{H}_{k}\hat{H}_l},
\\
0_{pN} & \text{otherwise}.
\end{cases}
\]
Hence, by setting $w_i={}^t(w_1^i,\ldots,w_p^i)$, we have
	\begin{align*}
	(\varphi\circ F_{\hat{H}_{kl}})(\bm{v})&=\sum_{i=1}^{\hat{q}}E_{N_i}\sum_{j=1}^{p}A_{H_j}w_{j}^{i} \\
	&=-E_{N_k}\sum_{j=1}^{p}A_{H_j}\sum_{\substack{k<l'\le \hat{q} \\ \hat{H}_{l'}\in\mathcal{C}_{\hat{H}_k,y} \\ c_{\hat{H}_{k}\hat{H}_{l'}}=c_{\hat{H}_{k}\hat{H}_l}}}\sum_{m=1}^{p}E_{N_{l'}}A_{H_m}(v_m^k-v_m^{l'})
	\\
	&\quad -\sum_{\substack{k<l'\le \hat{q} \\ \hat{H}_{l'}\in\mathcal{C}_{\hat{H}_k,y} \\ c_{\hat{H}_{k}\hat{H}_{l'}}=c_{\hat{H}_{k}\hat{H}_l}}}E_{N_{l'}}\sum_{j=1}A_{H_j}\sum_{m=1}^{p}E_{N_k}A_{H_m}(v_m^{l'}-v_m^k)
	\\
	&=\sum_{m=1}^{p}\sum_{\substack{k<l'\le \hat{q} \\ \hat{H}_{l'}\in\mathcal{C}_{\hat{H}_k,y} \\ c_{\hat{H}_{k}\hat{H}_{l'}}=c_{\hat{H}_{k}\hat{H}_l}}}\sum_{j=1}^{p}(E_{N_k}A_{H_j}E_{N_{l'}}-E_{N_{l'}}A_{H_j}E_{N_k})A_{H_m}(v_m^{l'}-v_m^k).
	\end{align*}
Then, we have
	\[
	\varphi(F_{[(k,l)]}\bm{v}):=\bm{C}-\bm{D},
	\]
where 
	\begin{align*}
	\bm{C}&=\sum_{\substack{k=1\\ (k,l)\in[(k,l)] \text{ with }\eqref{Eq_Case2_rep}}}^{\hat{q}}\sum_{m=1}^{p}\sum_{\substack{k<l'\le \hat{q} \\ \hat{H}_{l'}\in\mathcal{C}_{\hat{H}_k,y} \\ c_{\hat{H}_{k}\hat{H}_{l'}}=c_{\hat{H}_{k}\hat{H}_l}}}\sum_{j=1}^{p}(E_{N_k}A_{H_j}E_{N_{l'}}-E_{N_{l'}}A_{H_j}E_{N_k})A_{H_m}v_m^{l'}, \\
	\bm{D}&=\sum_{\substack{k=1\\ (k,l)\in[(k,l)] \text{ with }\eqref{Eq_Case2_rep}}}^{\hat{q}}\sum_{m=1}^{p}\sum_{\substack{k<l'\le \hat{q} \\ \hat{H}_{l'}\in\mathcal{C}_{\hat{H}_k,y} \\ c_{\hat{H}_{k}\hat{H}_{l'}}=c_{\hat{H}_{k}\hat{H}_l}}}\sum_{j=1}^{p}(E_{N_k}A_{H_j}E_{N_{l'}}-E_{N_{l'}}A_{H_j}E_{N_k})A_{H_m}v_m^k.
	\end{align*}
To show $\bm{C}=\bm{D}$, we prepare the following lemma.
\begin{lemma}
	\begin{equation}
	\label{Eq_Case2_lem3}
	\sum_{\substack{ \hat{H}_{l'}\in\mathcal{C}_{\hat{H}_k,y} \\ c_{\hat{H}_k\hat{H}_{l}}=c_{\hat{H}_k\hat{H}_{l'}}}}\sum_{j=1}^p(E_{N_k}A_{H_j}E_{N_{l'}}-E_{N_{l'}}A_{H_j}E_{N_k})A_{H_m}=O
		\end{equation}
for $\hat{p}+1\le k\le \hat{q}$, $1\le l\le \hat{q}$ such that $c_{\hat{H}_k\hat{H}_l}\neq b_{H_j}$ for any $q+1\le j\le r$, and $1\le m\le p$. 
	\begin{equation}
	\label{Eq_Case2_lem4}
	\sum_{\substack{\hat{p}+1\le {l'} \le \hat{q} \\ c_{\hat{H}_k\hat{H}_l}=c_{\hat{H}_k\hat{H}_{l'}}}}\sum_{j=1}^p(E_{N_k}A_{H_j}E_{N_{l'}}-E_{N_{l'}}A_{H_j}E_{N_k})A_{H_m}=O
	\end{equation}
for $1\le k\le\hat{p}$, $\hat{p}+1\le l\le\hat{q}$ such that $c_{\hat{H}_k\hat{H}_l}\neq b_{H_j}$ for any $q+1\le j\le r$, and $1\le m\le p$. 
\end{lemma}
	\begin{proof}
		By using \eqref{Eq_Pfaff_V_x2} and \eqref{Eq_Pfaff_V_y2}, we compute the completely integrable condition for \eqref{Eq_Pfaff_V}. 
	Then, from the condition \eqref{eq:cond12p} we have
	\[
	\left[B_{\hat{H}_k},\sum_{\substack{ \hat{H}_{l'}\in\mathcal{C}_{\hat{H}_k,y} \\ c_{\hat{H}_k\hat{H}_{l}}=c_{\hat{H}_k\hat{H}_{l'}}}}B_{\hat{H}_{l'}}\right]=O
	\]
	for $\hat{p}+1\le k\le \hat{q}$, $1\le l\le \hat{q}$ such that $c_{\hat{H}_k\hat{H}_l}\neq b_{H_j}$ for any $q+1\le j\le r$,
which gives \eqref{Eq_Case2_lem3}.
Also, from the condition \eqref{eq:cond16p} we have
	\[
	\left[B_{\hat{H}_k},\sum_{\substack{\hat{p}+1\le {l'} \le \hat{q} \\ c_{\hat{H}_k\hat{H}_l}=c_{\hat{H}_k\hat{H}_{l'}}}}B_{\hat{H}_{l'}}\right]=O
	\]
	for $1\le k\le\hat{p}$, $\hat{p}+1\le l\le\hat{q}$ such that $c_{\hat{H}_k\hat{H}_l}\neq b_{H_j}$ for any $q+1\le j\le r$, 
	which gives \eqref{Eq_Case2_lem4}. 
	\end{proof}
With this lemma in mind, we divide the summation on $k$ in $\bm{D}$ as follows:
	\begin{align*}
	\bm{D}&=\left(\sum_{\substack{k=1\\ (k,l)\in[(k,l)]\\ \text{ with }\eqref{Eq_Case2_rep}}}^{\hat{p}}+
	\sum_{\substack{k=\hat{p}+1\\ (k,l)\in[(k,l)] \\ \text{ with }\eqref{Eq_Case2_rep}}}^{\hat{q}}\right)
	\sum_{m=1}^{p}\sum_{\substack{k<l'\le \hat{q} \\ \hat{H}_{l'}\in\mathcal{C}_{\hat{H}_k,y} \\ c_{\hat{H}_{k}\hat{H}_{l'}}=c_{\hat{H}_{k}\hat{H}_l}}}\sum_{j=1}^{p}(E_{N_k}A_{H_j}E_{N_{l'}}-E_{N_{l'}}A_{H_j}E_{N_k})A_{H_m}v_m^k 
	\\
	&=\sum_{\substack{k=1\\ (k,l)\in[(k,l)] \text{ with }\eqref{Eq_Case2_rep}}}^{\hat{p}}\sum_{m=1}^{p}\sum_{\substack{\hat{p}+1\le l'\le \hat{q} \\ \hat{H}_{l'}\in\mathcal{C}_{\hat{H}_k,y} \\ c_{\hat{H}_{k}\hat{H}_{l'}}=c_{\hat{H}_{k}\hat{H}_l}}}\sum_{j=1}^{p}(E_{N_k}A_{H_j}E_{N_{l'}}-E_{N_{l'}}A_{H_j}E_{N_k})A_{H_m}v_m^k 
	\\
	&\quad +
	\sum_{\substack{k=\hat{p}+1\\ (k,l)\in[(k,l)] \text{ with }\eqref{Eq_Case2_rep}}}^{\hat{q}}\sum_{m=1}^{p}\sum_{\substack{k<l'\le \hat{q} \\ \hat{H}_{l'}\in\mathcal{C}_{\hat{H}_k,y} \\ c_{\hat{H}_{k}\hat{H}_{l'}}=c_{\hat{H}_{k}\hat{H}_l}}}\sum_{j=1}^{p}(E_{N_k}A_{H_j}E_{N_{l'}}-E_{N_{l'}}A_{H_j}E_{N_k})A_{H_m}v_m^k.
	\end{align*}
Here we note that any representative $(k,l)\in[(k,l)]$ for $1\le k\le \hat{p}$ satisfies $\hat{p}+1\le l \le \hat{q}$. 
Hence, thanks to \eqref{Eq_Case2_lem4}, we see that the first term vanishes.
By using \eqref{Eq_Case2_lem3} for the second term, we have
	\begin{align*}
	\bm{D}&=\sum_{\substack{k=\hat{p}+1\\ (k,l)\in[(k,l)] \text{ with }\eqref{Eq_Case2_rep}}}^{\hat{q}}\sum_{m=1}^{p}\sum_{\substack{k<l'\le \hat{q} \\ \hat{H}_{l'}\in\mathcal{C}_{\hat{H}_k,y} \\ c_{\hat{H}_{k}\hat{H}_{l'}}=c_{\hat{H}_{k}\hat{H}_l}}}\sum_{j=1}^{p}(E_{N_k}A_{H_j}E_{N_{l'}}-E_{N_{l'}}A_{H_j}E_{N_k})A_{H_m}v_m^k
	\\
	&=-\sum_{\substack{k=\hat{p}+1\\ (k,l)\in[(k,l)] \text{ with }\eqref{Eq_Case2_rep}}}^{\hat{q}}\sum_{m=1}^{p}\sum_{\substack{1\le l'< k \\ \hat{H}_{l'}\in\mathcal{C}_{\hat{H}_k,y} \\ c_{\hat{H}_{k}\hat{H}_{l'}}=c_{\hat{H}_{k}\hat{H}_l}}}\sum_{j=1}^{p}(E_{N_k}A_{H_j}E_{N_{l'}}-E_{N_{l'}}A_{H_j}E_{N_k})A_{H_m}v_m^k.
	\end{align*}
Here we notice that this summation can be expressed as 
\begin{align*}
\bm{D}=-\sum_{\substack{l'=1\\ (l',k)\in[(k,l)] \text{ with }\eqref{Eq_Case2_rep}}}^{\hat{q}}\sum_{m=1}^{p}\sum_{\substack{l'< k\le q \\ \hat{H}_{k}\in\mathcal{C}_{\hat{H}_{l'},y} \\ c_{\hat{H}_{l'}\hat{H}_{k}}=c_{\hat{H}_{k}\hat{H}_l}}}\sum_{j=1}^{p}(E_{N_k}A_{H_j}E_{N_{l'}}-E_{N_{l'}}A_{H_j}E_{N_k})A_{H_m}v_m^k.
\end{align*}
Moreover, since $c_{\hat{H}_k\hat{H}_{l'}}=c_{\hat{H}_k\hat{H}_{l}}=c_{\hat{H}_{l'}\hat{H}_l}$ holds, we have
\begin{align*}
\bm{D}=-\sum_{\substack{l'=1\\ (l',k)\in[(k,l)] \text{ with }\eqref{Eq_Case2_rep}}}^{\hat{q}}\sum_{m=1}^{p}\sum_{\substack{l'< k\le q \\ \hat{H}_{k}\in\mathcal{C}_{\hat{H}_{l'},y} \\ c_{\hat{H}_{l'}\hat{H}_{k}}=c_{\hat{H}_{l'}\hat{H}_l}}}\sum_{j=1}^{p}(E_{N_k}A_{H_j}E_{N_{l'}}-E_{N_{l'}}A_{H_j}E_{N_k})A_{H_m}v_m^k.
\end{align*}
By exchanging the running index $k\leftrightarrow l'$, we have $\bm{C}=\bm{D}$, 
which shows \eqref{Eq_compati_8}.

\section{Generalized middle convolution in several variables}
\label{Sec_GMC}

In this section, we introduce two transformations on $\mathfrak{M}_{\mathcal{T}_{x_\ell},\mathcal{S}_{x_\ell}}$, equivalently, two endofunctors of $\mathscr{P}_{x_\ell}^{\text{\rm Int}}(\mathcal{T}_{x_\ell},\mathcal{S}_{x_\ell})$.

\begin{definition}
For $\lambda=(\lambda_{H})_{H \in \mathcal{T}_{x_\ell}} \in \mathbb{C}^{\# \mathcal{T}_{x_\ell}}$, we define the operation 
	\[
	\begin{array}{cccc}
	add_{\lambda}^{x_\ell}: 
	&\mathfrak{P}_{\mathcal{T}_{x_\ell},\mathcal{S}_{x_\ell}} & \to&  \mathfrak{P}_{\mathcal{T}_{x_\ell},\mathcal{S}_{x_\ell}} 
	\\
	& \rotatebox{90}{$\in$} & &  \rotatebox{90}{$\in$}
	\\
	& \Omega &\mapsto & \ds \Omega+\sum_{H\in\mathcal{T}_{x_\ell}}\lambda_{H}d\log f_{H},
	\end{array}
	\]
which we call the \emph{addition in $x_\ell$-direction} with $\lambda$.
\end{definition}
We remark the addition operator $add_{\lambda}^{x_\ell}$ can be regarded as an (invertible) map
	\[
	add_{\lambda}^{x_\ell}\,:\,\mathfrak{M}_{\mathcal{T}_{x_\ell},\mathcal{S}_{x_\ell}}
\to\mathfrak{M}_{\mathcal{T}_{x_\ell},\mathcal{S}_{x_\ell}}.
	\]
Then, we define the middle convolution in $x_\ell$-direction as follows.

\begin{definition}
For $\beta=(\beta_{H})_{H\in\mathcal{S}_{x_\ell}} \in \mathbb{C}^{\#\mathcal{S}_{x_\ell}}$, we define the map 
	\[
	mc_{\beta}^{x_\ell}:=\mathcal{ML}^{-x_\ell}\circ add_{-\beta}^{x_\ell} \circ \mathcal{ML}^{x_\ell}
	:\mathfrak{M}_{\mathcal{T}_{x_\ell},\mathcal{S}_{x_\ell}}~\to~\mathfrak{M}_{\mathcal{T}_{x_\ell},\mathcal{S}_{x_\ell}},
	\]
which we call the \emph{(generalized) middle convolution in $x_\ell$-direction} $mc_{\beta}^{x_\ell}$ with $\beta$.
\end{definition}

The map $add_{\lambda}^{x_\ell}$ (resp. $mc_{\beta}^{x_\ell}$) induces the endofunctor on $\mathscr{P}_{x_\ell}^{\text{\rm Int}}(\mathcal{T}_{x_\ell},\mathcal{S}_{x_\ell})$.
We call this functor the \emph{addition functor in $x_\ell$-direction} (resp. \emph{middle convolution functor in $x_\ell$-direction}).
By abuse of notation, we write $add_{\lambda}^{x_\ell}$ (resp. $mc_{\beta}^{x_\ell}$) for the addition functor (resp. middle convolution functor).

	As a corollary of Theorem \ref{Thm_conn}, we obtain some fundamental properties of the middle convolution functor.

\begin{corollary}
	Suppose that $(\calV,\Omega)\in\mathscr{P}_{x_\ell}^{\text{\rm Int}}(\mathcal{T}_{x_\ell},\mathcal{S}_{x_\ell})$ is irreducible and non-exceptional in $x_{\ell}$-direction. 
	Then 
	$mc_{0}^{x_{\ell}}(\calV,\Omega)\sim(\calV,\Omega)$.
	Furthermore, for any $\beta\in\bbC^{\#\calS_{x_{\ell}}}$ such that $add_{-\beta}^{x_{\ell}}\circ\mathcal{ML}^{x_{\ell}}(\calV,\Omega)$ is non-exceptional in $x_{\ell}$-direction, $mc_{\beta}^{x_{\ell}}(\calV,\Omega)$ is also irreducible and non-exceptional in $x_{\ell}$-direction.
	In this case, 
	\begin{equation}\label{eq:mccorpde}
	mc_{\gamma}^{x_{\ell}}\circ mc_{\beta}^{x_{\ell}}(\calV,\Omega)\sim mc_{\beta+\gamma}^{x_{\ell}}(\calV,\Omega)
	\end{equation}
	holds for any $\gamma\in\bbC^{\#\calS_{x_{\ell}}}$.
Analogous statements hold for the irreducible and non-exceptional system \eqref{Eq_Pfaff}, under the corresponding non-exceptionality condition on $add_{-\beta}^{x_{\ell}}\circ \mathcal{ML}^{x_{\ell}}$.\end{corollary}
This corollary is shown in the same way as Theorem \ref{thm:odemc} using Theorem \ref{Thm_conn} instead of Theorem \ref{thm:odeconn}.

\begin{remark}

As mentioned in Introduction, Haraoka \cite{Haraoka2012} introduced the middle convolution in several variables for Pfaffian systems with logarithmic singularity
	\begin{equation}\label{Eq_Pfaff_reg}
	du=\Omega u, \quad 
	\Omega=\sum_{H \in \mathcal{A}}A_H \, d\log f_H.
	\end{equation}
Our middle convolution gives a generalization of Haraoka's middle convolution. 
In fact, 
	\[
	\begin{aligned}
	\mathfrak{P}^{\mathrm{reg}}_{\mathcal{T}_{x_\ell}}:=\left\{
	\Omega=\sum_{H \in \mathcal{A}}A_H \, d\log f_H
	~\middle\vert~
	\begin{array}{l}
	A_H \in \mathrm{Mat}(N,\mathbb{C}),\,N\ge 1, \\
	\text{$\Omega$ satisfies $\Omega \wedge\Omega =0$, } \\
	\mathcal{A}_{x_\ell}=\mathcal{T}_{x_\ell}
	\end{array}
	\right\},
	\qquad 
	\mathfrak{M}^{\mathrm{reg}}_{\mathcal{T}_{x_\ell}}:=\mathfrak{P}^{\mathrm{reg}}_{\mathcal{T}_{x_\ell}}/\sim
	\end{aligned}
	\]
are the subsets of $\mathfrak{P}_{\mathcal{T}_{x_\ell},\{x_\ell=0\}}$ and $\mathfrak{M}_{\mathcal{T}_{x_\ell},\{x_\ell=0\}}$, respectively.
Then, for any $\beta\in\mathbb{C}$, our middle convolution $mc_{\beta}^{x_\ell}: \mathfrak{M}^{\mathrm{reg}}_{\mathcal{T}_{x_\ell}}\to \mathfrak{M}^{\mathrm{reg}}_{\mathcal{T}_{x_\ell}}$ coincides with the Haraoka's middle convolution in $x_\ell$-direction. 
\end{remark}

\section{Examples}
\label{sec:extwo}
We give some examples in the case of two variables. 
Let $(x_1,x_2)=(x,y)$ be a coordinate of $\mathbb{C}^2$ and $\mathcal{A}=\{H_1,H_2,H_3\}$ be the hyperplane arrangement in $\bbC^2$ consisting of 
\[
H_1=\{x=0\}, \quad
H_2=\{x=1\}, \quad
H_3=\{x=y\}.
\]
That is, 
\[
f_{H_1}=x,\quad
f_{H_2}=x-1,\quad
f_{H_3}=x-y.
\]
Then we have
\begin{equation}\label{eq:Pfaffex0}
\begin{aligned}
&\calA_{x}=\{H_1,H_2,H_3\}, \quad \calA_{y}=\{H_3\}, 
\quad 
\calC_{H_1,y}=\{H_3\}, 
\quad 
\calC_{H_2,y}=\{H_3\},
\quad 
\calC_{H_3,y}=\{H_1,H_2\}, 
\\
& c_{H_1H_3}=0, \quad c_{H_2H_3}=1. 
\end{aligned}
\end{equation}
We consider the following rank one Pfaffian system 
\begin{equation}\label{eq:Pfaffex}
du=\Omega u, \quad \Omega=\sum_{H\in\mathcal{A}}\alpha_H\,d\log f_H=\alpha_{H_1}\frac{dx}{x}+\alpha_{H_2}\frac{dx}{x-1}+\alpha_{H_3}\frac{d(x-y)}{x-y},
\end{equation}
where $\alpha_{H_1},\alpha_{H_2},\alpha_{H_3}\in\bbC\setminus\{0\}$ are constants. 
This equation can be expressed as
\[
\begin{cases}
 \ds\frac{\p u}{\p x}=\left(\frac{\alpha_{H_1}}{x}+\frac{\alpha_{H_2}}{x-1}+\frac{\alpha_{H_3}}{x-y}\right)u \\[11pt]
 \ds\frac{\p u}{\p y}=\frac{\alpha_{H_3}}{y-x}\,u.
 \end{cases}
\]
We consider applying the middle Laplace transform in $x$-direction for \eqref{eq:Pfaffex}.
For this equation, we have $\hat{H}_1=\{x=0\}$ and 
\[
\widetilde{\calA}=\calA\cup\{H_{13},H_{23}\},\quad H_{13}:=\{y=0\},~\quad H_{23}:=\{y-1=0\}.
\]
and hence obtain $\mathcal{B}=\{\hat{H}_1,\,H_{13},\,H_{23}\}$ (cf. \eqref{def:tilA} and \eqref{eq:Bdef}).
Then, by operating the middle Laplace transform, we have the rank three system
\begin{equation}\label{eq:Pfaffex1}
\begin{aligned}
dV&=\mathcal{ML}^{x}(\Omega)V\\
&=\left((B_x+yB_{xy})\,dx+xB_{xy}\,dy+B_{H_1}\frac{dx}{x}+B_{H_{13}}\frac{dy}{y}+B_{H_{23}}\frac{dy}{y-1}\right)V
\end{aligned}
\end{equation}
where 
\begin{align*}
&B_x=-\begin{pmatrix}
0 & & \\
& 1 & \\
& & 0
\end{pmatrix},
\quad 
B_{xy}=-\begin{pmatrix}
0 & & \\
& 0 & \\
& & 1
\end{pmatrix},
\quad 
B_{H_1}=-\begin{pmatrix}
\alpha_{H_1} & \alpha_{H_2} & \alpha_{H_3} \\
\alpha_{H_1} & \alpha_{H_2} & \alpha_{H_3} \\
\alpha_{H_1} & \alpha_{H_2} & \alpha_{H_3} 
\end{pmatrix}, \\
&B_{H_{13}}=\begin{pmatrix}
\alpha_{H_3} & 0 & -\alpha_{H_3} \\
0 & 0 & 0 \\
-\alpha_{H_1} & 0 &\alpha_{H_1}
\end{pmatrix},
\quad 
B_{H_{23}}=\begin{pmatrix}
0 & 0 & 0 \\
0 & \alpha_{H_3} &  -\alpha_{H_3} \\
0 & -\alpha_{H_2} &\alpha_{H_2}
\end{pmatrix}.
\end{align*}
In this case, we do not need to consider the projection to the quotient space (\textbf{Step 3}) since $\calK^x=\bigoplus_{i=1}^3\Ker\alpha_{H_i}=\{0\}$.
In other words, $\mathcal{L}^{x}(\Omega)=\mathcal{ML}^x(\Omega)$ in this case.
The equation \eqref{eq:Pfaffex1} is essentially the same as the Pfaffian system satisfied by the Humbert's confluent hypergeometric function $\Phi_1$ of two variables, which can be verified directly (for example, see \cite{Mukai}).
\begin{remark}
The function $u(x,y)=x^{\alpha _{H_1}} (x-1)^{\alpha _{H_2}} (x-y)^{\alpha _{H_3}}$ satisfies \eqref{eq:Pfaffex}. 
Therefore, by construction, we see that the equation \eqref{eq:Pfaffex1} has an integral representation of solutions 
\[
V(x,y)=\int_{\Delta}t^{\alpha_{H_1}}(t-1)^{\alpha_{H_2}}(t-y)^{\alpha_{H_3}}e^{-tx}\,\vec{\varphi}, 
\quad 
\vec{\varphi}={}^t \left(\frac{dt}{t},\frac{dt}{t-1},\frac{dt}{t-y}\right),
\]
which can be used to study the global behavior of solutions.
\end{remark}

\subsection{Appell's $F_1$}
Let $\beta\in\bbC$ be a complex number. 
By considering the addition in $x$-direction $add_{-\beta}^x$  for \eqref{eq:Pfaffex1} with $-\beta$, we have 
\begin{equation}\label{eq:Pfaffex2}
\begin{aligned}
&dW=(add_{-\beta}^{x}\circ\mathcal{ML}^{x})(\Omega)W\\
&=\left((B_x+yB_{xy})\,dx+xB_{xy}\,dy+(B_{H_1}-\beta)\frac{dx}{x}+B_{H_{13}}\frac{dy}{y}+B_{H_{23}}\frac{dy}{y-1}\right)W.
\end{aligned}
\end{equation}
We assume $\beta\neq 0,-(\alpha_{H_1}+\alpha_{H_2}+\alpha_{H_3})$.
Then, it holds that $\Ker(B_{H_1}-\beta)=\{0\}$. 
By applying the inverse middle Laplace transform $\mathcal{ML}^{-x}$, we obtain the rank three system
\begin{equation}\label{eq:PfaffF1}
\begin{aligned}
dZ&=(\mathcal{ML}^{-x}\circ add_{-\beta}^x\circ \mathcal{ML}^x)(\Omega)Z=mc_{\beta}^{x}(\Omega)Z \\
&=\left(C_{H_1}\frac{dx}{x}+C_{H_2}\frac{dx}{x-1}+C_{H_3}\frac{d(x-y)}{x-y}+C_{H_{13}}\frac{dy}{y}+C_{H_{23}}\frac{dy}{y-1}\right)Z
\end{aligned}
\end{equation}
where
\begin{align*}
&C_{H_1}=\begin{pmatrix}
\alpha_{H_1}+\beta & \alpha_{H_2} & \alpha_{H_3} \\
0 & 0 &0 \\
0 & 0 &0 
\end{pmatrix}, \quad 
C_{H_2}=\begin{pmatrix}
0 & 0 &0 \\
\alpha_{H_1} & \alpha_{H_2}+\beta & \alpha_{H_3} \\
0 & 0 &0 
\end{pmatrix}, \\ 
&C_{H_3}=\begin{pmatrix}
0 & 0 &0 \\
0 & 0 &0 \\
\alpha_{H_1} & \alpha_{H_2} & \alpha_{H_3} +\beta
\end{pmatrix}, 
\quad 
C_{H_{13}}=B_{H_{13}}, \quad C_{H_{23}}=B_{H_{23}}.
\end{align*}
We remark that, in this case, it is also unnecessary to consider the projection to the quotient space, since $\Ker(B_{H_1}-\beta)=\{0\}$.
The resulting system \eqref{eq:PfaffF1} is essentially the same as the Pfaffian system satisfied by the hypergeometric function Appell's $F_1$ of two variables (for example, see \cite{HaraokaBook}).

\subsection{Haraoka's confluent $F_4$}
We give another example. 
By applying the addition $add_{(-\alpha_{H_1}-\alpha_{H_3},-\alpha_{H_2}-\alpha_{H_3})}^y$ in $y$-direction for \eqref{eq:Pfaffex1} with $(-\alpha_{H_1}-\alpha_{H_3},-\alpha_{H_2}-\alpha_{H_3})\in\bbC^2$, we have
\begin{equation}\label{eq:Pfaffex3}
\begin{aligned}
dW&=(add_{(-\alpha_{H_1}-\alpha_{H_3},-\alpha_{H_2}-\alpha_{H_3})}^y\circ\mathcal{ML}^{x})(\Omega)W\\
 &=\left((B_x+yB_{xy})\,dx+xB_{xy}\,dy
+B_{H_1}\frac{dx}{x}+(B_{H_{13}}-\alpha_{H_1}-\alpha_{H_3})\frac{dy}{y}\right. \\
&\qquad +\left.(B_{H_{23}}-\alpha_{H_2}-\alpha_{H_3})\frac{dy}{y-1}\right)W.
\end{aligned}
\end{equation}
We then consider applying the inverse middle Laplace transform $\mathcal{ML}^{-y}$ in $y$-direction. 
Since $\dim\Ker(B_{H_{13}}-\alpha_{H_1}-\alpha_{H_3})=\dim\Ker(B_{H_{23}}-\alpha_{H_2}-\alpha_{H_3})=1$, the rank of the resulting equation is given by $3\times 2-2=4$. 
To see that, we first apply the inverse Laplace transform $\mathcal{L}^{-y}$, we have 
\begin{equation}
\begin{aligned}
dZ&=(\mathcal{L}^{-y}\circ add_{(-\alpha_{H_1}-\alpha_{H_3},-\alpha_{H_2}-\alpha_{H_3})}^y\circ\mathcal{ML}^{x})(\Omega)Z \\
&=\left(C_x\,dx+C_y\,dy+C_{H_1}\frac{dx}{x}+C_{H_{3}}\frac{d(x-y)}{x-y}+C_{H_{13}}\frac{dy}{y}\right)Z
\end{aligned}
\end{equation}
where 
\begin{align}
&C_{x}=\begin{pmatrix}
B_x & \\
& B_x+B_{xy}
\end{pmatrix}
=\begin{pmatrix}
0 & & & & & \\
& -1 & & & & \\
& & 0 & & & \\
& & & 0 & &\\
& & & & -1 & \\
& & & & & -1
\end{pmatrix}, 
\\
&C_{y}=\begin{pmatrix}
O_3 &\\
& I_{3}
\end{pmatrix}
=\begin{pmatrix}
0 & & & & & \\
& 0 & & & & \\
& & 0 & & & \\
& & & 1 & &\\
& & & & 1 & \\
& & & & & 1
\end{pmatrix}, \\
&C_{H_1}=
\begin{pmatrix}
B_{H_1} & \\
& B_{H_1}
\end{pmatrix}
=\begin{pmatrix}
-\alpha_{H_1} & -\alpha_{H_2} & -\alpha_{H_3} & 0 & 0 & 0\\
-\alpha_{H_1} & -\alpha_{H_2} & -\alpha_{H_3} & 0 & 0 & 0\\
-\alpha_{H_1} & -\alpha_{H_2} & -\alpha_{H_3} & 0 & 0 & 0\\
 0 & 0 & 0&-\alpha_{H_1} & -\alpha_{H_2} & -\alpha_{H_3} \\
 0 & 0 & 0&-\alpha_{H_1} & -\alpha_{H_2} & -\alpha_{H_3}\\
 0 & 0 & 0&-\alpha_{H_1} & -\alpha_{H_2} & -\alpha_{H_3}
\end{pmatrix},
 \\
&C_{H_3}=
-\begin{pmatrix}
E_{N_2}(B_{H_{13}}-\alpha_{H_1}-\alpha_{H_3}) & E_{N_2}(B_{H_{23}}-\alpha_{H_2}-\alpha_{H_3})\\
E_{N_2}(B_{H_{13}}-\alpha_{H_1}-\alpha_{H_3}) & E_{N_2}(B_{H_{23}}-\alpha_{H_2}-\alpha_{H_3})
\end{pmatrix} \\
&\qquad =\begin{pmatrix}
0 & 0 & 0 & 0 & 0 & 0\\
0 & 0 & 0 & 0 & 0 & 0\\
\alpha_{H_1} & 0 & \alpha_{H_3} & 0 & \alpha_{H_2} & \alpha_{H_3}\\
0 & 0 & 0 & 0 & 0 & 0\\
0 & 0 & 0 & 0 & 0 & 0\\
\alpha_{H_1} & 0 & \alpha_{H_3} & 0 & \alpha_{H_2} & \alpha_{H_3}
\end{pmatrix},
 \\
&C_{H_{13}}=
-\begin{pmatrix}
E_{N_1}(B_{H_{13}}-\alpha_{H_1}-\alpha_{H_3}) & E_{N_1}(B_{H_{23}}-\alpha_{H_2}-\alpha_{H_3})\\
E_{N_1}(B_{H_{13}}-\alpha_{H_1}-\alpha_{H_3}) & E_{N_1}(B_{H_{23}}-\alpha_{H_2}-\alpha_{H_3})
\end{pmatrix} \\
&\qquad =\begin{pmatrix}
\alpha_{H_1} & 0 & \alpha_{H_3} & \alpha_{H_2}+\alpha_{H_3} & 0 & 0\\
0 & \alpha_{H_1}+\alpha_{H_3} & 0 & 0 & \alpha_{H_2} & \alpha_{H_3}\\
0 & 0 & 0 & 0 & 0 & 0\\
\alpha_{H_1} & 0 & \alpha_{H_3} & \alpha_{H_2}+\alpha_{H_3} & 0 & 0\\
0 & \alpha_{H_1}+\alpha_{H_3} & 0 & 0 & \alpha_{H_2} & \alpha_{H_3}\\
0 & 0 & 0 & 0 & 0 & 0
\end{pmatrix}.
\end{align}
Here we used 
\[
E_{N_1}=\begin{pmatrix} 1 & & \\
& 1 & \\
& & 0
\end{pmatrix},\quad
E_{N_2}=\begin{pmatrix} 0 & & \\
& 0 & \\
& & 1
\end{pmatrix}.
\]
We find that 
\[
\Ker(B_{H_{13}}-\alpha_{H_1}-\alpha_{H_3})=\left\langle
\begin{pmatrix}
-\alpha_{H_3} \\
0 \\
\alpha_{H_1}
\end{pmatrix}
\right\rangle,
\quad 
\Ker(B_{H_{23}}-\alpha_{H_2}-\alpha_{H_3})=\left\langle
\begin{pmatrix}
0 \\
-\alpha_{H_3} \\
\alpha_{H_2}
\end{pmatrix}
\right\rangle
\]
and hence
\[
\mathcal{K}^y=\Ker(B_{H_{13}}-\alpha_{H_1}-\alpha_{H_3})\oplus\Ker(B_{H_{23}}-\alpha_{H_2}-\alpha_{H_3})
=\left\langle
 \begin{pmatrix}
-\alpha_{H_3} \\
0 \\
\alpha_{H_1} \\
0 \\ 
0 \\
0
\end{pmatrix},
\,
\begin{pmatrix}
0 \\
0 \\
0 \\
0 \\
-\alpha_{H_3} \\
\alpha_{H_2}
\end{pmatrix}
\right\rangle.
\]
By using them, we set 
\[
P=\begin{pmatrix}
-\alpha_{H_3}  & 0 & 0 & 0 & 0 & 0 \\
 0 & 0 & 0 & 0 & \alpha_{H_3}  & 0 \\
 \alpha_{H_1}  & 0 & 0 & 0 & 0 & 1 \\
 0 & 0 & 0 & \alpha_{H_3}  & 0 & 0 \\
 0 & -\alpha_{H_3}  & 0 & 0 & 0 & 0 \\
 0 & \alpha_{H_2}  & 1 & 0 & 0 & 0
\end{pmatrix}.
\]
Then, we have 
\begin{align*}
&P^{-1}C_xP=\left(\begin{array}{c|c} * & * \\ \hline \\[-10pt] O\, & \,\bar{C}_x\end{array}\right), 
\quad 
P^{-1}C_yP=\left(\begin{array}{c|c} * & * \\ \hline \\[-10pt] O\, & \,\bar{C}_y\end{array}\right), 
\quad 
P^{-1}C_{H_1}P=\left(\begin{array}{c|c} * & * \\ \hline \\[-10pt] O\, & \,\bar{C}_{H_1}\end{array}\right),\\
&P^{-1}C_{H_3}P=\left(\begin{array}{c|c} * & * \\ \hline \\[-10pt] O\, & \,\bar{C}_{H_3}\end{array}\right), 
\quad
P^{-1}C_{H_{13}}P=\left(\begin{array}{c|c} * & * \\ \hline \\[-10pt] O\, & \,\bar{C}_{H_{13}}\end{array}\right) 
\end{align*}
where
\begin{align*}
&\bar{C}_x=\begin{pmatrix}
-1 & & & \\
& 0 & & \\
& & -1 & \\
& & & 0
\end{pmatrix},
\quad 
\bar{C}_y=\begin{pmatrix}
1 & & & \\
& 1 & & \\
& & 0 & \\
& & & 0
\end{pmatrix}, \\
&
\bar{C}_{H_1}=\begin{pmatrix}
 -\alpha_{H_2}-\alpha_{H_3} & -\alpha_{H_1}(\alpha_{H_2}+\alpha_{H_3})& 0 & 0 \\
 -1& -\alpha_{H_1} & 0 & 0 \\
 0 &0 & -\alpha_{H_2} & -1 \\
 0 & 0 & -\alpha_{H_2}(\alpha_{H_1}+\alpha_{H_3}) & -\alpha_{H_1}-\alpha_{H_3}
\end{pmatrix},
\\
&
\bar{C}_{H_3}=\begin{pmatrix}
\alpha_{H_3} & 0 & 0 &\alpha_{H_3} \\
0 & 0 & 0 & 0 \\
0 & 0 & 0 & 0 \\
\alpha_{H_3} & 0 & 0 &\alpha_{H_3}
\end{pmatrix},\\
&
\bar{C}_{H_{13}}=\begin{pmatrix}
\alpha_{H_2}  & 0 & \alpha_{H_2}  (\alpha_{H_1} +\alpha_{H_3} ) & 0 \\
 0 & \alpha_{H_2} +\alpha_{H_3}  & 0 & 1 \\
 1 & 0 & \alpha_{H_1} +\alpha_{H_3}  & 0 \\
 0 & \alpha_{H_1}  (\alpha_{H_2} +\alpha_{H_3} ) & 0 & \alpha_{H_1}
\end{pmatrix}.
\end{align*}
As a result, we have the equation of the inverse middle Laplace transform $\mathcal{ML}^{-y}$ in $y$-direction for \eqref{eq:Pfaffex3}: 
\begin{equation}\label{eq:PfaffcF4}
\begin{aligned}
dz&=(\mathcal{ML}^{-y}\circ add_{(-\alpha_{H_1}-\alpha_{H_3},-\alpha_{H_2}-\alpha_{H_3})}^y\circ\mathcal{ML}^{x})(\Omega)z \\
&=\left(\bar{C}_x\,dx+\bar{C}_y\,dy+\bar{C}_{H_1}\frac{dx}{x}+\bar{C}_{H_{3}}\frac{d(x-y)}{x-y}+\bar{C}_{H_{13}}\frac{dy}{y}\right)z.
\end{aligned}
\end{equation}
We can verify that this equation is essentially the same as the Pfaffian system derived by Haraoka \cite{Haraoka2024}, who obtained it by applying middle convolution and confluence to equation \eqref{eq:Pfaffex}.
As mentioned in the introduction, in \cite{Haraoka2024}, Haraoka also studied the Stokes phenomenon around $\{x=\infty\}\times\{y=\infty\}$. 
It would be an interesting problem to understand his result from the perspective of the middle Laplace transform.

\section*{Acknowledgements}
The author is deeply grateful to Kazuki Hiroe for providing invaluable comments and suggestions during the preparation of this paper.
In particular, the content of sections \ref{sec:odefunc}, \ref{sec:odeinv}, \ref{Sec_Categorical} and \ref{Sec_Inv} owes much to discussions with Hiroe. 
Without his continued support and encouragement, this paper would not have been completed. 
The author thanks the anonymous referee for their careful reading and many valuable comments. 
The author is also grateful to Saiei-Jaeyeong Matsubara-Heo and Takafumi Matsumoto for their helpful advice on meromorphic connections. 
Finally, the author deeply respects Yoshishige Haraoka and Toshio Oshima for their inspiring work on the middle convolution for holonomic systems, and expresses his sincere appreciation to them for their many valuable comments and encouragement. 
This work was supported by JSPS KAKENHI Grant Number 24K22826.


\end{document}